\documentclass[a4paper,11pt]{amsart}
\usepackage[left=25mm,top=25mm,bottom=25mm,right=25mm]{geometry}
\usepackage{amssymb,dsfont,amscd}
\usepackage{cmap}                       
\usepackage{hyperxmp}                   
\usepackage{array}
\usepackage[utf8]{inputenc}             
\usepackage[english]{babel}
\usepackage{graphicx}
\usepackage{booktabs}
\usepackage[pdfdisplaydoctitle = true,
      colorlinks = true,
      urlcolor = blue,
      citecolor = blue,
      linkcolor = blue,
      pdfstartview = FitH,
      pdfpagemode = UseNone,
      bookmarksnumbered = true,
      unicode=true]{hyperref}           
\graphicspath{{figs/}}
\newtheorem{thm}{Theorem}[section]
\newtheorem{cor}[thm]{Corollary}
\newtheorem{lem}[thm]{Lemma}
\newtheorem{prop}[thm]{Proposition}
\theoremstyle{definition}
\newtheorem{defn}[thm]{Definition}
\theoremstyle{remark}
\newtheorem{rem}[thm]{Remark}
\newtheorem{ex}[thm]{Example}

\numberwithin{equation}{section}
\newcommand{\CC}{\mathbb{C}}                
\newcommand{\RR}{\mathbb{R}}                
\newcommand{\ZZ}{\mathbb{Z}}                

\newcommand{\Ela}{\mathbb{E}\mathrm{la}}    
\newcommand{\HH}{\mathbb{H}}                
\newcommand{\TT}{\mathbb{T}}                
\newcommand{\Sym}{\mathbb{S}}               
\newcommand{\VV}{\mathbb{V}}                
\newcommand{\WW}{\mathbb{W}}                

\newcommand{\Sn}[1]{\mathrm{S}_{#1}}        
\newcommand{\RSn}[1]{\Sn{#1}^{\RR}}         
\newcommand{\Hn}[1]{\mathcal{H}_{#1}}       
\newcommand{\Pn}[1]{\mathcal{P}_{#1}}       

\newcommand{\cov}{\mathbf{Cov}}
\newcommand{\inv}{\mathbf{Inv}}
\newcommand{\pol}{\mathrm{Pol}}

\newcommand{\SL}{\mathrm{SL}}               
\newcommand{\SU}{\mathrm{SU}}               
\newcommand{\OO}{\mathrm{O}}                
\newcommand{\SO}{\mathrm{SO}}               
\newcommand{\octa}{\mathbb{O}}              
\newcommand{\ico}{\mathbb{I}}               
\newcommand{\tetra}{\mathbb{T}}             
\newcommand{\DD}{\mathbb{D}}                
\newcommand{\triv}{\mathds{1}}		        
\newcommand{\id}{\mathrm{I}}                

\newcommand{\so}{\mathfrak{so}}             
\newcommand{\slc}{\mathfrak{sl}}            

\newcommand{\bxi}{\pmb{\xi}}                
\newcommand{\ee}{\pmb{e}}                   
\newcommand{\nn}{\pmb{n}}                   
\newcommand{\uu}{\pmb{u}}                   
\newcommand{\vv}{\pmb{v}}                   
\newcommand{\ww}{\pmb{w}}                   
\newcommand{\xx}{\pmb{x}}                   

\newcommand{\ff}{\mathbf{f}}                
\newcommand{\bg}{\mathbf{g}}                
\newcommand{\bw}{\mathbf{w}}                
\newcommand{\hh}{\mathbf{h}}                

\newcommand{\qq}{\mathrm{q}}                
\newcommand{\rp}{\mathrm{p}}                
\newcommand{\ra}{\mathrm{a}}                
\newcommand{\rb}{\mathrm{b}}                
\newcommand{\rh}{\mathrm{h}}                
\newcommand{\rt}{\mathrm{t}}

\newcommand{\ba}{\mathbf{a}}
\newcommand{\bb}{\mathbf{b}}
\newcommand{\bc}{\mathbf{c}}
\newcommand{\bd}{\mathbf{d}}
\newcommand{\be}{\mathbf{e}}
\newcommand{\bq}{\mathbf{q}}                
\newcommand{\bh}{\mathbf{h}}

\newcommand{\bt}{\mathbf{t}}
\newcommand{\bv}{\mathbf{v}}

\newcommand{\bE}{\mathbf{E}}                
\newcommand{\bH}{\mathbf{H}}                
\newcommand{\bS}{\mathbf{S}}                
\newcommand{\bT}{\mathbf{T}}                

\newcommand{\trans}[3]{\lbrace #1,#2\rbrace_{#3}}

\DeclareMathOperator{\Ad}{Ad}
\DeclareMathOperator{\tr}{tr}

\DeclareMathOperator{\2dots}{:}
\DeclareMathOperator{\3dots}{\raisebox{-0.25ex}{\vdots}}

\newcommand{\norm}[1]{\lVert#1\rVert}       
\newcommand{\set}[1]{\left\{#1\right\}}     

\newcommand{\tq}[1]{{{\mathbf{#1}}}}
\newcommand{\td}[1]{{\mathbf{#1}}}

\newcommand{\dg}[2]{{\vphantom{#2}}^{\textrm{#1}}#2}
\newcommand{\cv}[2]{{\dg{#1,#2}{\pmb{C}}}}

\newcommand{\ab}{\left(\td{a}\bb\right)}
\newcommand{\abS}{\left(\td{a}\bb\right)^s}

\hypersetup{
  pdfauthor = {M. Olive, B. Kolev, R. Desmorat, B. Desmorat}, 
  pdftitle = {Characterization of the symmetry class of an elasticity tensor using polynomial covariants}, 
  pdfsubject = {MSC 2010: 74E10 (15A72 74B05)}, 
  pdfkeywords = {Covariants of tensors; Characterization of symmetry classes; Covariant basis for fourth order harmonic tensors; elasticity tensor}, 
  pdflang = en, 
  }
\begin{document}

\title[Covariants and Symmetry Classes]{Characterization of the symmetry class of an elasticity tensor using polynomial covariants}%

\author{M. Olive}
\address[Marc Olive]{Université Paris-Saclay, ENS Paris-Saclay, CNRS,  LMT - Laboratoire de Mécanique et Technologie, 91190, Gif-sur-Yvette, France}
\email{marc.olive@math.cnrs.fr}

\author{B. Kolev}
\address[Boris Kolev]{Université Paris-Saclay, ENS Paris-Saclay, CNRS,  LMT - Laboratoire de Mécanique et Technologie, 91190, Gif-sur-Yvette, France}
\email{boris.kolev@ens-paris-saclay.fr}

\author{R. Desmorat}
\address[Rodrigue Desmorat]{Université Paris-Saclay, ENS Paris-Saclay, CNRS,  LMT - Laboratoire de Mécanique et Technologie, 91190, Gif-sur-Yvette, France}
\email{rodrigue.desmorat@ens-paris-saclay.fr}

\author{B. Desmorat}
\address[Boris Desmorat]{Sorbonne Université, CNRS, UMR 7190, Institut d'Alembert, F-75252 Paris Cedex 05, France}
\email[Boris Desmorat]{boris.desmorat@sorbonne-universite.fr}


\date{May 7, 2021}%


\begin{abstract}
  We formulate effective necessary and sufficient conditions to identify the symmetry class of an elasticity tensor, a fourth-order tensor which is the cornerstone of the theory of elasticity and a toy model for linear constitutive laws in physics. The novelty is that these conditions are written using \emph{polynomial covariants}. As a corollary, we deduce that the symmetry classes are \emph{affine algebraic sets}, a result which seems to be new. Meanwhile, we have been lead to produce a minimal set of 70 generators for the covariant algebra of a fourth-order harmonic tensor and introduce an original generalized cross product on totally symmetric tensors. Finally, using these tensorial covariants, we produce a new minimal set of 294 generators for the invariant algebra of the elasticity tensor.
\end{abstract}

\maketitle


\begin{scriptsize}
  \setcounter{tocdepth}{2}
  \tableofcontents
\end{scriptsize}

\section{Introduction}
\label{sec:intro}

Having access to an explicit description of the structure of the orbit space of a real representation of a compact Lie group is important in many fields of mathematics, physics, biology and mechanics. It appears, for instance, in bifurcation theory~\cite{IG1984,GSS1988,CLM1991} when symmetry is involved, in physics when one needs to classify allowed patterns of spontaneous symmetry breaking~\cite{Mic1980,AS1981,AS1983}, in neuroimaging when measuring by Diffusion Tensor Imaging, the anisotropic biomarkers relevant to various diseases~\cite{DVW+2009,LMM+2009,PGD2014,CV2015,GHP2018} or in mechanics of materials when one needs to classify constitutive tensors such as the elasticity tensor~\cite{Lov1905,BKO1994,FV1996,FBG1998}.

Given a real representation $(\VV,G)$ of a compact group $G$, let $\Sigma_{[H]}$ be the subset of vectors $\vv$ in $\VV$ whose symmetry group is conjugate to the subgroup $H$. If $\Sigma_{[H]} \ne \emptyset$, then, $\Sigma_{[H]}$ is called an \emph{isotropy stratum} and the conjugacy class $[H]$ of $H$, an \emph{isotropy class}, a \emph{symmetry class} or an \emph{orbit type}. It is known that there are only a finite number of \emph{isotropy classes}~\cite{Mos1957,Man1962} and thus
\begin{equation*}
  V=\Sigma_{[H_{1}]}\sqcup \dotsb \sqcup \Sigma_{[H_{s}]}.
\end{equation*}
The closure $\overline{\Sigma}_{[H]}$ of a stratum (for the usual topology on $\VV$) is itself a union of strata~\cite{DK2000}. In common language it can be described as the set of vectors $\vv \in \VV$ whose symmetry group contains a conjugate of $H$.

An important result of the present work is that, for the representation of $\SO(3)$ on the space of \emph{elasticity tensors},  the \emph{closed strata} $\overline{\Sigma}_{[H]}$ are \emph{real affine algebraic sets}~\cite[Chapter 1]{CLO2007}. We conjecture that this might be true in general for any real representation of a compact algebraic group but we do not have a proof of this fact and we are not aware of such a result in the literature.

It is known, however, that the orbit space $\VV/G$ is a semialgebraic set~\cite{AS1981,AS1983,PS1985,Sch1989}, \textit{i.e}
defined by a finite sequence of polynomial equations and inequalities~\cite{Cos2002}. An explicit construction of a semialgebraic structure requires however the knowledge of an \emph{integrity basis} (a generating set of the polynomial invariant algebra of $\VV$) which involves sophisticated techniques in general~\cite{Wey1997,Stu2008}. For instance, this was achieved only recently for the elasticity tensor~\cite{OKA2017}, a fourth-order tensor on $\RR^{3}$. On the other hand, the description we provide of the algebraic structure on closed strata is based on \emph{polynomial covariants} rather than polynomial invariants. Surprisingly, this was found to be much more effective.

The first appearance of the concept of \emph{polynomial covariants} goes back to the Nineteenth Century and was introduced in the \emph{theory of binary forms}~\cite{GY2010,Olv1999}. It was reformulated more recently by Kraft and Procesi~\cite{KP2000} in the general case (see also~\cite[Chapter 5]{Dol2003}). Given two linear representations $\VV$ and $\WW$ of a group $G$, they define a \emph{polynomial covariant} of $\VV$ of type $\WW$ as a polynomial equivariant mapping between $\VV$ and $\WW$. This definition turns the set of polynomial covariants into a vector space. In this paper, we have extended this definition in such a way that the set of covariants is a \emph{commutative algebra}, namely, we have defined the \emph{covariant algebra} $\cov(\VV,\WW)$ simply as the invariant algebra
\begin{equation*}
  \RR[\VV\oplus \WW^{*}]^{G},
\end{equation*}
where $\WW^{*}$ is the dual space of $\WW$ endowed with the dual linear representation. This algebra contains the set of polynomial covariants as defined by Kraft and Procesi which corresponds to the subspace of polynomial covariants of degree 1 in $\WW^{*}$. The advantage of this extended definition is that, thanks to Hilbert's theorem~\cite{Hil1993,Stu2008}, this algebra is finitely generated. When $\VV$ is a tensorial representation of the orthogonal group $\OO(3)$ and $\WW = \RR^{3}$, this algebra, that we have simply denoted by $\cov(\VV)$, happens to be much more useful to describe the isotropic strata than its subalgebra $\inv(\VV)$, the invariant algebra of $\VV$.

In the specific case of the elasticity tensor, a fourth-order tensor $(E_{ijkl})$ on $\RR^{3}$ with index symmetry
\begin{equation*}
  E_{ijkl} = E_{jikl} = E_{ijlk} = E_{klij},
\end{equation*}
there are exactly \emph{eight symmetry classes}~\cite{FV1996}: $[\SO(3)]$ (isotropic), $[\octa]$ (cubic), $[\OO(2)]$ (transversely isotropic), $[\DD_{3}]$ (trigonal), $[\DD_{4}]$ (tetragonal), $[\DD_{2}]$ (orthotropic), $[\ZZ_{2}]$ (monoclinic) and $[\triv]$ (triclinic). They are described in~\autoref{sec:symmetry-classes}. The problem of their determination has a long history, recalled by Forte and Vianello in~\cite{FV1996}. These authors have definitively clarified the mathematical problem of classifying the symmetry classes of the representation of $\SO(3)$ on $\Ela$, the $21$-dimensional space of elasticity tensors. They removed the link with crystallographic point groups which was extremely confusing and lead to the false assumption that there were ten, rather than eight, symmetry classes~\cite{Fed1968,Cow1987,HP1991}. These eight classes were confirmed in 2001, using an alternative approach~\cite{CVC2001}, where symmetry planes rather than rotations play the central role. Finally, in 2014, a definitive and systematic way to determine the symmetry classes of any finite dimensional representation of the groups $\SO(2)$, $\SO(3)$, $\OO(2)$ or $\OO(3)$ was formulated~\cite{OA2013,Oli2014,OA2014a,Oli2019}, using a strategy which was initiated in the nineteens~\cite{CLM1991,CG1994,CG1996}.

The explicit calculation of the symmetry class \emph{of a given elasticity tensor} has been an active subject of research in mechanics of deformable solids. Besides, the problem becomes even more complicated if one considers that, in \emph{real life}, a measured elasticity tensor (assuming that one can access to all of its components) is subject to experimental errors and has therefore no symmetry but is nevertheless \emph{close} to a given theoretical tensor with a \emph{given symmetry}~\cite{GTT1963,MN2006}. Concerning this, it is worth to cite the excellent work of François and coauthors~\cite{Fra1995,FBG1998} who performed a deep experimental and numerical study of the problem using acoustic measurements on polyhedral testing samples of raw materials.

In the present work, we address the following theoretical question.

\begin{quote}
  \emph{Can one decide by finitely many algebraic calculations what is the symmetry class of a given elasticity tensor?}
\end{quote}

In addition to experimental and numerical approaches, the literature is abundant about formulations of coordinate-free criteria to characterize elasticity tensors which have \emph{exactly} a given symmetry class. Some authors~~\cite{BBS2007} have used the \emph{Kelvin representation}~\cite{TKel1856,TKel1878,Ryc1984} of the elasticity tensor to achieve this goal. They have formulated sufficient conditions involving the multiplicity of the $6$ eigenvalues of the Kelvin representation and of the eigenvalues of its eigenvectors (the \emph{eigenstrains}, which are second-order tensors). Other approaches have used the harmonic decomposition $(\bH, \bd^{\prime}, \bv^{\prime}, \lambda, \mu)$ of the elasticity tensor $\bE$. For instance, following~\cite{Cow1987}, some authors \cite{Cow1989,Jar1994,Bae1998a,CVC2001} have extracted partial information about the symmetry class of $\bE$ from its second-order harmonic components $\bd^{\prime}$ and $\bv^{\prime}$. In the same spirit, but to avoid loosing important information present in the harmonic fourth-order component $\bH$, Baerheim~\cite{Bae1998a} has used the \emph{harmonic factorization} introduced by Sylvester~\cite{Syl1909} (see also~\cite{Bac1970} and \cite{OKDD2018} for a more modern treatment). This factorization allows to decompose an harmonic tensor of order $n$ as an $n$-tuple of vectors, the \emph{Maxwell multipoles}~\cite{Bac1970}. Baerheim has used these multipoles to detect the different symmetry classes of $\bE$. This approach requires, however to solve polynomial equations of degree eight.

More recently, in~\cite{AKP2014}, the authors have suggested to reconsider the question in the general framework of \emph{Real Algebraic Geometry}~\cite{BCR1998}. They have used a generating set of the invariant algebra of \emph{fourth-order harmonic tensors} $\HH^{4}(\RR^{3})$ proposed in~\cite{BKO1994} to characterize the symmetry classes of a tensor $\bH\in\HH^{4}(\RR^{3})$, writing down polynomial equations and inequalities involving the generators of the invariant algebra. These relations become, however, increasingly complicated when the symmetry group becomes smaller and only the cubic, transversely isotropic, tetragonal, trigonal and orthotropic classes have been characterized this way. The same approach was already used by Vianello~\cite{Via1997} in 1997 for the full 2D elasticity tensor, where there are only four symmetry classes: $[\ZZ_{2}]$, $[\DD_{2}]$, $[\DD_{4}]$ and $[\OO(2)]$, and where formulas are considerably much simpler than in 3D.

In this paper, we propose to give a definitive answer to this classification problem for the full elasticity tensor, using \emph{polynomial covariants} rather than invariants (and avoiding, this way, increasing complexity). The answer is furnished by theorem~\ref{thm:main}, which is our main result : the symmetry classes of an elasticity tensor are characterized by polynomial equations involving its covariants. From the mechanical point of view, these polynomial equations provide necessary and sufficient conditions to belong to each of the eight elasticity symmetry classes. These criteria are particularly simple, since one needs only to check that some polynomial functions defined on the components of the elasticity tensor vanish and thus solving any algebraic equation is not necessary. In mathematical terms, the closures of the isotropy strata of the elasticity tensor are real affine algebraic sets and their equations are provided in theorem~\ref{thm:main}.

In order to obtain these results, we have been lead to formulate rigorously what is the \emph{covariant algebra} of a given real representation $\VV$ of the rotation group $\SO(3)$. It was also necessary to introduce a generalization of the cross-product for totally symmetric tensors. Using these tools, we were able to explicit a minimal set of \emph{70 generators for the covariant algebra of $\HH^{4}(\RR^{3})$} in~\autoref{tab:cov-basis-H4}. These fundamental covariants are the main tools which has allowed us to characterize first, the symmetry class of a tensor $\bH \in \HH^{4}(\RR^{3})$ and then, of a full elasticity tensor $\bE$.

A by-product of these achievements is the production of a new set of generators for the invariant algebra of the elasticity tensor, using the covariants in~\autoref{tab:cov-basis-H4} and the covariant tensor operations of section~\ref{sec:covariant-tensor-operations}. These generators are given in \autoref{sec:elasticity-invariants} and shall be more useful for the mechanical community than the original invariants furnished in~\cite{OKA2017}, which were described using \emph{transvectants}~\cite{Olv1999}.

{The explicit characterization of membership to given symmetry class (provided in theorem~\ref{thm:main}) is a valuable tool when considering the synthesis problem of new metamaterials~\cite{DS2020,BSP2019,YTB2019}. New arrangements of given materials at the micro-scale are often studied using an optimization problem which involves the effective elasticity tensor and an homogenization approach~\cite{MHB2017,MBH2017,YATM2020}. In this context, we think that the explicit proposed characterization might be efficiently integrated into the (numerically solved) optimization problem. Moreover, the proposed methodology is not limited to classical elasticity. It could be applied to strain gradient elasticity (for which a justification can be found in~\cite{MAd2020}), or to related continuum models such as micromorphic continuum~\cite{FS2020}. The synthesis problem of such continua, obtained using lattice structure~\cite{EdS2019,SYAC2020,Ere2019}, would also benefit of such explicit symmetry classes characterization. For these, even large deformation~\cite{BEDH2020} or singular constitutive laws~\cite{EAC2019} could be considered.}

\subsection*{Organization of the paper}

The paper is organized as follows. In~\autoref{sec:sym-harm-tensors}, we provide basic definitions and recall the link between totally symmetric tensors and homogeneous polynomials. In~\autoref{sec:covariant-tensor-operations}, we recall the basic covariant operations on tensors and introduce the \emph{generalized cross-product} between totally symmetric tensors. The~\autoref{sec:polynomial-covariants} is devoted to the definition of polynomial covariants of a linear representation and basic facts about the covariant algebra. A minimal generating set of 70 polynomial covariants for $\HH^{4}(\RR^{3})$ is provided in~\autoref{sec:H4-covariant-algebra}. The symmetry classes are introduced in section~\ref{sec:symmetry-classes} and a way to compute them is provided. The~\autoref{sec:covariants-symmetry-dimension} provides several lemmas which connect the dimension of covariant spaces of order one and two to their symmetry class. In~\autoref{sec:covariant-criteria}, several criteria which restrict the symmetry class of one or several totally symmetric tensors, using polynomial covariants are formulated. The characterization of the symmetry class of a fourth-order harmonic tensor $\bH$ using polynomial covariants is given in~\autoref{sec:H4-symmetry-classes} and the result for a full elasticity tensor $\bE$ is given in section~\ref{sec:Ela-symmetry-classes}. In addition, three appendices are provided. In~\autoref{sec:binary-forms-covariants}, we recall the basics about the spaces of binary forms of degree $n$, $\Sn{n}$ (which are models for irreducible algebraic representations of $\SL(2,\CC)$) and we relate the invariant algebra of $\Sn{2n}\oplus\Sn{2}$ to the covariant algebra of $\Sn{2n}$. In~\autoref{sec:harmonic-tensors-covariants}, we explain how we have been able to compute a minimal set of generators for the covariant algebra of $\HH^{4}(\RR^{3})$ using the knowledge of a minimal set of generators for the covariant algebra of $\Sn{8}$.  Finally, in~\autoref{sec:elasticity-invariants}, we provide a new set of 297 generators for the invariant algebra of the elasticity tensor using the tensorial covariants provided in section~\ref{sec:H4-covariant-algebra}.

\section{Symmetric and harmonic tensors}
\label{sec:sym-harm-tensors}

Let $\TT^{n}(\RR^{3})$ be the vector space of $n$-th order tensors on the Euclidean space $\RR^{3}$. Thanks to the Euclidean product, we do not have to distinguish between \emph{upper} and \emph{lower} indices. Therefore, an $n$-th order tensor may always be considered as a $n$-linear mapping
\begin{equation*}
  \bT: \RR^{3} \times \dotsb \times \RR^{3} \to \RR, \qquad (\xx_{1},\dotsc,\xx_{n}) \mapsto \bT(\xx_{1},\dotsc,\xx_{n}).
\end{equation*}
The subspace $\Sym^{n}(\RR^{3})$ of totally symmetric tensors can be identified with the vector space $\Pn{n}(\RR^{3})$ of homogeneous polynomials of degree $n$. This isomorphism generalizes, to higher order tensors, the well-known connection between quadratic forms and symmetric bilinear forms obtained by \emph{polarization}.

\begin{ex}
  For instance, the polynomial representation of a totally symmetric fourth-order tensor $\bS = (S_{ijkl})$ is given by
  \begin{multline*}
    S_{1111}\,x^{4} + S_{2222}\,y^{4} + S_{3333}\,z^{4} + 12S_{1122}\,x^{2}y^{2} + 12S_{1133}\,x^{2}z^{2}
    \\
    + 12S_{2233}\,y^{2}z^{2} + 4S_{1222}\,xy^{3} + 4S_{1112}\,yx^{3} + 4S_{1113}\,zx^{3} + 4S_{1333}\,xz^{3}
    \\
    + 4S_{2223}\,zy^{3} + 4S_{2333}\,yz^{3} + 6S_{1123}\,yzx^{2} + 6S_{1223}\,xzy^{2} + 6S_{1233}\,xyz^{2}.
  \end{multline*}
\end{ex}

Contracting two indices $i,j$ on a totally symmetric tensor $\bS$ does not depend on the particular choice of the pair $i,j$. Thus, we can refer to this contraction without any reference to a particular choice of indices. We will denote this contraction as $\tr \bS$, which is a totally symmetric tensor of order $n-2$ and is called the \emph{trace} of $\bS$.

\begin{defn}
  An $n$-th order totally symmetric and \emph{traceless} tensor will be called an \emph{harmonic tensor} and the subspace of $\Sym^{n}(\RR^{3})$ of harmonic tensors will be denoted by $\HH^{n}(\RR^{3})$ (or simply $\HH^{n}$, if there is no ambiguity).
\end{defn}

\begin{rem}
  In the correspondence between totally symmetric tensors and homogeneous polynomials, a \emph{traceless} totally symmetric tensor $\bH$ corresponds to an \emph{harmonic polynomial} $\rh$ (\textit{i.e.} with vanishing Laplacian: $\triangle \rh = 0$) and this justifies the appellation of \emph{harmonic tensor}. The space of homogeneous harmonic polynomials of degree $n$ will be denoted by $\Hn{n}(\RR^{3})$.
\end{rem}

The natural action of the special orthogonal group $\SO(3)$ (or the full orthogonal group $\OO(3)$) on $\RR^{3}$ induces the tensorial representation $\rho_n$ on $\TT^{n}(\RR^{3})$, defined by
\begin{equation*}
  (\rho_n(g)(\bT))(\xx_{1},\dotsc,\xx_{n})=(g \star \bT)(\xx_{1},\dotsc,\xx_{n}) : =  \bT(g^{-1} \xx_{1},\dotsc,g^{-1} \xx_{n}),
\end{equation*}
where $\bT \in \TT^{n}(\RR^{3})$ and $g \in \SO(3)$. Under this linear representation, the subspaces $\Sym^{n}(\RR^{3})$ and $\HH^{n}(\RR^{3})$ are invariant. Moreover, $\HH^{n}(\RR^{3})$ is \emph{irreducible}~\cite{GSS1988} (its only invariant subspaces are itself and the null space).

\begin{thm}[Harmonic decomposition]
  Every finite dimensional representation $\VV$ of the rotation group $\SO(3)$ can be decomposed into a direct sum of irreducible representations, each of them being isomorphic to an harmonic tensor space $\HH^{n}(\RR^{3})$, by an equivariant isomorphism.
\end{thm}

\begin{rem}
  An alternative model for the irreducible representations of $\SO(3)$ is furnished by the spaces of harmonic polynomials $\Hn{n}(\RR^{3})$, where the action of $\SO(3)$ on polynomials is given by $(g \star \rp)(\xx) : =  \rp(g^{-1} \xx)$.
\end{rem}

\begin{ex}
  Every homogeneous polynomial of degree $n$ can be decomposed as the following:
  \begin{equation}\label{eq:polynomial-harmonic-decomposition}
    \rp = \rh_{0} + \qq\,\rh_{1} + \dotsb + \qq^{r}\rh_{r},
  \end{equation}
  where $\qq = x^{2} + y^{2} + z^{2}$, $r = [n/2]$ -- with $[\cdot]$ integer part -- and $\rh_{k}$ is a harmonic polynomial of degree $n-2k$.
\end{ex}

\begin{defn}
  Given a homogeneous polynomial $\rp$, the highest order component in~\eqref{eq:polynomial-harmonic-decomposition}, namely $\rh_{0}$, which is uniquely defined, is called the \emph{harmonic projection} of $\rp$ and denoted $(\rp)_{0}$.
\end{defn}

\section{Covariant operations on tensors}
\label{sec:covariant-tensor-operations}

In this section, we will introduce three operations on tensors, which \emph{commute with the action of the rotation group} and are thus called \emph{covariant operations}. The first one is the \emph{symmetric tensor product}.

\begin{defn}[Symmetric tensor product]
  The symmetric tensor product between two tensors $\bT^{1} \in \TT^{p}(\RR^{3})$ and $\bS^{2} \in \TT^{q}(\RR^{3})$ is defined as
  \begin{equation*}
    \bT^{1}\odot\bT^{2} : =  (\bT^{1} \otimes \bT^{2})^{s} \in \Sym^{p + q}(\RR^{3}),
  \end{equation*}
  where the total symmetrisation of a tensor $\bT \in \TT^{n}(\RR^{3})$, noted $\bT^{s} \in \Sym^{n}(\RR^{3})$, is defined as
  \begin{equation*}
    \bT^{s}(\xx_{1},\dotsc,\xx_{n}) : =  \frac{1}{n!}\sum_{\sigma \in \mathfrak{S}_{n}} \bT(\xx_{\sigma(1)},\dotsc,\xx_{\sigma(n)})
  \end{equation*}
  where $\mathfrak{S}_{n}$ is the symmetric group on $n$ letters.
\end{defn}

\begin{rem}
  When \emph{restricted to totally symmetric tensors}, the polynomial counterpart of the symmetric tensor product is just the usual product of polynomials. This product is thus associative and commutative. It is equivariant relative to either the rotation group $\SO(3)$ and the full orthogonal group $\OO(3)$.
\end{rem}

\begin{rem}
  The \emph{harmonic decomposition}~\eqref{eq:polynomial-harmonic-decomposition} of an homogenous polynomial of degree $n$ leads thus to the following harmonic decomposition of a totally symmetric tensor $\bS \in \Sym^{n}(\RR^{3})$:
  \begin{equation}\label{eq:symmetric-harmonic-decomposition}
    \bS = \bH_{0} + \bq \odot \bH_{1} + \dotsb + \bq^{\odot r-1}\odot \bH_{r-1}+ \bq^{\odot r} \odot\bH_{r},
  \end{equation}
  where $\bH_{k}$ is an harmonic tensor of degree $n-2k$. In this formula, $\bq \in \Sym^{2}(\RR^{3})$ is the \emph{Euclidean metric tensor} (which writes as $\bq =(\delta_{ij})$ in any orthonormal basis) and $\bq^{\odot k}$ means the symmetrized tensorial product of $k$ copies of $\bq$.
\end{rem}

The second one is the contraction between two tensors $\bT^1\in\TT^{p}(\RR^{3})$ and $\bT^{2}\in\TT^{q}(\RR^{3})$ over one or several subscripts. This operation uses the Euclidean structure represented by the canonical Euclidean metric tensor $\bq = (q_{ij})$ and its inverse $\bq^{-1} = (q^{ij})$. It is defined as follows:
\begin{equation*}
  (\bT^1 \overset{(r)}{\cdot} \bT^{2})_{i_{1} \dotsb i_{p-r}j_{r+1} \dotsb j_{q}} = q^{i_{p-r+1}j_{1}} \dotsm q^{i_{p}j_{r}} T^1_{i_{1}\dotsb i_{p}} \, T^{2}_{j_{1} \dotsb j_{q}}.
\end{equation*}
The $r$-contraction of two tensors is an $\OO(3)$-equivariant mapping
\begin{equation*}
  \TT^{p}(\RR^{3}) \times \TT^{q}(\RR^{3}) \to \TT^{p+q-2r}(\RR^{3}),
\end{equation*}
and for $n=p=q$, the $n$-contraction corresponds to the canonical scalar product on $\TT^{n}(\RR^{3})$.

\begin{ex}
  In an \emph{orthonormal basis} $(\ee_{i})$, we have
  \begin{align*}
    (\bT^1 \cdot \bT^{2})_{i_{1} \dotsb i_{p-1}j_{2} \dotsb j_{q}}  & = T^1_{i_{1}\dotsb i_{p-1}k} \, T^{2}_{kj_{2} \dotsb j_{q}},     \\
    (\bT^1 \2dots \bT^{2})_{i_{1} \dotsb i_{p-2}j_{3} \dotsb j_{q}} & = T^1_{i_{1}\dotsb i_{p-2}kl} \, T^{2}_{klj_{3} \dotsb j_{q}},   \\
    (\bT^1 \3dots \bT^{2})_{i_{1} \dotsb i_{p-3}j_{4} \dotsb j_{q}} & = T^1_{i_{1}\dotsb i_{p-3}klm} \, T^{2}_{klmj_{4} \dotsb j_{q}}.
  \end{align*}
\end{ex}

\begin{defn}[Symmetric $r$-contraction]
  The \emph{symmetric $r$-contraction} between two totally symmetric tensors $\bS^{1} \in \Sym^{p}(\RR^{3})$ and $\bS^{2} \in \Sym^{q}(\RR^{3})$ is defined as
  \begin{equation*}
    (\bS^1 \overset{(r)}{\cdot} \bS^{2})^{s}.
  \end{equation*}
\end{defn}

\begin{rem}
  The polynomial counterpart of the symmetric $r$-contraction is obtained as follows. If $\bS^1, \bS^{2}$ correspond respectively to the polynomials $\rp_{1}, \rp_{2}$, then, $(\bS^1 \overset{(r)}{\cdot} \bS^{2})^s$ corresponds to the polynomial
  \begin{equation*}
    \rp = \frac{(p-r)!}{p!}\frac{(q-r)!}{q!} \sum_{k_{1}+k_{2}+k_{3}=r} \frac{r!}{k_{1}!k_{2}!k_{3}!}\frac{\partial^r \rp_{1}}{\partial x^{k_{1}}\partial y^{k_{2}}\partial z^{k_{3}}}\frac{\partial^r \rp_{2}}{\partial x^{k_{1}}\partial y^{k_{2}}\partial z^{k_{3}}}.
  \end{equation*}
\end{rem}

The third covariant operation is the \emph{generalized cross product}, which extends the standard cross product between vectors of $\RR^{3}$ to symmetric tensors of arbitrary order.

\begin{defn}[Generalized cross product]\label{def:generalized-cross-product}
  The generalized cross product (or \emph{Lie-Poisson product}) between two totally symmetric tensors $\bS^{1} \in \Sym^{p}(\RR^{3})$ and $\bS^{2} \in \Sym^{q}(\RR^{3})$ is defined as
  \begin{equation*}
    \bS^{1} \times \bS^{2} : = - \left(\bS^{1}\cdot \pmb \varepsilon \cdot \bS^{2}\right)^s
    \in \Sym^{p + q -1}(\RR^{3}).
  \end{equation*}
  where $\pmb \varepsilon$ is the \emph{Levi--Civita} tensor. In any orthonormal basis, we get
  \begin{equation*}
    (\bS^{1}\times\bS^{2})_{i_{1}\dotsb i_{p+q-1}} := (\varepsilon_{i_{1}jk}S^{1}_{ji_{2}\dotsb i_{p}}S^{2}_{ki_{p+1} \dotsb i_{p+q-1}})^{s}
  \end{equation*}
\end{defn}

\begin{rem}
  The generalized cross product is skew-symmetric:
  \begin{equation*}
    \bS^{2} \times \bS^{1} = -\bS^{1} \times \bS^{2}.
  \end{equation*}
  Its polynomial counterpart is (up to a scaling factor) the \emph{Lie--Poisson bracket} on $\so^{*}(3,\RR)$, the dual of the Lie algebra of the rotation group, isomorphic to $\RR^{3}$ (see~\cite[Section 1.6]{AK1998} for the definition and an overview of the use of this bracket in mechanics). More precisely, if $\rp_{1}, \rp_{2}$ are the polynomial representatives of $\bS^{1},\bS^{2}$, then the polynomial representative of $\bS^{1} \times \bS^{2}$ is
  \begin{equation*}
    \frac{1}{pq} \{\rp_{1}, \rp_{2}\}_{LP} = \frac{1}{pq}\det(\xx,\nabla \rp_{1}, \nabla \rp_{2}).
  \end{equation*}
  This product is equivariant relative to the rotation group $\SO(3)$ but not to full orthogonal group $\OO(3)$. In that later case, we get
  \begin{equation*}
    (g \star \bS^{1}) \times (g \star \bS^{2}) = (\det g) \left(g \star(\bS^{1} \times \bS^{2})\right).
  \end{equation*}
\end{rem}

\begin{rem}\label{rem:Sxq=0}
  Note that if $\bq$ is the Euclidean metric tensor, then $\bS \times \bq = 0$ for every totally symmetric tensor $\bS$ (indeed, the radial function $\qq = x^{2} + y^{2} + z^{2}$ is a \emph{Casimir function}~\cite[Section 1.6]{AK1998} for the Lie-Poisson bracket on $\so^{*}(3,\RR)$). In particular,
  $\bS \times \ba = \bS \times \ba'$ for every symmetric \emph{second-order} tensor $\ba$, where
  \begin{equation*}
    \label{eq:dev}
    \ba^\prime=\ba-\frac{1}{3} \tr (\ba)\, \bq
  \end{equation*}
  is the deviatoric (\emph{i.e.} harmonic) part of $\ba$.
\end{rem}

\section{Polynomial covariants}
\label{sec:polynomial-covariants}

Let $\VV$ be a finite dimensional representation of a group $G$. The linear action of $G$ on $\VV$ extends naturally to the algebra $\RR[\VV]$ of real polynomial functions defined on $\VV$ by
\begin{equation*}
  (g\star \rp)(\vv) : =  \rp(g^{-1}\star \vv), \qquad \rp \in \RR[\VV], \, g \in G.
\end{equation*}
A polynomial $\rp \in \RR[\VV]$ is \emph{invariant} if $g\star \rp = \rp$ for all $g\in G$. The set $\RR[\VV]^{G}$, also noted $\inv(\VV)$, of all invariant polynomials is a sub-algebra of $\RR[\VV]$, called the \emph{invariant algebra} of $\VV$. In~\cite{KP2000}, Kraft and Procesi have generalized the concept of invariants in the following way.

\begin{defn}
  Given two representations $\VV$ and $\WW$ of a group $G$, we define $\pol(\VV,\WW)$ to be the space of polynomial mappings $\rp$ from $\VV$ to $\WW$ (\emph{i.e} each component function is a polynomial expression of the components of $\vv \in \VV$, and such in any basis). A \emph{polynomial covariant of $\VV$ of type $\WW$} is a $G$-equivariant polynomial mapping $\rp : \VV \to \WW$, which means that
  \begin{equation*}
    \rp(g\star \vv) = g\star \rp(\vv), \qquad \forall \vv \in \VV, \, \forall g \in G.
  \end{equation*}
\end{defn}

The problem with this definition is that the set $\pol(\VV,\WW)^{G}$, of polynomial covariant of $\VV$ of type $\WW$ is only a vector space and not an algebra. We will therefore extend this definition as follows (see also~\cite[Page 184]{SPV1994} for a more general and abstract definition of this concept).

\begin{defn}\label{def:covariant-algebra}
  Let $\VV, \WW$ be finite dimensional representations of a group $G$. The \emph{covariant algebra of $\VV$ of type $\WW$}, noted $\cov(\VV,\WW)$, is defined as the invariant algebra
  \begin{equation*}
    \RR[\VV \oplus \WW^{*}]^{G},
  \end{equation*}
  where $\WW^{*}$ is the dual vector space of $\WW$.
\end{defn}

\begin{rem}
  We can define similarly, $\mathbf{Con}(\VV,\WW)$, the \emph{contravariant algebra of $\VV$ of type $\WW$} as $\RR[\VV \oplus \WW]^{G}$. However, if $\WW$ and $\WW^{*}$ are equivalent representations (for instance if the representation $\WW$ is unitary), we do not have to distinguish between these two algebras which are canonically isomorphic.
\end{rem}

Note that the covariant algebra $\cov(\VV,\WW)$ has a natural bi-graduation. It is the direct sum of the finite dimensional vector spaces $\cov_{d,k}(\VV,\WW)$ of bi-homogeneous polynomial $p(\vv,\omega)$:
\begin{itemize}
  \item of total degree $d$ in $\vv \in \VV$, called the \textbf{degree} of the covariant,
  \item and, of total degree $k$ in $\omega \in \WW^{*}$, called the \textbf{order} of the covariant.
\end{itemize}
Furthermore, the subspace of covariants of order $0$ is identical to the invariant algebra of $\VV$.

\begin{rem}
  The vector space of polynomial covariants $\pol(\VV,\WW)^{G}$ can thus be identified with
  \begin{equation*}
    \cov_{1}(\VV,\WW) = \bigoplus_{k=0}^{+\infty}\cov_{k,1}(\VV,\WW),
  \end{equation*}
  the vector space of first-order covariants.
\end{rem}

In this paper, we will only be interested when $G = \SO(3)$ and $\WW$ is the Euclidean space $\RR^{3}$ (in which case, we do not have to make any difference between the covariant and the contravariant algebras), and we will set
\begin{equation*}
  \cov(\VV) := \RR[\VV \oplus \RR^{3}]^{\SO(3)}.
\end{equation*}
An element $\rp \in \cov(\VV)$ is thus a polynomial which can be written as
\begin{equation*}
  \rp(\vv,\xx) = \sum_{i,j,k} p_{ijk}(\vv)x^{i}y^{j}z^{k},
\end{equation*}
where each coefficient $p_{ijk}(\vv)$ is a polynomial function of $\vv$ and such that
\begin{equation*}
  \rp(g \star \vv, \xx) = \rp(\vv,g^{-1} \star \xx),
\end{equation*}
for all $\vv \in \VV$, $\xx \in \RR^{3}$ and $g \in \SO(3)$.

\begin{rem}
  Any homogeneous polynomial covariant of $\vv \in \VV$ of degree $d$ and of type $\Sym^{k}(\RR^{3})$ can thus be identified with a polynomial in $\cov_{d,k}(\VV)$.
\end{rem}

One fundamental result, obtained in the nineteenth century, is that the invariant and covariant algebras of a finite dimensional representation of a compact group is finitely generated.

\begin{thm}[Hilbert's Theorem~\cite{Hil1993}]
  The covariant algebra $\cov(\VV)$ is finitely generated, i.e. there exists a finite set $\mathcal{B} : =  \set{\rp_{1},\dotsc,\rp_s}$ in $\cov(\VV)$ such that
  \begin{equation*}
    \cov(\VV) = \RR[\rp_{1},\dotsc,\rp_s].
  \end{equation*}
  Moreover, one can always find such a system where the $\rp_{j}$ are bi-homogeneous, both in $\vv \in \VV$ and $\xx \in \RR^{3}$.
\end{thm}

\begin{defn}
  A set of generators $\mathcal{B}$ for $\cov(\VV)$ is called an \emph{integrity basis}. An integrity basis $\mathcal{B}$ is \emph{minimal} if no proper subset of it is an integrity basis.
\end{defn}

\begin{ex}
  A minimal integrity basis for $\cov(\Sym^{2}(\RR^{3}))$ is provided by three invariants $\tr(\ba)$, $\tr(\ba^{2})$ and $\tr(\ba^{3})$, three order $2$ covariants $\bq$, $\ba$ and $\ba^{2}$ and one order $3$ covariant $\ba\times \ba^{2}$.
\end{ex}

\begin{rem}\label{rem:cardinal-minimal-basis}
  Of course, a minimal integrity basis is not unique. However, its cardinality $n(\VV)$ is a constant. To see this, as in~\cite{DL1985/86}, set
  \begin{equation*}
    \cov^{ + }(\VV) : =  \sum_{d + m>0} \cov_{d,m}(\VV),
  \end{equation*}
  which is an ideal of the graded algebra $\cov(\VV)$. Then $(\cov^{ + }(\VV))^{2}$ is the space of covariants which can be written as a sum of reducible covariants. For each $(d,m)$ such that $d + m>0$, let $\delta_{d,m}$ be the codimension of $(\cov^{ + }(\VV))^{2}_{d,m}$ in $\cov_{d,m}(\VV)$. Since $\cov(\VV)$ is finitely generated, there exists an integer $p$ such that $\delta_{d,m} = 0$ for $d + m \ge p$ and we can define
  \begin{equation*}
    n(\VV) : =  \sum_{d,m} \delta_{d,m}.
  \end{equation*}
  Then, any \emph{minimal integrity basis} is of cardinal $n(\VV)$. As far as we know, there is no way to obtain the constant $n(\VV)$ but to compute a minimal basis.
\end{rem}

\section{A minimal integrity basis for the covariant algebra of \texorpdfstring{$\HH^{4}$}{H4}}
\label{sec:H4-covariant-algebra}

In this section, we propose to describe a minimal integrity basis for $\cov(\HH^{4})$. As detailed in \autoref{sec:harmonic-tensors-covariants}, $\cov(\HH^{4})$ is connected with the invariant algebra $\inv(\Sn{8}\oplus\Sn{2})$ (theorem~\ref{thm:decomplexification}), where $\Sn{n}$ is the space of binary forms of degree $n$ (see \autoref{sec:binary-forms-covariants}). This algebra is itself connected to the covariant algebra of the binary form of degree 8, $\cov(\Sn{8})$ (theorem~\ref{thm:basis-for-S2n-oplus-S2}). A minimal covariant basis for $\cov(\Sn{8})$ is known at least partially since 1880 and was first produced by von Gall~\cite{vGal1880} (see also~\cite{Bed2008,Cro2002,Oli2017,OKA2017}). These results have been used to obtain degrees and orders of a minimal basis for $\cov(\HH^{4})$ which are given in~\autoref{tab:deg-ord-cov-H4}.

\begin{table}[ht]
  \begin{center}
    \setlength{\arraycolsep}{1pt}
    \begin{tabular}{ccccccccccccc}
      \toprule
      degree / order & 0 & 1  & 2  & 3  & 4 & 5 & 6 & 7 & 9 & \# & Cum \\
      \midrule
      0              & - & -  & 1  & -  & - & - & - & - & - & 1  & 1   \\
      1              & - & -  & -  & -  & 1 & - & - & - & - & 1  & 2   \\
      2              & 1 & -  & 1  & -  & 1 & - & 1 & - & - & 4  & 6   \\
      3              & 1 & -  & 1  & 1  & 1 & 1 & 1 & 1 & 1 & 8  & 14  \\
      4              & 1 & -  & 2  & 1  & 1 & 2 & 1 & 1 & 1 & 10 & 24  \\
      5              & 1 & 1  & 2  & 2  & 1 & 3 & - & 1 & - & 11 & 35  \\
      6              & 1 & 1  & 2  & 3  & 1 & 1 & - & - & - & 9  & 44  \\
      7              & 1 & 2  & 2  & 3  & - & - & - & - & - & 8  & 52  \\
      8              & 1 & 2  & 2  & 2  & - & - & - & - & - & 7  & 59  \\
      9              & 1 & 3  & 1  & -  & - & - & - & - & - & 5  & 64  \\
      10             & 1 & 2  & -  & -  & - & - & - & - & - & 3  & 67  \\
      11             & - & 2  & -  & -  & - & - & - & - & - & 2  & 69  \\
      12             & - & 1  & -  & -  & - & - & - & - & - & 1  & 70  \\
      \midrule
      Tot            & 9 & 14 & 14 & 12 & 6 & 7 & 3 & 3 & 2 & 70 &     \\
      \bottomrule
    \end{tabular}
    \caption{Degrees and orders of a minimal covariant basis for $\cov(\HH^{4})$}
    \label{tab:deg-ord-cov-H4}
  \end{center}
\end{table}

Once we know the information provided in~\autoref{tab:deg-ord-cov-H4}, we have a lot of freedom in the choice of an explicit minimal basis. Checking that a system of 70 arbitrary covariants satisfying the requirements of~\autoref{tab:deg-ord-cov-H4} is a minimal integrity basis requires moreover the knowledge of the \emph{Hilbert series}~\cite{Spr1980}
\begin{equation*}
  H(z,t): = \sum_{d,k\geq 0} a_{d,k}z^dt^k,
\end{equation*}
which encodes the dimension $a_{d,k}$ of each finite dimensional vector space $\cov_{d,k}(\HH^{4})$. However, the Hilbert series $H(z,t)$ is a rational function which can be computed \textit{a priori}~\cite{Spr1980,Spr1983,LP1990,Stu2008,Bed2009,Bed2011}, using the Molien-Weyl formula~\cite{Stu2008}:
\begin{equation*}
  H(z,t) = \int_{\SO(3)} \frac{1}{\det (I - t\rho_{1}(g))}\frac{1}{\det (I - z\rho_{4}(g))} \, d\mu(g)
\end{equation*}
where $d\mu$ is the Haar measure on $\SO(3)$ (see~\cite[Section 4.1]{Ste1994}), $\rho_{1}$ is the standard representation of $\SO(3)$ on $\RR^{3}$ and $\rho_{4}$ is the representation of $\SO(3)$ on $\HH^{4}$. Thus, for each module $\cov_{d,k}(\HH^{4})$ where $(d,k)$ appears in~\autoref{tab:deg-ord-cov-H4}, we have checked inductively on $n = d+k$ that adding new covariants of immediate superior degree/order to the subspace generated by reducible covariants of lower order/degree, we obtain a vector space of dimension $a_{d,k}$.

\begin{thm}\label{thm:H4-covariants}
  The polynomial covariant algebra of $\HH^{4}$ is generated by a minimal basis of 70 homogeneous covariant polynomials, which degree/order are provided in~\autoref{tab:deg-ord-cov-H4}. An explicit basis has been computed in~\autoref{tab:cov-basis-H4}.
\end{thm}

In~\autoref{tab:cov-basis-H4}, we have introduced the following symmetric second-order covariants
\begin{equation*}
  \bd_{2} : = \tr_{13} \bH^{2}, \qquad   \bd_{3} : =  \tr_{13} \bH^{3}, \qquad \bc_{k} : =  \bH^{k-2}\2dots\bd_{2}, \quad k \ge 3.
\end{equation*}
where $\bH^{n}:=\bH:\bH^{n-1}$ for $n\geq 2$ and $\tr_{13}\tq{A}$ of a fourth order tensor $\tq{A}$ is defined as $(\tr_{13}\tq{A})_{ij}:=\tq{A}_{kikj}$ (in any orthonormal basis). We have also used the simplified notation $\ba\bb:=\ba\cdot\bb$, when $\ba$ and $\bb$ are second order tensors.

\begin{rem}\label{rem:d3c3}
  Note the following relation
  \begin{equation*}
    \bc_{3} = 2\,\bd_{3}',
  \end{equation*}
  which can be checked by a direct calculation.
\end{rem}

\begin{rem}
  In addition to $\bd_{2} $ and $\bd_{3}$, the following second-order covariants were introduced in~\cite{BKO1994}:
  \begin{equation}\label{eq:Boehler-covariants}
    \begin{aligned}
      \bd_{4}  & : =  {\bd_{2}}^{2},                             & \bd_{5} & : =  \bd_{2} (\bH\2dots \bd_{2}),         & \bd_{6} & : =  {\bd_{2}}^{3},                         \\
      \bd_{7}  & : = {\bd_{2}}^{2} (\bH\2dots \bd_{2}),          & \bd_{8} & : = {\bd_{2}}^{2} (\bH^{2}\2dots\bd_{2}), & \bd_{9} & : = {\bd_{2}}^{2} (\bH\2dots{\bd_{2}}^{2}), \\
      \bd_{10} & : = {\bd_{2}}^{2} (\bH^{2}\2dots{\bd_{2}}^{2}).
    \end{aligned}
  \end{equation}
  For $k = 2,3,4,6$, the $\bd_{k}$ are \emph{symmetric}, while they are not for $k = 5,7,8,9,10$. None of them are harmonic. These covariants were used to define the following invariants:
  \begin{equation}\label{eq:Boehler-invariants}
    J_{k} : =  \tr \bd_{k} , \qquad k = 2, \dotsc ,10,
  \end{equation}
  which constitute a \emph{minimal integrity basis} for $\HH^{4}$ (see~\cite{BKO1994}). In~\autoref{tab:cov-basis-H4}, we did not use the invariants $J_{k}$ but an alternative set of generators $I_k$. The nine invariants $J_{k}$ are not algebraically independent (neither are the nine invariants $I_k$); they are subject to some algebraic relations, which have been calculated first by Shioda~\cite{Shi1967} (with some minor errors).
\end{rem}

\begin{table}[p]
  \centering
  \begin{minipage}{0.49\linewidth}
    \resizebox{\textwidth}{!}{%
      \renewcommand{\arraystretch}{1.2}
      \begin{tabular}{|c|c|c|c|c|}
        \hline
        \# & Cov.          & Deg. & Ord. & Formula                                                              \\
        \hline \rule{0pt}{3ex}
        1  & $I_{2}$       & 2    & 0    & $\tr(\bd_{2})$                                                       \\
        2  & $I_{3}$       & 3    & 0    &
        $\tr(\bd_{3})$                                                                                          \\
        3  & $I_{4}$       & 4    & 0    & $\tr( {\bd_{2}}^{2})$                                                \\
        4  & $I_{5}$       & 5    & 0    & $\tr (\bd_{2} \bd_{3})$                                              \\
        5  & $I_{6}$       & 6    & 0    & $\tr ({\bd_{2}}^{3})$                                                \\
        6  & $I_{7}$       & 7    & 0    & $\tr ({\bd_{2}}^{2} \bd_{3})$                                        \\
        7  & $I_{8}$       & 8    & 0    & $\tr (\bd_{2} {\bd_{3}}^{2})$                                        \\
        8  & $I_{9}$       & 9    & 0    & $\tr ({\bd_{3}}^{3})$                                                \\
        9  & $I_{10}$      & 10   & 0    & $\tr ({\bd_{2}}^{2}{\bd_{3}}^{2})$                                   \\
        \hline \rule{0pt}{3ex}
        10 & $\cv{5}{1}$   & 5    & 1    & $\vv_{5} = \pmb{\varepsilon} \2dots (\bd_{2}\bc_{3})$                \\
        11 & $\cv{6}{1}$   & 6    & 1    & $\vv_{6} = \pmb{\varepsilon} \2dots (\bd_{2}\bc_{4})$                \\
        12 & $\cv{7a}{1}$  & 7    & 1    & $\vv_{7a} = \pmb{\varepsilon} \2dots ({\bd_{2}}^{2}\bc_{3})$         \\
        13 & $\cv{7b}{1}$  & 7    & 1    & $\vv_{7b} = \pmb{\varepsilon} \2dots (\bc_{4}\bc_{3})$               \\
        14 & $\cv{8a}{1}$  & 8    & 1    & $\vv_{8a} = \pmb{\varepsilon} \2dots (\bd_{2}{\bc_{3}}^{2})$         \\
        15 & $\cv{8b}{1}$  & 8    & 1    & $\vv_{8b} = \pmb{\varepsilon} \2dots ({\bd_{2}}^{2}\bc_{4})$         \\
        16 & $\cv{9a}{1}$  & 9    & 1    & $\vv_{9a} = \pmb{\varepsilon} \2dots (\bd_{2}\bc_{4}\bc_{3})$        \\
        17 & $\cv{9b}{1}$  & 9    & 1    & $\vv_{9b} = \pmb{\varepsilon} \2dots (\bc_{3}\bd_{2}\bc_{4})$        \\
        18 & $\cv{9c}{1}$  & 9    & 1    & $\vv_{9c} = \pmb{\varepsilon} \2dots (\bd_{2}\bc_{3}\bc_{4})$        \\
        19 & $\cv{10a}{1}$ & 10   & 1    & $\vv_{10a} = \pmb{\varepsilon} \2dots ({\bd_{2}}^{2}{\bc_{3}}^{2})$  \\
        20 & $\cv{10b}{1}$ & 10   & 1    & $\vv_{10b} = \pmb{\varepsilon} \2dots ({\bc_{3}}^{2}\bc_{4})$        \\
        21 & $\cv{11a}{1}$ & 11   & 1    & $\vv_{11a} = \pmb{\varepsilon} \2dots (\bc_{3}{\bc_{4}}^{2})$        \\
        22 & $\cv{11b}{1}$ & 11   & 1    & $\vv_{11b} = \pmb{\varepsilon} \2dots ({\bd_{2}}^{2}\bc_{3}\bc_{4})$ \\
        23 & $\cv{12}{1}$  & 12   & 1    & $\vv_{12} = \pmb{\varepsilon} \2dots (\bd_{2}{\bc_{3}}^{2}\bc_{4})$  \\ \hline \rule{0pt}{3ex}
        0  & $\cv{0}{2}$   & 0    & 2    & $\tq{q}$                                                             \\
        24 & $\cv{2}{2}$   & 2    & 2    & $\bd_{2}$                                                            \\
        25 & $\cv{3}{2}$   & 3    & 2    & $\bc_{3}$                                                            \\
        26 & $\cv{4a}{2}$  & 4    & 2    & $\bc_{4}$                                                            \\
        27 & $\cv{4b}{2}$  & 4    & 2    & ${\bd_{2}}^{2}$                                                      \\
        28 & $\cv{5a}{2}$  & 5    & 2    & $\bc_{5}$                                                            \\
        29 & $\cv{5b}{2}$  & 5    & 2    & $(\bd_{2}\bc_{3})^{s}$                                               \\
        30 & $\cv{6a}{2}$  & 6    & 2    & $(\bd_{2}\bc_{4})^{s}$                                               \\
        31 & $\cv{6b}{2}$  & 6    & 2    & ${\bc_{3}}^{2}$                                                      \\
        32 & $\cv{7a}{2}$  & 7    & 2    & $({\bd_{2}}^{2}\bc_{3})^{s}$                                         \\
        33 & $\cv{7b}{2}$  & 7    & 2    & $(\bc_{4}\bc_{3})^{s}$                                               \\
        \hline
      \end{tabular}}
  \end{minipage}
  \begin{minipage}{0.5\linewidth}
    \centering
    \resizebox{\textwidth}{!}{%
      \begin{tabular}{|c|c|c|c|c|}
        \hline
        \# & Cov.              & Deg. & Ord. & Formula                                   \\
        \hline \rule{0pt}{3ex}
        34 & $\cv{8a}{2}$      & 8    & 2    & $(\bd_{2}{\bc_{3}}^{2})^{s}$              \\
        35 & $\cv{8b}{2}$      & 8    & 2    & ${\bc_{4}}^{2}$
        \\
        36 & $\cv{9}{2}$       & 9    & 2    & $({\bd_{2}}^{2}\bc_{5})^{s}$
        \\ \hline \rule{0pt}{3ex}
        37 & $\cv{3}{3}$       & 3    & 3    & $\tr(\bH\times \bd_{2})$
        \\
        38 & $\cv{4}{3}$       & 4    & 3    & $\tr(\bH\times \bc_{3})$
        \\
        39 & $\cv{5a}{3}$      & 5    & 3    & $\bd_{2} \times \bc_{3}$
        \\
        40 & $\cv{5b}{3}$      & 5    & 3    & $\tr(\bH\times {\bd_{2}}^{2})$
        \\
        41 & $\cv{6a}{3}$      & 6    & 3    & $\bd_{2} \times {\bd_{2}}^{2}$
        \\
        42 & $\cv{6b}{3}$      & 6    & 3    & $\bd_{2} \times \bc_{4}$
        \\
        43 & $\cv{6c}{3}$      & 6    & 3    & $\tr(\bH\times \bc_{5})$
        \\
        44 & $\cv{7a}{3}$      & 7    & 3    & ${\bd_{2}}^{2} \times \bc_{3}$
        \\
        45 & $\cv{7b}{3}$      & 7    & 3    & $\bc_{3} \times \bc_{4}$
        \\
        46 & $\cv{7c}{3}$      & 7    & 3    & $\bd_{2} \times \bc_{5}$
        \\
        47 & $\cv{8a}{3}$      & 8    & 3    & $\bd_{2} \times {\bc_{3}}^{2}$
        \\
        48 & $\cv{8b}{3}$      & 8    & 3    & $\bc_{3} \times \bc_{5}$
        \\ \hline \rule{0pt}{3ex}
        49 & $\cv{1}{4} = \bH$ & 1    & 4    & $\bH$
        \\
        50 & $\cv{2}{4}$       & 2    & 4    & $(\bH^{2})^{s}$
        \\
        51 & $\cv{3}{4}$       & 3    & 4    & $(\bH^{3})^{s}$
        \\
        52 & $\cv{4}{4}$       & 4    & 4    & $(\bH^{4})^{s}$
        \\
        53 & $\cv{5}{4}$       & 5    & 4    & $(\bH\cdot {\bd_{2}}^{2})^{s}$
        \\
        54 & $\cv{6}{4}$       & 6    & 4    & $(\bH^{2}\cdot {\bd_{2}}^{2})^{s}$
        \\ \hline \rule{0pt}{3ex}
        55 & $\cv{3}{5}$       & 3    & 5    & $\bH\times \bd_{2}$
        \\
        56 & $\cv{4a}{5}$      & 4    & 5    & $\bH\times \bc_{3}$
        \\
        57 & $\cv{4b}{5}$      & 4    & 5    & $(\bH^{2})^{s}\times \bd_{2}$
        \\
        58 & $\cv{5a}{5}$      & 5    & 5    & $\bH\times {\bd_{2}}^{2}$
        \\
        59 & $\cv{5b}{5}$      & 5    & 5    & $\bH\times \bc_{4}$
        \\
        60 & $\cv{5c}{5}$      & 5    & 5    & $(\bH^{2})^{s}\times \bc_{3}$
        \\
        61 & $\cv{6}{5}$       & 6    & 5    & $\bH\times \bc_{5}$
        \\ \hline \rule{0pt}{3ex}
        62 & $\cv{2}{6}$       & 2    & 6    & $(\bH\cdot \bH)^{s}$
        \\
        63 & $\cv{3}{6}$       & 3    & 6    & $(\bH^{2}\cdot \bH)^{s}$
        \\
        64 & $\cv{4}{6}$       & 4    & 6    & $(\bH^{2}\cdot \bH^{2})^{s}$
        \\ \hline \rule{0pt}{3ex}
        65 & $\cv{3}{7}$       & 3    & 7    & $\bH \times (\bH^{2})^{s}$
        \\
        66 & $\cv{4}{7}$       & 4    & 7    & $\bH \times (\bH^{3})^{s}$
        \\
        67 & $\cv{5}{7}$       & 5    & 7    & $(\bH^{2})^{s} \times (\bH^{3})^{s}$
        \\ \hline \rule{0pt}{3ex}
        68 & $\cv{3}{9}$       & 3    & 9    & $(\bH\cdot \bH)^{s} \times \bH $
        \\
        69 & $\cv{4}{9}$       & 4    & 9    & $ (\bH\cdot \bH)^{s}\times (\bH^{2})^{s}$
        \\
        \hline
      \end{tabular}}
  \end{minipage}
  \caption{A minimal covariant basis for $\cov(\HH^{4})$}
  \label{tab:cov-basis-H4}
\end{table}

\section{Symmetry classes}
\label{sec:symmetry-classes}

Symmetry plays a fundamental role in the study of tensor representations. In this section, we recall the definitions of \emph{symmetry groups} and \emph{symmetry classes} of a vector $\vv$ in a finite dimensional representation $\VV$ of a compact group $G$.

\begin{defn}[Symmetry group]
  The \emph{symmetry group} of a vector $\vv \in \VV$ is defined as the subgroup
  \begin{equation*}
    G_{\vv} : =  \set{g\in G,\quad g\star \vv = \vv}.
  \end{equation*}
\end{defn}

\begin{defn}[Symmetry class]
  The \emph{symmetry class} (or isotropy class) of a vector $\vv$ is the conjugacy class of its symmetry group, where the conjugacy class $[H]$ of a subgroup $H$ is defined as
  \begin{equation*}
    [H] : =  \set{gHg^{-1},\quad g\in G}.
  \end{equation*}
\end{defn}

There is, of course, no obstruction to extend the concept of symmetry classes to a finite or infinite family of vectors belonging to different (or same) representations of $G$.

\begin{defn}
  Let $\mathcal{F}$ be a finite or infinite family of vectors belonging to different (or same) representations of $G$. We define the isotropy group of $\mathcal{F}$ as the subgroup
  \begin{equation*}
    G_{\mathcal{F}} : =  \bigcap_{\vv \in \mathcal{F}} G_{\vv}.
  \end{equation*}
  The \emph{symmetry class} of $\mathcal{F}$ is the conjugacy class of $G_{\mathcal{F}}$ in $G$.
\end{defn}

\begin{rem}
  Note that if $\mathcal{F}$ is a vector space and $(\vv_{i})_{i\in I}$ is any generating set of $\mathcal{F}$, then
  \begin{equation*}
    G_{\mathcal{F}} = \bigcap_{i\in I} G_{\vv_{i}}.
  \end{equation*}
  In particular, if $(\vv_{1}, \dotsc ,\vv_{p})$ is a basis of $\mathcal{F}$, then
  \begin{equation*}
    G_{\mathcal{F}} = \bigcap_{j = 1}^{p} G_{\vv_{j}}.
  \end{equation*}
\end{rem}

Since every symmetry group of a vector $\vv$ in $\VV$ is a closed subgroup of $G$, we are mainly interested in the closed subgroups of $G$ up to conjugacy. Now we have the following result which can be deduced from~\cite[Proposition 1.9]{Bre1972}.

\begin{lem}
  The set of conjugacy classes of a compact group $G$ is a \emph{partially ordered set} (poset) induced by inclusion, which is defined as follows:
  \begin{equation*}
    [H_{1}] \preceq [H_{2}] \quad \text{if $H_{1}$ is conjugate to a subgroup of $H_{2}$ in $G$}.
  \end{equation*}
\end{lem}

\begin{defn}\label{def:at-least-at-most}
  Since the symmetry classes of a given representation $\VV$ form a poset, we will say that a vector $\vv \in \VV$  (resp. a family $\mathcal{F}$) is \emph{at least} in a given symmetry class $[H]$, if $[H] \preceq [G_{\vv}]$	(resp. $ [H] \preceq[G_{\mathcal{F}}]$). Similarly, we will say that it is \emph{at most} in the symmetry class $[H]$, if $[G_{\vv}] \preceq [H]$ (resp. $[G_{\mathcal{F}}]\preceq [H]$).
\end{defn}

Since we are interested in representations of the rotation group $\SO(3)$, we will recall the following result~\cite{GSS1988}.

\begin{lem}\label{lem:SO3-closed-subgroups}
  Every closed subgroup of $\SO(3)$ is conjugate to one of the following list:
  \begin{equation*}
    \SO(3),\, \OO(2),\, \SO(2),\, \DD_{n} (n \ge 2),\, \ZZ_{n} (n \ge 2),\, \tetra,\, \octa,\, \ico,\, \text{and}\, \triv
  \end{equation*}
  where:
  \begin{itemize}
    \item $\OO(2)$ is the subgroup generated by all the rotations around the $z$-axis and the order 2 rotation $\sigma : (x,y,z)\mapsto (x,-y,-z)$ around the $x$-axis;
    \item $\SO(2)$ is the subgroup of all the rotations around the $z$-axis;
    \item for $n \ge 2$, $\ZZ_{n}$ is the unique cyclic subgroup of order $n$ of $\SO(2)$, the subgroup of rotations around the $z$-axis;
    \item for $n \ge 2$, $\DD_{n}$ is the \emph{dihedral} group, of order $2n$. It is generated by $\ZZ_{n}$ and $\sigma :(x,y,z)\mapsto (x,-y,-z)$;
    \item $\tetra$ is the \emph{tetrahedral} group, the orientation-preserving symmetry group of a given tetrahedron, which has order 12;
    \item $\octa$ is the \emph{octahedral} group, the orientation-preserving symmetry group of a given cube, which has order 24;
    \item $\ico$ is the \emph{icosahedral} group, the orientation-preserving symmetry group of a given dodecahedron, which has order 60;
    \item $\triv$ is the trivial subgroup, containing only the unit element.
  \end{itemize}
\end{lem}

\begin{rem}[The octahedral group]\label{rem:octahedral-group}
  The octahedral group $\octa$ is defined as the orientation-preserving symmetry group of a cube whose edges are parallel to the axes of a the canonical basis $(\ee_{1}, \ee_{2}, \ee_{3})$ of $\RR^{3}$. It corresponds to the subgroup
  \begin{equation*}
    \set{g \in \SO(3);\; g\star \ee_{i} = \pm \ee_{j}}
  \end{equation*}
  of $\SO(3)$ which contains 24 elements:
  \begin{itemize}
    \item the identity $I$;
    \item 3 order 2 rotations around the axes $\ee_{1}$, $\ee_{2}$, $\ee_{3}$;
    \item 6 order 4 rotations around the axes $\ee_{1}$, $\ee_{2}$, $\ee_{3}$;
    \item 6 order 2 rotations around the axes $\ee_{1} \pm \ee_{2}$, $\ee_{1} \pm \ee_{3}$, $\ee_{2} \pm \ee_{3}$;
    \item 8 order 3 rotations around the axis $\ee_{1} \pm \ee_{2} \pm \ee_{3}$.
  \end{itemize}
\end{rem}

It is a classical fact, that for any representation $\VV$ of a Lie group $G$, there exists only a finite number of symmetry classes~\cite{Mos1957,Man1962}. These classes have been detailed by Ihrig-Golubistky~\cite{IG1984} (see also~\cite{Oli2017}) for irreducible representations of $\SO(3)$. We get, in particular, the following posets:
\begin{enumerate}
  \item For $\HH^{1}$: $[\SO(2)] \preceq [\SO(3)]$.
  \item For $\HH^{2}$: $[\DD_{2}] \preceq [\OO(2)] \preceq [\SO(3)]$.
  \item For $\HH^{3}$: see~\autoref{fig:poset-H3}.
  \item For $\HH^{4}$: see~\autoref{fig:poset-H4} (same as for the elasticity tensor~\cite{FV1996}).
  \item For $\HH^{5}$: see~\autoref{fig:poset-H5}.
\end{enumerate}

The determination of symmetry classes for \emph{reducible representations} of $\SO(3)$ has been achieved by Olive~\cite{Oli2017} (see also~\cite{CLM1991}), who formulated an algorithm to compute theses classes, provided a decomposition into irreducible representations is known. Using these results and the fact that
\begin{equation*}
  \Sym^{n}(\RR^{3})\simeq \HH^{n} \oplus \HH^{n-2} \oplus \dotsb \oplus \HH^{n-2r},\quad r=[n/2],
\end{equation*}
by~\eqref{eq:symmetric-harmonic-decomposition}, we deduce the following proposition.

\begin{prop}\label{prop:symmetry-classes}
  We have the following results.
  \begin{enumerate}
    \item The symmetry classes for $n$ ($n \ge 2$) first-order tensors are
          \begin{equation*}
            \set{[\triv], [\SO(2)], [\SO(3)]}.
          \end{equation*}
    \item The symmetry classes for $n$ ($n \ge 2$) second-order symmetric tensors are
          \begin{equation*}
            \set{[\triv], [\ZZ_{2}], [\DD_{2}], [\OO(2)], [\SO(3)]}.
          \end{equation*}
    \item The symmetry classes for one third-order totally symmetric tensor are (see Erratum~\cite{OKDD2022})
          \begin{equation*}
            \set{[\triv], [\ZZ_{2}], [\ZZ_{3}], [\DD_{3}], [\tetra], [\SO(2)], [\SO(3)]}.
          \end{equation*}
    \item The symmetry classes for one fourth-order totally symmetric tensor are (like for the elasticity tensor)
          \begin{equation*}
            \set{[\triv], [\ZZ_{2}], [\DD_{2}], [\DD_{3}], [\DD_{4}], [\octa], [\OO(2)], [\SO(3)]}.
          \end{equation*}
    \item The symmetry classes for one fifth-order totally symmetric tensor are
          \begin{equation*}
            \set{[\triv], [\ZZ_{2}], [\ZZ_{3}], [\ZZ_{4}], [\ZZ_{5}], [\DD_{2}], [\DD_{3}], [\DD_{4}], [\DD_{5}], [\tetra], [\SO(2)], [\SO(3)]}.
          \end{equation*}
  \end{enumerate}
\end{prop}

\begin{rem}\label{rem:vanishing-isotropic-symmetric-odd-order-tensors}
  The harmonic decomposition of a totally symmetric tensor~\eqref{eq:polynomial-harmonic-decomposition} of odd order contains only factors $\HH^{k}$ with $k$ odd. Moreover, an isotropic tensor in $\HH^{^k}$ vanishes necessarily if $k \ge 1$ odd. Thus any totally symmetric isotropic tensor of \emph{odd order} vanishes. This is however not true for an even order totally symmetric isotropic tensor.
\end{rem}

\begin{figure}[ht]
  \includegraphics[scale = 1]{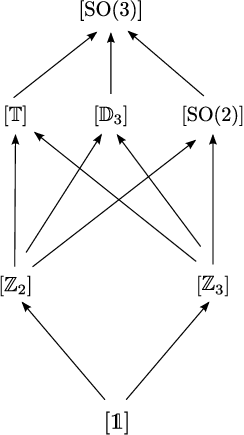}
  \caption{The poset of symmetry classes for $\HH^{3}$~\cite{OKDD2022}.}
  \label{fig:poset-H3}
\end{figure}

\begin{figure}[ht]
  \includegraphics[scale = 1]{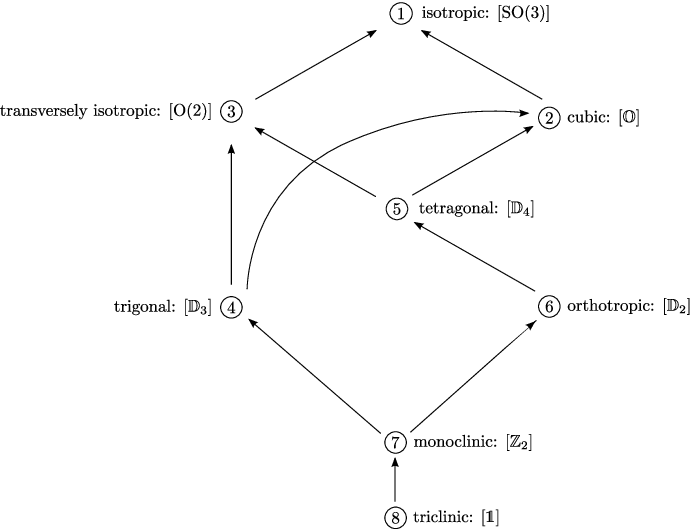}
  \caption{The poset of symmetry classes for $\HH^{4}$ and $\Ela$.}
  \label{fig:poset-H4}
\end{figure}

\begin{figure}[ht]
  \includegraphics[scale = 1]{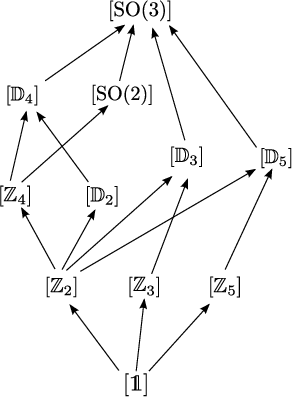}
  \caption{The poset of symmetry classes for $\HH^{5}$.}
  \label{fig:poset-H5}
\end{figure}

\section{Dimension of covariant spaces and symmetry}
\label{sec:covariants-symmetry-dimension}

Given a linear representation $\VV$ of $\SO(3)$ and $\vv \in \VV$, we define $\cov_{k}(\vv)$ as the set of all $k$-order polynomial covariants of $\vv$ (see~\autoref{sec:polynomial-covariants}). Note that whereas $\cov(\VV)$ is a polynomial algebra, and $\cov_{k}(\VV)$ is an infinite dimensional vector space, $\cov_{k}(\vv)$ is the set of all \emph{evaluations} of these covariants on the vector $\vv$. As such, it is a subspace of the finite dimensional real vector space $\Pn{k}(\RR^{3})$ of homogeneous polynomials of degree $k$ on $\RR^{3}$, or equivalently of the space $\Sym^{k}(\RR^{3})$ of totally symmetric tensors of order $k$. In this section, we will focus on polynomial covariants of order one and two of a vector $\vv \in \VV$ and relate the symmetry class of $\cov_{1}(\vv)$ and $\cov_{2}(\vv)$ with their respective dimension.

Recall that, thanks to proposition~\ref{prop:symmetry-classes}, the possible symmetry classes for the space $\cov_{1}(\vv)$ are
\begin{equation*}
  [\triv], \quad [\SO(2)], \quad [\SO(3)],
\end{equation*}
whereas, for $\cov_{2}(\vv)$, they are
\begin{equation*}
  [\triv], \quad [\ZZ_{2}], \quad [\DD_{2}], \quad [\OO(2)], \quad [\SO(3)].
\end{equation*}

\begin{prop}\label{prop:cov1-symmetry-classes}
  Given $\vv \in \VV$, $\dim \cov_{1}(\vv)$ is either $0$, $1$ or $3$. Moreover, the symmetry class of $\cov_{1}(\vv)$ is:
  \begin{enumerate}
    \item $[\SO(3)]$ if and only if $\cov_{1}(\vv) = \set{0}$;
    \item $[\SO(2)]$ if and only if $\dim \cov_{1}(\vv) = 1$;
    \item $[\triv]$ if and only if $\dim \cov_{1}(\vv) = 3$.
  \end{enumerate}
\end{prop}

\begin{proof}
  (1) If the symmetry class of $\cov_{1}(\vv)$ is $[\SO(3)]$, then, every first-order covariant vanishes and thus $\cov_{1}(\vv) = \set{0}$. Conversely, if $\cov_{1}(\vv) = 0$ then its symmetry class is $[\SO(3)]$.

  (2) Suppose now that the symmetry class of $\cov_{1}(\vv)$ is $[\SO(2)]$. Without loss of generality, we can suppose that the isotropy group of  $\cov_{1}(\vv)$ is exactly $\SO(2)$. Then, $\dim \cov_{1}(\vv) \ge 1$ but all first-order covariant are colinear to $\ee_{3}$ and thus $\dim \cov_{1}(\vv) = 1$. Conversely, suppose that $\dim \cov_{1}(\vv) = 1$ and let $\uu \ne 0$ be a basis of $\cov_{1}(\vv)$. Then the symmetry class of $\cov_{1}(\vv)$ is just $[G_{\uu}] = [\SO(2)]$.

  (3) Finally, suppose that the symmetry class of $\cov_{1}(\vv)$ is $[\triv]$. Then $\dim \cov_{1}(\vv) \ge 2$. But if $\uu, \ww$ are two independent first-order covariants then $\uu \times \ww$ is also a first-order covariant, so that $\dim \cov_{1}(\vv)  = 3$. Conversely, if $\dim \cov_{1}(\vv)  = 3$, we can find two independent covariants $\uu, \ww$ and thus
  \begin{equation*}
    G_{\uu} \cap G_{\ww} = \triv.
  \end{equation*}
\end{proof}

The case of $\cov_{2}(\vv)$ is more involving. Note first that the Euclidean second-order tensor $\bq$ is always in $\cov_{2}(\vv)$, thus $\dim \cov_{2}(\vv) \ge 1$ for every $\vv \in \VV$. Moreover, given two covariants $\ba,\bb$ in $\cov_{2}(\vv)$, then
\begin{equation*}
  (\ba\bb)^{s} : =  \frac{1}{2}(\ba\bb + \bb\ba)
\end{equation*}
belongs to $\cov_{2}(\vv)$, where $\ba\bb$ is the standard matrix product.

\begin{lem}\label{lem:q-a-a2}
  Let $\ba \in \Sym^{2}(\RR^{3})$. Then,
  \begin{enumerate}
    \item $\ba$ is orthotropic if and only if $\dim \langle \bq,\ba,\ba^{2} \rangle = 3$;
    \item $\ba$ is transversely isotropic if and only if $\dim \langle \bq,\ba,\ba^{2} \rangle = 2$.
  \end{enumerate}
\end{lem}

\begin{proof}
  Without loss of generality, we can suppose that $\ba = \mathrm{diag}(\lambda_{1}, \lambda_{2}, \lambda_{3})$. Then, it is easy to check that $\bq$, $\ba$, $\ba^{2}$ are linearly independent if and only if
  \begin{equation*}
    (\lambda_{2}-\lambda_{1})(\lambda_{3}-\lambda_{1})(\lambda_{3}-\lambda_{2}) \ne 0.
  \end{equation*}
  Thus, we get (1). Moreover, if $\ba$ is transversely isotropic and has thus a double eigenvalue then $\dim \langle \bq,\ba,\ba^{2} \rangle \le 2$ but it cannot be one, otherwise, $\ba$ would be isotropic. Conversely if $\dim \langle \bq,\ba,\ba^{2} \rangle = 2$, then $\ba$ has a double eigenvalue and is hence at least transversely isotropic but it cannot be isotropic (otherwise $\dim \langle \bq,\ba,\ba^{2} \rangle = 1$). This achieves the proof.
\end{proof}

Recall that a pair $(\ba,\bb)$ of symmetric second-order tensors is either isotropic, transversely isotropic, orthotropic, monoclinic or triclinic by proposition~\ref{prop:symmetry-classes}. We have, moreover, the following result.

\begin{lem}\label{lem:orthotropic-linear-combination}
  Let $(\ba,\bb)$ be a pair of symmetric second-order tensors, which is either orthotropic, monoclinic or triclinic. Then, there exists a linear combination of $\ba$ and $\bb$ which is orthotropic.
\end{lem}

\begin{proof}
  If either $\ba$ or $\bb$ is orthotropic, we are done. Otherwise, both $\ba$ and $\bb$ are transversely isotropic, neither being isotropic. Let $\ba^{\prime}$ and $\bb^{\prime}$ be the deviatoric parts of $\ba$ and $\bb$ respectively. Note that $\ba^{\prime}$, $\bb^{\prime}$ are linearly independent and both transversely isotropic. If we can show that there exists a linear combination $\alpha \ba^{\prime} + \beta \bb^{\prime}$ which is orthotropic, then, we are done, because
  \begin{equation*}
    \alpha \ba + \beta \bb = \alpha \ba^{\prime} + \beta \bb^{\prime} + \frac{1}{3}\left(\alpha \tr \ba + \beta \tr \bb\right)\bq
  \end{equation*}
  is orthotropic. Let
  \begin{equation*}
    \tilde{\ba} = \ba^{\prime} - \frac{\tr(\ba^{\prime}\bb^{\prime})}{\tr ({\bb^{\prime}}^{2})}\bb^{\prime}, \qquad  \tilde{\bb} = \bb^{\prime}
  \end{equation*}
  If $\tilde{\ba}$ is orthotropic, we are done. Otherwise $\tilde{\ba}$, $\tilde{\bb}$ are two linearly independent, transversely isotropic deviators such that $\tr(\tilde{\ba}\tilde{\bb}) = 0$. Now, the discriminant of the characteristic polynomial of a deviatoric tensor $\bd$ writes as
  \begin{equation*}
    (\tr(\bd^{2}))^{3}/2 - 3(\tr(\bd^{3}))^{2}.
  \end{equation*}
  Hence a deviatoric tensor $\bd$ is orthotropic if and only if
  \begin{equation*}
    (\tr(\bd^{2}))^{3} - 6(\tr(\bd^{3}))^{2} \ne 0.
  \end{equation*}
  Let $\bd(t): = t \tilde{\ba} + (1-t)\tilde{\bb}$. Then $\bd(t)$ is orthotropic if and only if
  \begin{equation*}
    p(t): = (\tr(\bd(t)^{2}))^{3}-6(\tr(\bd(t)^{3}))^{2} \ne 0.
  \end{equation*}
  Moreover, a direct computation shows that the coefficient of $t^{2}$ in the polynomial $p(t)$ is
  \begin{equation*}
    3\tr({\tilde{\ba}}^{2})\tr({\tilde{\bb}}^{2})^{2}\ne 0.
  \end{equation*}
  Hence, there exists $t \in \RR$ such that $p(t)\ne 0$ and for this value, $\bd(t)$ is orthotropic. We have thus found a linear combination of $\tilde{\ba}$, $\tilde{\bb}$, and therefore of $\ba^{\prime}$, $\bb^{\prime}$ which is orthotropic. This achieves the proof.
\end{proof}

\begin{cor}\label{cor:dimension-trans-iso-subspaces}
  Let $F$ be a sub-vector space of $\Sym^{2}(\RR^{3})$ with $\dim F \ge 3$. Then, $F$ contains an orthotropic element.
\end{cor}

\begin{proof}
  Suppose that each element in $F$ is at least transversely isotropic. Then, by lemma~\ref{lem:orthotropic-linear-combination}, each pair $(\ba,\bb)$ of elements in $F$ is at least transversely isotropic. If each element in $F$ is isotropic, then $F$ is of dimension 0 or 1. If $F$ contains a transversely isotropic element $\bt$, then for every $\ba \in F$, the pair $(\bt,\ba)$ is transversely isotropic and thus $\ba = \alpha \bt + \beta \bq$. Thus $\dim F \le 2$. This achieves the proof.
\end{proof}

Given an orthonormal basis $(\ee_{1},\ee_{2},\ee_{3})$ of $\RR^{3}$, we will consider the following natural basis of $\Sym^{2}(\RR^{3})$
\begin{equation}\label{eq:S2-orthogonal-basis}
  \be_{ii}: = \ee_{i}\otimes \ee_{i}, \qquad \be_{ij}: = \ee_{i}\otimes \ee_{j} + \ee_{j}\otimes \ee_{i},\quad (i<j),
\end{equation}
which is orthogonal but not orthonormal.

\begin{prop}\label{prop:cov2-symmetry-classes}
  Given $\vv \in \VV$, $\dim \cov_{2}(\vv)$ is either $1$, $2$, $3$, $4$ or $6$. Moreover, the symmetry class of $\cov_{2}(\vv)$ is:
  \begin{enumerate}
    \item $[\SO(3)]$ if and only if $\dim \cov_{2}(\vv) = 1$;
    \item $[\OO(2)]$ if and only if $\dim \cov_{2}(\vv) = 2$;
    \item $[\DD_{2}]$ if and only if $\dim \cov_{2}(\vv) = 3$;
    \item $[\ZZ_{2}]$ if and only if $\dim \cov_{2}(\vv) = 4$;
    \item $[\triv]$ if and only if $\dim \cov_{2}(\vv) = 6$.
  \end{enumerate}
\end{prop}

\begin{proof}
  (1) If the symmetry class of $\cov_{2}(\vv)$ is $[\SO(3)]$, then, every symmetric second-order covariant is proportional to $\bq$ and hence $\dim \cov_{2}(\vv) = 1$. Conversely, if $\dim \cov_{2}(\vv) = 1$ then $\bq$ generates $\cov_{2}(\vv)$, and its symmetry class is $[\SO(3)]$.

  (2) Suppose that the symmetry class of $\cov_{2}(\vv)$ is $[\OO(2)]$. Then, without loss of generality, we can suppose that each symmetric second-order covariant writes as $\mathrm{diag}(\lambda,\lambda,\mu)$ and hence that $\dim \cov_{2}(\vv) \le 2$. Since it cannot be $1$, otherwise $\cov_{2}(\vv)$ would be reduced the one-dimensional space generated by $\bq$, it must be $2$. Conversely, if $\dim \cov_{2}(\vv) = 2$, then, there exists some non--isotropic second-order covariant $\ba$ such that $(\bq,\ba)$ is basis of $\cov_{2}(\vv)$. Since $\ba$ cannot be orthotropic, otherwise $(\bq,\ba,\ba^{2})$ would be linearly independent, by lemma~\ref{lem:q-a-a2}, $\ba$ is necessarily transversely isotropic and so is $\cov_{2}(\vv)$.

  (3) Suppose that the symmetry class of $\cov_{2}(\vv)$ is $[\DD_{2}]$. Then without loss of generality we can assume that each symmetric second-order covariant writes as $\mathrm{diag}(\lambda_{1},\lambda_{2},\lambda_{3})$ and hence that $\dim \cov_{2}(\vv) \le 3$. Since this dimension cannot be $1$, neither $2$ due to points (1) and (2), its must be $3$. Conversely, suppose that $\dim \cov_{2}(\vv) = 3$. Then, by corollary~\ref{cor:dimension-trans-iso-subspaces}, $\cov_{2}(\vv)$ contains an orthotropic tensor $\bc$, and we are done by lemma~\ref{lem:q-a-a2} because the orthotropic triplet $(\bq,\bc,\bc^{2})$ is a basis of $\cov_{2}(\vv)$.

  (4) Suppose that the symmetry class of $\cov_{2}(\vv)$ is $[\ZZ_{2}]$. Then without loss of generality we can suppose that each symmetric second-order covariant writes as
  \begin{equation*}
    \left(
    \begin{array}{ccc}
      a_{11} & a_{12} & 0      \\
      a_{12} & a_{22} & 0      \\
      0      & 0      & a_{33} \\
    \end{array}
    \right)
  \end{equation*}
  and hence that $\dim \cov_{2}(\vv) \le 4$. Since this dimension is necessarily $>3$ by (1), (2) and (3), it is $4$. Conversely, suppose that $\dim \cov_{2}(\vv) = 4$. Then by corollary~\ref{cor:dimension-trans-iso-subspaces}, there exists an orthotropic covariant $\bc$ in $\cov_{2}(\vv)$ and without loss of generality, we can suppose that this covariant is diagonal. Then, by lemma~\ref{lem:q-a-a2}, $\langle \bq, \bc, \bc^{2} \rangle$ is a vector basis of the space of diagonal tensors, so that $\langle \bq, \bc, \bc^{2} \rangle=\langle \be_{11}, \be_{22}, \be_{33} \rangle$, and thus each $\be_{ii}$ belongs to $\cov_{2}(\vv)$. Let $\ba$ be a second-order covariant such that $(\be_{11},\be_{22},\be_{33},\ba)$ is a basis of $\cov_{2}(\vv)$. Without loss of generality, we can assume that
  \begin{equation*}
    \ba =  a_{12}\be_{12} + a_{13}\be_{13} + a_{23}\be_{23},
  \end{equation*}
  where the $a_{ij}$ do not vanish altogether, for instance $a_{12}\neq 0$. Then
  \begin{equation*}
    (\be_{11}\ba)^{s} + (\be_{22}\ba)^{s} - (\be_{33}\ba)^{s} = a_{12}\be_{12}
  \end{equation*}
  belongs to $\cov_{2}(\vv)$ and so does $\be_{12}$. Hence
  \begin{equation*}
    (\be_{11},\be_{22},\be_{33},\be_{12})
  \end{equation*}
  is a basis of $\cov_{2}(\vv)$ which has therefore the symmetry $[\ZZ_{2}]$.

  (5) Suppose that the symmetry class of $\cov_{2}(\vv)$ is $[\triv]$. Then
  \begin{equation*}
    \dim \cov_{2}(\vv) \ge 5,
  \end{equation*}
  by (1), (2), (3) and (4). By corollary~\ref{cor:dimension-trans-iso-subspaces}, there exists an orthotropic covariant $\bc$ in $\cov_{2}(\vv)$, and like in the proof of (4), we can assume that
  \begin{equation*}
    \be_{11},\be_{22},\be_{33} \in\cov_{2}(\vv).
  \end{equation*}
  Since, $\dim \cov_{2}(\vv) \ge 5$, the space $\cov_{2}(\vv)$ contains two linearly independent covariants, which write
  \begin{eqnarray*}
    \ba & = & a_{12}\be_{12} + a_{13}\be_{13} + a_{23}\be_{23}, \\
    \bb & = & b_{12}\be_{12} + b_{13}\be_{13} + b_{23}\be_{23},
  \end{eqnarray*}
  and we can assume (without loss of generality) that the minor
  \begin{equation*}
    a_{12}b_{13} - a_{13}b_{12} \ne 0.
  \end{equation*}
  As in the proof of (4), we conclude then that both $\be_{12}$ and $\be_{13}$ belong to $\cov_{2}(\vv)$. But then
  \begin{equation*}
    (\be_{12}\be_{13})^{s} = \frac{1}{2}\be_{23}
  \end{equation*}
  belongs to $\cov_{2}(\vv)$ and thus $\dim \cov_{2}(\vv)=6$. Conversely, suppose that $\dim \cov_{2}(\vv) = 6$. Then the only possibility is that the symmetry class of $\cov_{2}(\vv)$ is $[\triv]$ by (1), (2), (3) and (4). This achieves the proof.
\end{proof}

\section{Covariant criteria for tensor's symmetry}
\label{sec:covariant-criteria}

In this section, we formulate covariant criteria which restrict the symmetry class of second and fourth order tensors, using the vanishing of some of their covariants.

\subsection{Second order tensors}
\label{subsec:second-order-tensors}

The main results of this section are formulations of coordinate-free polynomial conditions to classify the symmetry class of an $n$-tuple of symmetric second-order tensors.

\begin{lem}\label{lem:axa2=0}
  Let $\ba$ be a symmetric second-order tensor. Then, $\ba$ is at least transversely isotropic if and only if $\ba \times \ba^{2} = 0$.
\end{lem}

\begin{proof}
  Without loss of generality, we can assume that
  \begin{equation*}
    \ba = \mathrm{diag}(\lambda_{1},\lambda_{2},\lambda_{3}).
  \end{equation*}
  Then, the polynomial form of $\ba \times \ba^{2}$, writes
  \begin{equation*}
    \left( \lambda_{2} - \lambda_{1} \right) \left( \lambda_{3} - \lambda_{1}\right) \left( \lambda_{3} - \lambda_{2} \right) xyz.
  \end{equation*}
  Thus, it vanishes if and only if $\ba$ is at least transversely isotropic. This achieves the proof.
\end{proof}

\begin{rem}\label{rem:tetrahedral-symmetry}
  If $\ba$ is not orthotropic, then $\ba \times \ba^{2}$ vanishes and is thus isotropic. Otherwise, the polynomial form of $\ba \times \ba^{2}$ writes $k\, xyz$, with $k= \left( \lambda_{2} - \lambda_{1} \right) \left( \lambda_{3} - \lambda_{1}\right) \left( \lambda_{3} - \lambda_{2} \right) \ne 0$. This form is, of course, invariant by $\DD_{2}$ but it has more symmetries. Indeed, it is invariant by the \emph{fourth-order} rotations around $Ox$, $Oy$ and $Oz$. Considering the symmetry classes of $\Sym^{3}(\RR^{3})$ (proposition~\ref{prop:symmetry-classes}), and \autoref{fig:poset-H3}, we conclude that $\ba \times \ba^{2}$ has \emph{tetrahedral} symmetry $[\tetra]$. An immediate corollary of this result is that $\tr(\ba \times \ba^{2}) = 0$ and thus that $\ba \times \ba^{2}$ is harmonic, independently of the symmetry of $\ba$.
\end{rem}

\begin{lem}\label{lem:axb=0}
  Let $\ba,\bb$ be symmetric second-order tensors and suppose that $\ba$ is transversely isotropic. Then, $(\ba,\bb)$ is transversely isotropic if and only if $\ba \times \bb = 0$.
\end{lem}

\begin{proof}
  Suppose first that $(\ba,\bb)$ is transversely isotropic then $\ba \times \bb$ is at least transversely isotropic and since it is a third-order totally symmetric tensor, it must be isotropic by proposition~\ref{prop:symmetry-classes} and thus vanishes by Remark~\ref{rem:vanishing-isotropic-symmetric-odd-order-tensors}. To prove the converse, we will use the polynomial representative $\ra,\rb$ of $\ba,\bb$ (see Section~\ref{sec:sym-harm-tensors}). The linear equation $\ba \times \bb = 0$ reads then $\det (\xx, \nabla \ra, \nabla \rb) = 0$. Without loss of generality we can assume that $G_{\ba} = \OO(2)$ and thus that
  \begin{equation*}
    \ra =  \lambda (x^{2} + y^{2}) + \mu z^{2}, \qquad \lambda \ne \mu .
  \end{equation*}
  The solution is then
  \begin{equation*}
    \rb = k_{1} (x^{^{2}} + y^{2}) + k_{2} z^{2},
  \end{equation*}
  which is invariant by $\OO(2)$. This achieves the proof.
\end{proof}

Given two symmetric second order tensors $\ba,\bb$ on the euclidean space $\RR^{3}$, their commutator, a second-order skew-symmetric tensor
\begin{equation*}
  [\ba,\bb]:=\ba\bb-\bb\ba
\end{equation*}
can be recast as the first-order covariant
\begin{equation*}
  \tr(\ba \times \bb) = \frac{1}{3} \, \pmb \varepsilon:(\ba \bb).
\end{equation*}

We have thus

\begin{lem}\label{lem:orthotropic-pair}
  The three conditions are equivalent :
  \begin{enumerate}
    \item the pair $(\ba, \bb)$ is \emph{at least orthotropic}.
    \item $\tr(\ba \times \bb) = 0$.
    \item $\ba, \bb$ commute.
  \end{enumerate}
\end{lem}

\begin{cor}\label{cor:orthotropic-pair}
  Let $\ba,\bb$ be symmetric second-order tensors. Then, $(\ba,\bb)$ is orthotropic if and only if $\tr(\ba \times \bb) = 0$ and
  \begin{equation*}
    \ba \times \ba^{2} \ne 0, \quad \text{or} \quad \bb \times \bb^{2} \ne 0, \quad \text{or} \quad \ba \times \bb \ne 0.
  \end{equation*}
\end{cor}

\begin{proof}
  If $(\ba,\bb)$ is orthotropic, then the first-order covariant $\tr(\ba \times \bb)$ is necessarily isotropic by proposition~\ref{prop:symmetry-classes} and thus vanishes by Remark~\ref{rem:vanishing-isotropic-symmetric-odd-order-tensors}. Moreover, either $\ba$ or $\bb$ is orthotropic and thus
  \begin{equation*}
    \ba \times \ba^{2} \ne 0, \qquad \text{or} \qquad \bb \times \bb^{2} \ne 0,
  \end{equation*}
  or both of them are transversely isotropic. In that case we necessarily have $\ba \times \bb \ne 0$ by lemma~\ref{lem:axb=0}. Conversely, if $\tr(\ba \times \bb) = 0$, then the pair $(\ba,\bb)$ is at least orthotropic by lemma~\ref{lem:orthotropic-pair}. If either $\ba$ or $\bb$ is orthotropic, then so is $(\ba,\bb)$. Otherwise, both $\ba$ and $\bb$ are at least transversely isotropic, but then the condition $\ba \times \bb \ne 0$ forbids the pair $(\ba,\bb)$ to be at least transversely isotropic. It is thus orthotropic.
\end{proof}

We will now formulate coordinate-free conditions to classify the symmetry class of an $n$-tuple of symmetric second-order tensors.

\begin{thm}\label{thm:n-quadratic-forms}
  Let $(\ba_{1}, \dotsc ,\ba_{n})$ be an $n$-tuple of second-order symmetric tensors. Then:
  \begin{enumerate}
    \item $(\ba_{1}, \dotsc ,\ba_{n})$ is \emph{isotropic} if and only if
          \begin{equation*}
            \ba_{k}^{\prime} = 0, \quad 1 \le k \le n ,
          \end{equation*}
          where $\ba_{k}^{\prime}$ is the deviatoric part of $\ba_{k}$.
    \item $(\ba_{1}, \dotsc ,\ba_{n})$ is \emph{transversely isotropic} if and only if there exists $\ba_{j}$ such that
          \begin{equation*}
            \ba_{j}^{\prime} \ne 0, \qquad \ba_{j} \times \ba_{j}^{2} = 0,
          \end{equation*}
          and
          \begin{equation*}
            \ba_{j} \times \ba_{k} = 0, \quad 1 \le k \le n .
          \end{equation*}
    \item $(\ba_{1}, \dotsc ,\ba_{n})$ is \emph{orthotropic} if and only if
          \begin{equation*}
            \tr(\ba_{k} \times \ba_{l}) = 0, \quad 1 \le k,l \le n ,
          \end{equation*}
          and there exists $\ba_{j}$ such that $\ba_{j} \times \ba_{j}^{2} \ne 0$ or there exists a pair $(\ba_{i},\ba_{j})$ such that $\ba_{i} \times \ba_{j} \ne 0$.

    \item $(\ba_{1}, \dotsc ,\ba_{n})$ is \emph{monoclinic} if and only if there exists a pair $(\ba_{i},\ba_{j})$ such that $\pmb{\omega} := \tr(\ba_{i} \times \ba_{j}) \ne 0$ and
          \begin{equation*}
            (\ba_{k}\pmb{\omega}) \times \pmb{\omega} = 0, \quad 1 \le k \le n .
          \end{equation*}
  \end{enumerate}
\end{thm}

\begin{proof}
  (1) $(\ba_{1}, \dotsc ,\ba_{n})$ is \emph{isotropic} if and only if $\ba_{k} =\lambda_{k}\bq$ for $1 \le k \le n$, which is equivalent to the condition that $\ba_{k}^{\prime} = 0$ for $1 \le k \le n$.

  (2) If $(\ba_{1}, \dotsc ,\ba_{n})$ is \emph{transversely isotropic}, then, each $\ba_{k}$ is at least transversely isotropic and one of them, say $\ba_{j}$, is transversely isotropic. Thus $\ba_{j}^{\prime} \ne 0$ and $\ba_{j} \times \ba_{j}^{2} = 0$ by lemma~\ref{lem:axa2=0}. Moreover, each pair $(\ba_{j}, \ba_{k})$ is at least transversely isotropic and thus $\ba_{j} \times \ba_{k} = 0$ by lemma~\ref{lem:axb=0}. Conversely, if conditions in $(2)$ are satisfied, then $\ba_{j}$ is transversely isotropic and each pair $(\ba_{j}, \ba_{k})$ is transversely isotropic by lemma~\ref{lem:axb=0}. Thus $(\ba_{1}, \dotsc ,\ba_{n})$ is transversely isotropic.

  (3) If $(\ba_{1}, \dotsc ,\ba_{n})$ is \emph{orthotropic}, then, the $\ba_{k}$ commute with each other and thus $\tr(\ba_{k} \times \ba_{l}) = 0$ ($1 \le k,l \le n$) by lemma~\ref{lem:orthotropic-pair}. Moreover, either there exists $j \in \set{1,\dotsc ,n}$ such that $\ba_{j}$ is orthotropic and thus $\ba_{j} \times \ba_{j}^{2} \ne 0$ or all the $\ba_{k}$ are at least transversely isotropic. In that case, a pair of them, say $(\ba_{i},\ba_{j})$ is orthotropic and thus $\ba_{i} \times \ba_{j} \ne 0$. Conversely, if $\tr(\ba_{k} \times \ba_{l}) = 0$ for all $k,l$, then, we can find a basis in which there are all diagonal and the symmetry class of $(\ba_{1}, \dotsc ,\ba_{n})$ is thus at least $[\DD_{2}]$. If there exists $\ba_{j}$ such that $\ba_{j} \times \ba_{j}^{2} \ne 0$, we are done. Otherwise, all the $\ba_{k}$ are at least transversely isotropic, but there exists a pair $(\ba_{i},\ba_{j})$ such that $\ba_{i} \times \ba_{j} \ne 0$. Hence, both $\ba_{i},\ba_{j}$ are transversely isotropic and the pair $(\ba_{i},\ba_{j})$ is orthotropic by lemma~\ref{lem:axb=0}.

  (4) If $(\ba_{1}, \dotsc ,\ba_{n})$ is \emph{monoclinic}, then, its elements have a common eigenvector, $\pmb{\omega}$, so that $(\ba_{k}\pmb{\omega}) \times \pmb{\omega} = 0$ ($1 \le k \le n$). Moreover, there exists a pair $(\ba_{i},\ba_{j})$ such that $\tr(\ba_{i} \times \ba_{j}) \ne 0$ and thus $\tr(\ba_{i} \times \ba_{j}) = \lambda \pmb{\omega}$ with $\lambda \ne 0$. Conversely, if $\pmb{\omega} := \tr(\ba_{i} \times \ba_{j}) \ne 0$, then $(\ba_{1}, \dotsc ,\ba_{n})$ is at most monoclinic. But the condition
  $(\ba_{k}\pmb{\omega}) \times \pmb{\omega} = 0$ for all $k$ means that $\pmb{\omega}$ is a common eigenvector of $\ba_{1}, \dotsc ,\ba_{n}$ and thus the symmetry group of $(\ba_{1}, \dotsc ,\ba_{n})$ contains the second-order rotation around $\pmb{\omega}$.
\end{proof}

\subsection{Fourth order tensors}
\label{subsec:fourth-order-tensors}

The main results of this section are formulations of coordinate-free polynomial conditions to classify the symmetry class of a family of one fourth-order harmonic tensor and one or several second-order harmonic tensors.

\begin{lem}\label{lem:Sxd=0}
  Let $\bt\in \Sym^{2}(\RR^{3})$ be transversely isotropic and $\bS\in \Sym^{4}(\RR^{3})$. Then, $(\bS,\bt)$ is transversely isotropic if and only if $\bS \times \bt = 0$.
\end{lem}

\begin{proof}
  Suppose first that $(\bS,\bt)$ is transversely isotropic, then $\bS \times \bt$ is at least transversely isotropic and since it is a fifth-order symmetric tensor, it must be isotropic by proposition~\ref{prop:symmetry-classes} and thus vanishes by Remark~\ref{rem:vanishing-isotropic-symmetric-odd-order-tensors}. To prove the converse, let $\rp, \rt$ be the polynomial representatives of $\bS,\bt$. Then, the linear equation $\bS \times \bt = 0$ reads $\det (\xx, \nabla \rp, \nabla \rt) = 0$. Without loss of generality we can assume that $G_{\bt} = \OO(2)$ and thus that
  \begin{equation*}
    \rt = \lambda (x^{2} + y^{2}) + \mu z^{2}, \qquad \lambda \ne \mu
  \end{equation*}
  and the solution is
  \begin{equation*}
    \rp = k_{1}z^{4} + k_{2} (x^{^{2}} + y^{2})z^{2} + k_{3}(x^{2}+y^{2})^{2},
  \end{equation*}
  which is invariant by $\OO(2)$. This achieves the proof.
\end{proof}

\begin{lem}\label{lem:tr(Hxd)=0}
  Let $\bt\in \Sym^{2}(\RR^{3})$ be transversely isotropic and $\bH\in \HH^{4}$. Then, $(\bH,\bt)$ is at least tetragonal if and only if $\tr (\bH \times \bt) = 0$.
\end{lem}

\begin{proof}
  Suppose first that $(\bH,\bt)$ is at least tetragonal, then $\tr (\bH \times \bt)$ is at least tetragonal and since it is a third-order symmetric tensor, it must be isotropic by proposition~\ref{prop:symmetry-classes} and thus vanishes by Remark~\ref{rem:vanishing-isotropic-symmetric-odd-order-tensors}. To prove the converse, let $\rp, \rt$ be the polynomial representatives of $\bH,\bt$. Then, the linear equation $\tr (\bH \times \bt) = 0$ reads $\triangle (\det (\xx, \nabla \rp, \nabla \rt)) = 0$, where $\triangle$ is the Laplacian. Without loss of generality we can assume that $G_{\bt} = \OO(2)$ and thus that
  \begin{equation*}
    \rt =  \lambda (x^{2} + y^{2}) + \mu z^{2}, \qquad \lambda \ne \mu
  \end{equation*}
  and the solution is
  \begin{equation*}
    \rp = k_{1} \left(6z^{2}(x^{2} + y^{2}) - (x^{4} + y^{4}) + 2z^{4}\right) + k_{2} \left(6x^{2}y^{2} - (x^{4} + y^{4})\right) + k_{3} xy\left(x^{2} - y^{2}\right),
  \end{equation*}
  which is invariant by $\ZZ_{4}$ and has thus at least the symmetry $[\DD_{4}]$. Hence, $(\bH,\bt)$ is at least tetragonal. This achieves the proof.
\end{proof}

The cubic symmetry appears, in practice, as the more subtle to deal with. We will formulate, in the next lemma, more precise statements which allow to detect the symmetry class of a pair $(\bH,\bt)$ when $\bH$ is cubic and $\bt$ is transversely isotropic. In that case, we know from~\cite{Oli2017} that the symmetry class of a pair $(\bH,\bt)$ is one of the following : triclinic, monoclinic, orthotropic, trigonal or tetragonal.

\begin{lem}\label{lem:cube-orientation}
  Let $\bH$ be a cubic fourth-order harmonic tensor and $\bt\in\Sym^{2}(\RR^{2})$ be transversely isotropic. Then
  \begin{enumerate}
    \item $(\bH,\bt)$ is tetragonal if and only if
          \begin{equation*}
            \tr(\bH \times \bt) = 0;
          \end{equation*}
    \item $(\bH,\bt)$ is trigonal if and only if
          \begin{equation*}
            \tr(\bH \times \bt) \ne 0, \quad \text{and} \quad \bt \times (\bH\2dots\bt)=0;
          \end{equation*}
    \item $(\bH,\bt)$ is orthotropic if and only if
          \begin{equation*}
            \bt \times (\bH\2dots\bt) \ne 0, \quad \text{and} \quad \tr \left(\bt \times (\bH\2dots\bt)\right)=0;
          \end{equation*}
    \item $(\bH,\bt)$ is monoclinic if and only if
          \begin{equation*}
            \tr(\bt \times (\bH\2dots\bt)) \ne 0, \quad \text{and} \quad \tr(\bt \times (\bH\2dots\bt)) \times \tr(\bt \times (\bH\2dots\bt)^{2}) =0.
          \end{equation*}
  \end{enumerate}
\end{lem}

\begin{rem}\label{rem:cub-orientation}
  The conditions in $(3)$ are equivalent to the fact that the pair $(\bt,\bH\2dots\bt)$ is orthotropic, whereas in $(4)$ they are equivalent to the fact that the pair $(\bt,\bH\2dots\bt)$ is monoclinic. The cases $(1)$ and $(2)$ cover all the cases where the pair $(\bt,\bH\2dots\bt)$ is transversely isotropic.
\end{rem}

\begin{proof}
  We will first investigate the four equations in lemma~\ref{lem:cube-orientation}. Without loss of generality, we can assume that $G_{\bH}=\octa$ and thus that the polynomial form of $\bH$ is given (up to a scaling factor) by
  \begin{equation*}
    \rp(x,y,z) = x^{4} + y^{4} + z^{4} - 3x^{2}y^{2} - 3x^{2}z^{2} - 3y^{2}z^{2}.
  \end{equation*}
  Now, every transversely isotropic second-order homogeneous polynomial $\rt$ writes as
  \begin{equation*}
    \rt(x,y,z) =(\mu - \lambda)(\nn \cdot \xx)^{2} + \lambda \qq,
  \end{equation*}
  where $\lambda \ne \mu$, $\nn = (n_{1},n_{2},n_{3})$ is a unit vector and $\qq = x^{2}+y^{2}+z^{2}$. We get thus:
  \begin{itemize}
    \item $\tr(\bH \times \bt) =0$ if and only if
          \begin{equation}\label{eq:sol1}
            n_{1}n_{2} = n_{1}n_{3} = n_{2}n_{3} = 0;
          \end{equation}
    \item $\bt \times (\bH\2dots\bt) =0$ if and only if
          \begin{equation}\label{eq:sol2}
            \begin{split}
              n_{1} n_{2} \left(n_{1}^{2}-n_{3}^{2}\right) & = n_{1} n_{3} \left(n_{1}^{2}-n_{2}^{2}\right) = n_{2} n_{3} \left(n_{1}^{2}-n_{2}^{2}\right) =0,
              \\
              n_{1} n_{2} \left(n_{2}^{2}-n_{3}^{2}\right) & = n_{1} n_{3} \left(n_{2}^{2}-n_{3}^{2}\right) = n_{2} n_{3} \left(n_{3}^{2}-n_{1}^{2}\right) =0.
            \end{split}
          \end{equation}
    \item $\tr \left(\bt \times (\bH\2dots\bt)\right) =0$ if and only if
          \begin{equation}\label{eq:sol3}
            n_{1} n_{2} \left(n_{1}^{2}-n_{2}^{2}\right) = n_{1} n_{3} \left(n_{1}^{2}-n_{3}^{2}\right) = n_{2} n_{3} \left(n_{2}^{2}-n_{3}^{2}\right) = 0;
          \end{equation}
    \item $\tr(\bt \times (\bH\2dots\bt)) \times \tr(\bt \times (\bH\2dots\bt)^{2}) =0$ if and only if
          \begin{equation}\label{eq:sol4}
            \begin{split}
              n_{1}^{2}n_{2}n_{3}(n_{2}^{2}-n_{3}^{2})(n_{1}^{2}-n_{3}^{2})(n_{1}^{2}-n_{2}^{2}) = 0, \\
              n_{1}n_{2}^{2}n_{3}(n_{2}^{2}-n_{3}^{2})(n_{1}^{2}-n_{3}^{2})(n_{1}^{2}-n_{2}^{2}) = 0, \\
              n_{1}n_{2}n_{3}^{2}(n_{2}^{2}-n_{3}^{2})(n_{1}^{2}-n_{3}^{2})(n_{1}^{2}-n_{2}^{2}) = 0.
            \end{split}
          \end{equation}
  \end{itemize}
  Note also that $\eqref{eq:sol1} \implies \eqref{eq:sol2} \implies \eqref{eq:sol3} \implies \eqref{eq:sol4}$. We will now prove each statement of lemma~\ref{lem:cube-orientation}.

  (1) Suppose first that the pair $(\bH,\bt)$ is tetragonal, then $\tr(\bH \times \bt) = 0$ by lemma~\ref{lem:tr(Hxd)=0}. Conversely, if $\tr(\bH \times \bt) = 0$ holds, then we get \eqref{eq:sol1} and $\nn$ is collinear to either
  \begin{equation*}
    \ee_{1}, \qquad \ee_{2}, \qquad \ee_{3}.
  \end{equation*}
  Then, both $\bH$ and $\bt$ are invariant by the rotation by $\pi/2$ around $\nn$ and the pair $(\bH,\bt)$ is tetragonal.

  (2) Suppose now that the pair $(\bH,\bt)$ is trigonal, then the pair of second-order covariants $(\bt, \bH\2dots\bt)$ is at least trigonal an thus transversely isotropic by proposition~\ref{prop:symmetry-classes}. Thus $\bt \times (\bH\2dots\bt)=0$ by lemma~\ref{lem:axb=0}. Moreover, $\tr(\bH \times \bt) \ne 0$ by point (1). Conversely, if $\bt \times (\bH\2dots\bt)=0$ and $\tr(\bH \times \bt) \ne 0$, then we get \eqref{eq:sol2} with at least $n_{i}n_{j} \ne 0$ for a pair $(i,j)$ ($i \ne j$). In that case, $\nn$ is collinear to either
  \begin{equation*}
    \ee_{1}+\ee_{2}+\ee_{3}, \quad \ee_{1}-\ee_{2}+\ee_{3}, \quad \ee_{1}+\ee_{2}-\ee_{3}, \quad \ee_{1}-\ee_{2}-\ee_{3}.
  \end{equation*}
  Then, both $\bH$ and $\bt$ are invariant by the rotation by angle $2\pi/3$ around $\nn$ and the pair $(\bH,\bt)$ is trigonal.

  (3) Suppose now that the pair $(\bH,\bt)$ is orthotropic, then the first order covariant $\tr \left(\bt \times (\bH\2dots\bt)\right)$ is at least orthotropic and thus vanishes. Moreover $\bt \times (\bH\2dots\bt) \ne 0$ by points (1) and (2). Conversely, if $\tr \left(\bt \times (\bH\2dots\bt)\right)=0$ and $\bt \times (\bH\2dots\bt) \ne 0$, then we get \eqref{eq:sol3} with at least $n_{i} = 0$ for some $i \in\set{1,2,3}$. In that case $\nn$ is collinear to either
  \begin{equation*}
    \ee_{1}+\ee_{2},\quad \ee_{1}-\ee_{2},\quad \ee_{1}+\ee_{3},\quad \ee_{1}-\ee_{3},\quad \ee_{2}+\ee_{3},\quad \ee_{2}-\ee_{3}.
  \end{equation*}
  Then, both $\bH$ and $\bt$ are invariant by the rotation by angle $\pi$ (a second-order rotation) around a pair of axes $\ee_{i}\pm\ee_{j}$ and the pair $(\bH,\bt)$ is orthotropic.

  (4) Finally, suppose that the pair $(\bH,\bt)$ is monoclinic, then the pair of first-order covariants $(\tr(\bt \times (\bH\2dots\bt)),\tr(\bt \times (\bH\2dots\bt)^{2}))$ is at least monoclinic and thus collinear. Moreover, $\tr(\bt \times \bH\2dots\bt) \ne 0$ by $(1)$, $(2)$ and $(3)$. Conversely, if
  \begin{equation*}
    \tr(\bt \times (\bH\2dots\bt)) \times \tr(\bt \times (\bH\2dots\bt)^{2}) =0,
  \end{equation*}
  then we get \eqref{eq:sol4}. Since $\tr(\bt \times (\bH\2dots\bt)) \ne 0$ cases $(1)$, $(2)$ and $(3)$ are excluded and thus the pair $(\bH,\bt)$ is either monoclinic or triclinic, so we are reduced to show that it is monoclinic. If $n_{i}=0$ for some $i$, both $\bH$ and $\bt$ are invariant by the second-order rotation around $\ee_{i}$. Otherwise, we get $n_{i}=\pm n_{j}$ for a pair $(i,j)$. In that case, both $\bH$ and $\bt$ are invariant by the second-order rotation around $n_{i}\ee_{i}-n_{j}\ee_{j}$. This achieves the proof.
\end{proof}

We will end this section by formulating criteria for detecting orthotropic and monoclinic symmetry for a general harmonic tensor $\bH\in \HH^{4}(\RR^{3})$, using second-order covariants.

\begin{lem}\label{lem:orthotropic-criteria}
  Let $\bc\in \Sym^{2}(\RR^{3})$ be orthotropic and $\bH\in \HH^{4}(\RR^{3})$. Then,
  \begin{equation*}
    G_{(\bc, \bH \2dots \bc, \bH \2dots \bc^{2})} = G_{(\bH,\bc)}.
  \end{equation*}
  In particular, $(\bH,\bc)$ is orthotropic (resp. monoclinic) if and only if
  \begin{equation*}
    (\bc, \bH \2dots \bc, \bH \2dots \bc^{2})
  \end{equation*}
  is orthotropic (resp. monoclinic).
\end{lem}

\begin{proof}
  Note that the inclusion $G_{(\bH,\bc)} \subset G_{(\bc, \bH \2dots \bc, \bH \2dots \bc^{2})}$ is obvious, since $(\bc, \bH \2dots \bc, \bH \2dots \bc^{2})$ are covariants of the pair $(\bH,\bc)$. To prove the reverse inclusion, we can assume, without loss of generality, that $G_{\bc} = \DD_{2}$ (an thus that $\bc$ is diagonal). Let $g \in G_{(\bc, \bH \2dots \bc, \bH \2dots \bc^{2})} \subset G_{\bc}$. Then, $g$ is either the identity or a second-order rotation $r$ around either $\ee_{1}$, $\ee_{2}$, or $\ee_{3}$. Without loss of generality, we can suppose that $r$ is the rotation around $\ee_{3}$. Then, $\ee_{3}$ is a common eigenvector of $\bc$, $\bH \2dots \bc$ and $\bH \2dots \bc^{2}$. Moreover, since $\bH$ is harmonic, we have $\bH\2dots \bq = 0$ and thus
  \begin{equation*}
    [(\bH \2dots \bd)\ee_{3}] \times \ee_{3} = 0, \quad \text{for} \quad \bd = \bq,\bc,\bc^{2}.
  \end{equation*}
  Since $(\bq,\bc,\bc^{2})$ and $(\be_{11},\be_{22},\be_{33})$ generate the same three-dimensional vector space of diagonal matrices, this last condition can be recast as
  \begin{equation*}
    [(\bH \2dots \be_{ii})\ee_{3}] \times \ee_{3} = 0, \quad \text{for} \quad i=1,2,3.
  \end{equation*}
  and thus
  \begin{equation*}
    H_{1113} = H_{1123} = H_{1223} = H_{1333} = H_{2223} = H_{2333} = 0,
  \end{equation*}
  which means that $\bH$ is invariant under $r$, and so $r \in G_{(\bH, \bc)}$.
\end{proof}

Note that if $\ba$, $\bb$ are transversely isotropic, second-order symmetric tensors, then the pair $(\ba,\bb)$ is either monoclinic, orthotropic or transversely isotropic (see~\cite{Oli2017}), and we get the following corollary.

\begin{cor}\label{cor:transversely-isotropic-pair-criteria}
  Let $\ba$, $\bb$ be transversely isotropic second-order symmetric tensors and $\bH\in \HH^{4}(\RR^{3})$.
  \begin{enumerate}
    \item If $(\ba,\bb)$ is orthotropic, then
          \begin{equation*}
            G_{(\ba, \bb, \bH \2dots \ba, \bH \2dots \bb)} = G_{(\bH,\ba,\bb)}.
          \end{equation*}
          In particular, $(\bH,\ba,\bb)$ is orthotropic (resp. monoclinic) if and only if $(\ba, \bb, \bH \2dots \ba, \bH \2dots \bb)$ is orthotropic (resp. monoclinic).
    \item If $(\ba,\bb)$ is monoclinic, then
          \begin{equation*}
            G_{(\ba, \bb, \bH \2dots \ba, \bH \2dots \bb, \bH \2dots (\ba \bb)^{s})} = G_{(\bH,\ba,\bb)}.
          \end{equation*}
          In particular, $(\bH,\ba,\bb)$ is monoclinic if and only if
          \begin{equation*}
            (\ba, \bb, \bH \2dots \ba, \bH \2dots \bb, \bH \2dots (\ba \bb)^{s})
          \end{equation*}
          is monoclinic.
  \end{enumerate}
\end{cor}

\begin{proof}
  (1) Suppose that $(\ba,\bb)$ is orthotropic. Then there exists a basis $(\ee_{i})$ in which both $\ba$ and $\bb$ are diagonal. Moreover, $(\bq, \ba, \bb)$ and $(\be_{11},\be_{22},\be_{33})$ generate the same three-dimensional vector space and the proof is similar to that of lemma~\ref{lem:orthotropic-criteria}.

  (2) Suppose that $(\ba,\bb)$ is monoclinic. Then, by lemma~\ref{lem:orthotropic-linear-combination}, there exists a linear combination $\bc$ of $\ba$ and $\bb$ which is orthotropic. But then, $\bc^{2}$ is a linear combination of $\bq$, $\ba$, $\bb$ and $(\ba\bb)^{s}$ and thus
  \begin{equation*}
    G_{(\ba, \bb, \bH \2dots \ba, \bH \2dots \bb, \bH \2dots (\ba \bb)^{s})} \subset G_{(\bc, \bH \2dots \bc, \bH \2dots \bc^{2})} = G_{(\bH,\bc)},
  \end{equation*}
  by lemma~\ref{lem:orthotropic-criteria}. Therefore
  \begin{equation*}
    G_{(\ba, \bb, \bH \2dots \ba, \bH \2dots \bb, \bH \2dots (\ba \bb)^{s})} \subset G_{(\bH,\bc)} \cap G_{(\ba,\bb)} \subset G_{(\bH, \ba,\bb)}.
  \end{equation*}
  The reverse inclusion being obvious, this achieves the proof.
\end{proof}

\section{Characterization of the Symmetry Classes of \texorpdfstring{$\HH^{4}$}{H4}}
\label{sec:H4-symmetry-classes}

In this section, we formulate coordinate-free conditions using covariants up to order 5 that identify the symmetry class of a given tensor $\bH\in \HH^{4}$ and we prove that these conditions are both necessary and sufficient. The partially ordered set of symmetry classes for $\HH^{4}$ is the same as the one for the elasticity tensor, pictured in~\autoref{fig:poset-H4}. The notations used in this section are those introduced in section~\ref{sec:H4-covariant-algebra}. We will start by connecting $\cov_{1}(\bH)$ and $\cov_{2}(\bH)$ by the following lemma.

\begin{lem}\label{lem:cov1-cov2}
  Let $\bH\in \HH^{4}$ be a fourth order harmonic tensor. Then
  \begin{enumerate}
    \item $\cov_{1}(\bH) = \set{0}$ if and only if $\cov_{2}(\bH)$ is at least orthotropic;
    \item $\dim \cov_{1}(\bH) = 1$ if and only if $\cov_{2}(\bH)$ is monoclinic;
    \item $\dim \cov_{1}(\bH) = 3$ if and only if $\cov_{2}(\bH)$ is triclinic.
  \end{enumerate}
\end{lem}

\begin{rem}
  In the proof of lemma~\ref{lem:cov1-cov2}, some arguments are general and relies on Section~\ref{sec:covariants-symmetry-dimension}, others depend on the very special case that, $\cov_{1}(\HH^{4})$ is generated by commutators of elements in $\cov_{2}(\HH^{4})$ (as can be checked in~\autoref{tab:cov-basis-H4}).
\end{rem}

\begin{proof}
  (1) If $\cov_{1}(\bH) = \set{0}$, then each commutator of a pair of elements in $\cov_{2}(\bH)$ vanishes. Thus all the elements of $\cov_{2}(\bH)$ commute together and they can be represented by diagonal matrices in a common basis. All these second-order covariants are thus invariant by $\DD_{2}$ and $\cov_{2}(\bH)$ is thus at least orthotropic. Conversely, if $\cov_{2}(\bH)$ is at least orthotropic, then, since $\cov_{1}(\bH)$ is generated by the commutators of $\cov_{2}(\bH)$, it vanishes.

  (2) If $\dim \cov_{1}(\bH) = 1$, then, by (1) $\cov_{2}(\bH)$ is either monoclinic or triclinic (see proposition~\ref{prop:cov2-symmetry-classes}). However, if $\cov_{2}(\bH)$ was triclinic and thus of dimension 6 by proposition~\ref{prop:cov2-symmetry-classes}, then $\cov_{2}(\bH)=\Sym^{2}(\RR^{3})$ and we could build two linearly independent commutators which belong to $\cov_{1}(\bH)$, which would lead to a contradiction. Therefore, $\cov_{2}(\bH)$ is monoclinic. Conversely, if $\cov_{2}(\bH)$ is monoclinic, then, $\cov_{1}(\bH)$ (which is generated by commutators of $\cov_{2}(\bH)$) is at least monoclinic and thus monoclinic by (1).

  (3) If $\dim \cov_{1}(\bH) = 3$, then $\cov_{2}(\bH)$ is necessarily triclinic by (1) and (2). Conversely, if $\cov_{2}(\bH)$ is triclinic, then $\dim \cov_{1}(\bH) \ge 2$ by (1) and (2) and thus $\dim \cov_{1}(\bH) = 3$ by proposition~\ref{prop:cov1-symmetry-classes}.
\end{proof}

The harmonic tensor $\bH$, being a particular elasticity tensor, can be represented by a symmetric endomorphism of the space $\Sym^{2}(\RR^{3})$, the so-called \emph{Kelvin representation}~\cite{TKel1856}. In the orthonormal basis
\begin{equation*}
  (\be_{11},\be_{22},\be_{33},\be_{23}/\sqrt{2},\be_{13}/\sqrt{2},\be_{12}/\sqrt{2}),
\end{equation*}
where $\be_{ij}$ was defined in~\eqref{eq:S2-orthogonal-basis}, the matrix of its component writes:
\begin{equation}\label{eq:harmonic-generic-matrix-form}
  [\bH] = \begin{pmatrix}
    A               & \sqrt{2}\,B \\
    \sqrt{2}\,B^{t} & 2\,C        \\
  \end{pmatrix}
\end{equation}
where
\begin{equation*}
  A : =   \begin{pmatrix}
    \Lambda_{2}+ \Lambda_{3} & -\Lambda_{3}              & -\Lambda_{2}              \\
    -\Lambda_{3}             & \Lambda_{3} + \Lambda_{1} & -\Lambda_{1}              \\
    -\Lambda_{2}             & -\Lambda_{1}              & \Lambda_{2} + \Lambda_{1} \\
  \end{pmatrix},
\end{equation*}
\begin{equation*}
  B : =  \begin{pmatrix}
    -X_{1}      & Y_{1}+Y_{2} & -Z_{2}      \\
    - X_{2}     & -Y_{1}      & Z_{1}+Z_{2} \\
    X_{1}+X_{2} & -Y_{2}      & -Z_{1}      \\
  \end{pmatrix}, \quad
  C : =  \begin{pmatrix}
    -\,\Lambda_{1} & -Z_{1}         & - Y_{1}        \\
    -Z_{1}         & -\,\Lambda_{2} & -X_{1}         \\
    -Y_{1}         & -X_{1}         & -\,\Lambda_{3}
  \end{pmatrix},
\end{equation*}
and $B^{t}$ is the transpose matrix of $B$.


\subsection{Case I: \texorpdfstring{$\cov_{2}(\bH)$}{Cov2(H)} is isotropic}

\begin{thm}\label{thm:cov2-isotropic}
  Let $\bH\in \HH^{4}$ be a fourth order harmonic tensor. The following propositions are equivalent.
  \begin{enumerate}
    \item $\cov_{2}(\bH)$ is isotropic;
    \item $\bH$ is either \emph{cubic} ($\bd_{2} \ne 0$) or \emph{isotropic} ($\bd_{2} = 0$);
    \item $\bd_{2}$ is isotropic.
  \end{enumerate}
\end{thm}

\begin{proof}
  We will show that $(1) \implies (2) \implies (3) \implies (1)$.

  Suppose first that $(1)$ is true, so that we have $\bd_{2} = \frac{1}{3} J_{2} \bq$ and $\bd_{3} = \frac{1}{3} J_{3} \bq$, since $\tr \bd_{2} = J_{2}$ and $\tr \bd_{3} = J_{3}$. Then, the covariants $\bd_{k}$ defined in~\eqref{eq:Boehler-covariants} write as
  \begin{equation}\label{eq:isotropic-Boehler-covariant}
    \begin{array} {lll}
      \bd_{2} = \frac{1}{3} J_{2} \bq, & \bd_{3} = \frac{1}{3} J_{3} \bq,        & \bd_{4} = \frac{1}{9} {J_{2}}^{2} \bq,
      \\
      \bd_{5} = 0,                     & \bd_{6} = \frac{1}{27} {J_{2}}^{3} \bq, & \bd_{7} = 0,
      \\
      \bd_{8} = 0,                     & \bd_{9}  = 0,                           & \bd_{10}  = 0,
    \end{array}
  \end{equation}
  and we get
  \begin{equation*}
    3J_{4} = {J_{2}}^{2}, \quad 9J_{6} = {J_{2}}^{3}, \quad J_{5} = J_{7} = J_{8} = J_{9} = J_{10} = 0.
  \end{equation*}
  Now, $\tr ({\bd_{3}}^{2})$ is an invariant of degree 6 and should be expressible as a linear combination of the invariants
  \begin{equation*}
    {J_{2}}^{3}, \quad {J_{3}}^{2}, \quad J_{2}J_{4}, \quad J_{6}.
  \end{equation*}
  In fact, the following relation, satisfied by any harmonic tensor $\bH \in \HH^{4}$ can be checked directly by computation:
  \begin{equation*}
    240 \, J_{6} + 39 \, {J_{2}}^{3} + 190 \, {J_{3}}^{2} - 198 \, J_{2}J_{4} - 540 \, \tr({\bd_{3}}^{2}) = 0.
  \end{equation*}
  When~\eqref{eq:isotropic-Boehler-covariant} are satisfied, this leads to the relation
  \begin{equation*}
    30\, {J_{3}}^{2} - {J_{2}}^{3} = 0.
  \end{equation*}
  If $J_{2} = 0$, then $\norm{\bH}^{2} = J_{2} = 0$, so that $\bH = 0$ is isotropic. Otherwise, we get
  \begin{equation}\label{eq:cubic-syzigies}
    \begin{aligned}
      3\,J_{4} & = {J_{2}}^{2}, & J_{5} & = 0, & 30\,{J_{3}}^{2} & = {J_{2}}^{3}, & 9\,J_{6} & = {J_{2}}^{3}, \\
      J_{7}    & = 0,           & J_{8} & = 0, & J_{9}           & = 0,           & J_{10}   & = 0,
    \end{aligned}
  \end{equation}
  and $J_{2} \ne 0$, which are, according to~\cite[Proposition 5.3]{AKP2014}, necessary and sufficient conditions for a tensor $\bH \in \HH^{4}$ to be cubic.

  The assertion $(2) \implies (3)$ is trivial because if $\bH$ is either cubic or isotropic, then $\bd_{2}$ as a covariant of $\bH$ inherits its symmetry and is thus necessarily isotropic.

  Suppose that $(3)$ is true, so that $\bd_{2} = \alpha \bq$, for some scalar $\alpha$. Then, using the fact that $\bH \2dots \bq = \tr_{34} \bH = 0$ we deduce first that
  \begin{equation*}
    \bc_{k} = 0, \qquad k \ge 3.
  \end{equation*}
  Now, using remark~\ref{rem:d3c3}, we have $\bc_{3} = 2\bd_{3}'$ and thus $\bd_{3} = \beta \bq$ for some scalar $\beta$. Since all first and second-order covariants in~\autoref{tab:cov-basis-H4} are build from $\bd_{2}$, $\bd_{3}$ and the $\bc_{k}$, we deduce that they are all isotropic. But every symmetric second-order covariant is obtained as a linear combination of either a product of an invariant with a second-order covariant from~\autoref{tab:cov-basis-H4} or the symmetric product of two first-order covariants from~\autoref{tab:cov-basis-H4}. Therefore, $\cov_{2}(\bH)$ is isotropic. This achieves the proof.
\end{proof}

\begin{rem}
  According to~\cite[Proposition 5.3]{AKP2014}, an harmonic tensor $\bH\in \HH^{4}$ is either cubic or isotropic if and only if the relations \eqref{eq:cubic-syzigies} are satisfied.
  Surprisingly, a consequence of theorem~\ref{thm:cov2-isotropic} (which uses this result) is that an harmonic tensor $\bH\in \HH^{4}$ is either cubic or isotropic if and only if $3\,J_{4} = {J_{2}}^{2}$ (one only needs to check the first relation of~\eqref{eq:cubic-syzigies}). Indeed, by theorem~\ref{thm:cov2-isotropic}, $\bH$ is at least cubic if and only if $\bd_{2}$ is isotropic, which is equivalent to $\bd_{2}^{\prime} = 0$. But
  \begin{equation*}
    \norm{\bd_{2}^{\prime}}^{2}  = \tr\left(\bd_{2} - \frac{1}{3}J_{2}\bq\right)^{2}
    = \tr \left(\bd_{2}^{2} - \frac{2}{3}J_{2}\bd_{2} + \frac{1}{9}J_{2}^{2}\bq\right)
    = \frac{1}{3}\left(3\,J_{4} - {J_{2}}^{2}\right).
  \end{equation*}
  The magic here is that in the proof of theorem~\ref{thm:cov2-isotropic}, the knowledge of an integrity basis of $\cov(\HH^{4})$ (given in~\autoref{tab:cov-basis-H4}) has been used (it was not known when~\cite{AKP2014} was written). This remark is important because it has always been pointed out by the elders that the knowledge of covariants, rather than invariants, is the cornerstone to understand the geometry of a representation. This statement is thus illustrated here.
\end{rem}


\subsection{Case II: \texorpdfstring{$\cov_{2}(\bH)$}{Cov2(H)} is transversely isotropic}

\begin{thm}\label{thm:cov2-transversely-isotropic}
  Let $\bH\in \HH^{4}$ be a fourth order harmonic tensor. The following propositions are equivalent.
  \begin{enumerate}
    \item $\cov_{2}(\bH)$ is transversely isotropic;
    \item $\bH$ is \emph{tetragonal}, \emph{trigonal} or \emph{transversely isotropic};
    \item the pair $(\bd_{2},\bc_{3})$ is transversely isotropic.
  \end{enumerate}
\end{thm}

\begin{rem}
  By virtue of lemma~\ref{lem:axa2=0}, lemma~\ref{lem:axb=0} and theorem~\ref{thm:cov2-transversely-isotropic}, condition $(3)$ in theorem~\ref{thm:cov2-transversely-isotropic} can be recast as
  \begin{equation*}
    {\bd_{2}}^{\prime} \ne 0, \qquad \bd_{2} \times (\bd_{2})^{2}  = 0, \qquad   \bc_{3} \times \bd_{2}  = 0.
  \end{equation*}
\end{rem}

\begin{proof}[Proof of theorem~\ref{thm:cov2-transversely-isotropic}]
  We will show that $(1) \implies (2) \implies (3) \implies (1)$.

  Suppose first that $(1)$ is true, then, without loss of generality, we can assume that all symmetric second-order covariants are invariant by the subgroup $\OO(2)$ and that at least one of them, say $\ba$ writes as
  \begin{equation*}
    \ba = \lambda\bq + \mu \pmb{\tau}, \qquad \mu \ne 0.
  \end{equation*}
  Thus, $\pmb{\tau} = \mathrm{diag}(1,1,-2)$ is an eigenvector of $\bH$ (because $\bH\2dots\bq=0$ and $\bH\2dots\ba$ is traceless) and we have
  \begin{equation*}
    Y_{2} = 0, \qquad X_{1} + X_{2} = 0, \qquad Z_{1} = 0, \qquad  \Lambda_{1} = \Lambda_{2}.
  \end{equation*}
  Now we compute $\bd_{2}$ from~\eqref{eq:harmonic-generic-matrix-form} and write that it must be invariant by $\OO(2)$ and we get
  \begin{eqnarray*}
    4Z_{2}X_{1} + (\Lambda_{1} + 4\Lambda_{3})Y_{1}& = &0,\\
    (\Lambda_{1} + 4\Lambda_{3})X_{1}-4Z_{2}Y_{1}& = &0.
  \end{eqnarray*}
  The solutions of this system break into two alternatives:
  \begin{enumerate}
    \item either $Z_{2} = 0$ and $\Lambda_{1} = -4\Lambda_{3}$,
    \item or, $X_{1} = Y_{1}  = 0$.
  \end{enumerate}
  In the first case, we get
  \begin{equation*}
    [\bH] =    \begin{pmatrix}
      -3\Lambda_{3}   & -\Lambda_{3}   & 4\Lambda_{3}  & -\sqrt{2}X_{1}  & \sqrt{2}Y_{1}  & 0             \\
      -\Lambda_{3}    & -3\Lambda_{3}  & -4\Lambda_{3} & -\sqrt{2} X_{1} & -\sqrt{2}Y_{1} & 0             \\
      4\Lambda_{3}    & 4\Lambda_{3}   & -8\Lambda_{3} & 0               & 0              & 0             \\
      - \sqrt{2}X_{1} & +\sqrt{2}X_{1} & 0             & 8\Lambda_{3}    & 0              & -2 Y_{1}      \\
      \sqrt{2}Y_{1}   & -\sqrt{2}Y_{1} & 0             & 0               & 8\Lambda_{3}   & -2 X_{1}      \\
      0               & 0              & 0             & -2 Y_{1}        & -2 X_{1}       & -2\Lambda_{3}
    \end{pmatrix}
  \end{equation*}
  which is at least trigonal since $g\star \bH = \bH$ for all $g\in \ZZ_{3}$. In the second case, we get
  \begin{equation*}
    [\bH] =    \begin{pmatrix}
      \Lambda_{1} + \Lambda_{3} & -\Lambda_{3}              & -\Lambda_{1} & 0             & 0             & -\sqrt{2}Z_{2} \\
      -\Lambda_{3}              & \Lambda_{1} + \Lambda_{3} & -\Lambda_{1} & 0             & 0             & \sqrt{2}Z_{2}  \\
      -\Lambda_{1}              & -\Lambda_{1}              & 2\Lambda_{1} & 0             & 0             & 0              \\
      0                         & 0                         & 0            & -2\Lambda_{1} & 0             & 0              \\
      0                         & 0                         & 0            & 0             & -2\Lambda_{1} & 0              \\
      -\sqrt{2}Z_{2}            & \sqrt{2}Z_{2}             & 0            & 0             & 0             & -2\Lambda_{3}
    \end{pmatrix}
  \end{equation*}
  which is at least tetragonal since $g\star \bH = \bH$ for all $g\in \ZZ_{4}$.

  The assertion $(2) \implies (3)$ is trivial because if $(2)$ is true then the pair $(\bd_{2},\bc_{3})$ is either transversely isotropic or isotropic, but it cannot be isotropic by virtue of theorem~\ref{thm:cov2-isotropic}.

  Finally, suppose that $(3)$ is true. Note first that $\bd_{2}$ cannot be isotropic, because of theorem~\ref{thm:cov2-isotropic}. Without loss of generality, we can assume, therefore, that
  \begin{equation*}
    \bd_{2}^{\prime} = \mu_{2} \pmb{\tau}, \qquad \bc_{3} = \mu_{3}\pmb{\tau},
  \end{equation*}
  where $\mu_{2} \ne 0$ and $\pmb{\tau} = \mathrm{diag}(1,1,-2)$. We get thus
  \begin{equation*}
    \bc_{3} = \bH\2dots \bd_{2} = \bH\2dots \bd_{2}^{\prime} = \mu_{2} \bH\2dots \pmb{\tau} = \mu_{3}\pmb{\tau}
  \end{equation*}
  leading to
  \begin{equation*}
    \bH\2dots \pmb{\tau} = \frac{\mu_{3}}{\mu_{2}} \pmb{\tau},
  \end{equation*}
  which means that $\pmb{\tau}$ is an eigenvector of $\bH$. But then
  \begin{equation*}
    \bc_{4} = \bH\2dots \bc_{3} = \frac{\mu_{3}^{2}}{\mu_{2}} \pmb{\tau}, \qquad \bc_{5} = \bH\2dots \bc_{4} = \frac{\mu_{3}^{3}}{\mu_{2}^{2}} \pmb{\tau},
  \end{equation*}
  and thus, the triple $(\bd_{2},\bc_{3},\bc_{4}, \bc_{5})$ is transversely isotropic. We deduce then from~\autoref{tab:cov-basis-H4}, that $\cov_{2}(\bH)$ is transversely isotropic.
\end{proof}

We will now formulate conditions which allow to distinguish between the three remaining cases: transversely isotropic, trigonal and tetragonal.

\begin{cor}\label{cor:transiso-trigo-tetra-criteria}
  Let $\bH\in \HH^{4}$ be a fourth order harmonic tensor. Then
  \begin{enumerate}
    \item $\bH$ is \emph{transversely isotropic} if and only if $\bd_{2}$ is transversely isotropic and
          \begin{equation*}
            \bH \times \bd_{2} = 0;
          \end{equation*}
    \item $\bH$ is \emph{tetragonal} if and only if $\bd_{2}$ is transversely isotropic,
          \begin{equation*}
            \bH\times \bd_{2} \ne 0, \quad \text{and} \quad \tr(\bH\times \bd_{2}) = 0;
          \end{equation*}
    \item $\bH$ is \emph{trigonal} if and only if $\bd_{2}$ is transversely isotropic,
          \begin{equation*}
            \tr(\bH \times \bd_{2}) \ne 0, \quad \text{and} \quad (\bH : \bd_{2}) \times \bd_{2} = 0.
          \end{equation*}
  \end{enumerate}
\end{cor}

\begin{proof}
  Note first that if $\bH$ is either transversely isotropic, tetragonal or trigonal then, $\bd_{2}$ is necessarily transversely isotropic, by theorem~\ref{thm:cov2-isotropic}.

  (1) If $\bH$ is \emph{transversely isotropic}, then, $\bH \times \bd_{2} = 0$ by lemma~\ref{lem:Sxd=0}. Conversely, if $\bd_{2}$ is transversely isotropic and $\bH \times \bd_{2} = 0$, then, $(\bH,\bd_{2})$ is transversely isotropic by lemma~\ref{lem:Sxd=0} and so is $\bH$.

  (2) If $\bH$ is \emph{tetragonal}, then, $\tr(\bH\times \bd_{2}) = 0$ by lemma~\ref{lem:tr(Hxd)=0} and $\bH\times \bd_{2} \ne 0$ by (1). Conversely, if the conditions in (2) are satisfied, then, $(\bH, \bd_{2})$ is at least tetragonal by lemma~\ref{lem:tr(Hxd)=0}, and so is $\bH$. Since $\bH$ cannot be isotropic or cubic by theorem~\ref{thm:cov2-isotropic} (because $\bd_{2}$ is assumed to be transversely isotropic), it is either tetragonal or transversely isotropic, the later case being excluded by the condition $\bH\times \bd_{2} \ne 0$.

  (3) If $\bH$ is \emph{trigonal}, then, the pair $(\bd_{2},\bc_{3})$ is transversely isotropic and thus
  \begin{equation*}
    (\bH : \bd_{2}) \times \bd_{2} = \bc_{3} \times \bd_{2} = 0,
  \end{equation*}
  by lemma~\ref{lem:axb=0}. Moreover, $\tr(\bH \times \bd_{2}) \ne 0$ by lemma~\ref{lem:tr(Hxd)=0}. Conversely, if the conditions in $(3)$ holds, then, the pair $(\bd_{2},\bc_{3})$ is transversely isotropic by lemma~\ref{lem:axb=0} and $\bH$ is either \emph{tetragonal}, \emph{trigonal} or \emph{transversely isotropic} by theorem~\ref{thm:cov2-transversely-isotropic}. Since $\bH$ cannot be transversely isotropic by lemma~\ref{lem:Sxd=0}, nor tetragonal by lemma~\ref{lem:tr(Hxd)=0}, it is necessarily trigonal.
\end{proof}

We will end this subsection with two lemmas which characterise the symmetry class of a pair $(\bH, \bt)$ where $\bH$ is a fourth-order harmonic tensor and $\bt$ is a transversely isotropic second-order symmetric tensor. This completes the results of Section~\ref{sec:covariant-criteria} and will be very useful to prove our main theorem in Section~\ref{sec:Ela-symmetry-classes}.

\begin{lem}\label{lem:trigonal-pair}
  Let $\bt \in \Sym^{2}(\RR^{3})$ be transversely isotropic and $\bH \in \HH^{4}(\RR^{3})$ be an harmonic fourth-order tensor. Then, the pair $(\bH, \bt)$ is trigonal if and only if
  \begin{equation}\label{eq:trigonal-pair}
    (\bH\2dots\bt)\times \bt=0,\quad \bd_{2}\times \bt =0, \quad \text{and} \quad \tr(\bH \times \bt) \ne 0.
  \end{equation}
\end{lem}

\begin{proof}
  Suppose first that $(\bH, \bt)$ is trigonal. Then the triplet of second-order covariants $(\bH\2dots\bt,\bd_{2},\bt)$ is at least trigonal and thus transversely isotropic by proposition~\ref{prop:symmetry-classes}. We have thus $(\bH\2dots\bt)\times \bt=0$ and $\bd_{2} \times \bt =0$ by lemma~\ref{lem:axb=0}. Moreover, $\tr(\bH \times \bt) \ne 0$, by lemma~\ref{lem:tr(Hxd)=0}. Conversely, suppose that conditions~\eqref{eq:trigonal-pair} are satisfied. Then, $\bd_{2}$ is at least transversely isotropic by lemma~\ref{lem:axb=0}.
  \begin{enumerate}
    \item If $\bd_{2}$ is isotropic, then $\bH$ is cubic by theorem~\ref{thm:cov2-isotropic} (it cannot vanish because we assume $\tr(\bH \times \bt) \ne 0$). But then, $(\bH, \bt)$ is trigonal by lemma~\ref{lem:cube-orientation}.
    \item If $\bd_{2}$ is transversely isotropic, then $\bd_{2}' = \lambda \bt'$ with $\lambda \ne 0$ and thus
          \begin{equation*}
            (\bH \2dots \bd_{2}) \times \bd_{2} =0, \quad \text{and} \quad \tr(\bH \times \bd_{2}) \ne 0.
          \end{equation*}
          Therefore $\bH$ is trigonal by corollary~\ref{cor:transiso-trigo-tetra-criteria} and so is the pair $(\bH, \bt)$.
  \end{enumerate}
\end{proof}

\begin{cor}\label{cor:transversely-isotropic-triplet}
  Let $\bt\in \Sym^{2}(\RR^{3})$ be a transversely isotropic and $\bH$ be an harmonic fourth-order tensor. Then, the pair $(\bH, \bt)$ is either trigonal, tetragonal or transversely isotropic if and only if the triplet $(\bd_{2}, \bt, \bH \2dots \bt)$ is transversely isotropic.
\end{cor}

\begin{proof}
  Suppose first that $(\bH, \bt)$ is either trigonal, tetragonal or transversely isotropic. Then the triplet of second-order covariants $(\bd_{2}, \bt, \bH \2dots \bt)$ is at least transversely isotropic and thus transversely isotropic. Conversely, suppose that $(\bd_{2}, \bt, \bH \2dots \bt)$ is transversely isotropic. Then we have
  \begin{equation*}
    \bd_{2} \times \bt = 0, \quad \text{and} \quad (\bH \2dots \bt) \times \bt = 0,
  \end{equation*}
  by lemma~\ref{lem:axb=0}. If $\tr(\bH \times \bt) = 0$, then, the pair $(\bH, \bt)$ is either tetragonal or transversely isotropic by lemma~\ref{lem:tr(Hxd)=0}. If $\tr(\bH \times \bt) \ne 0$, then, the pair $(\bH, \bt)$ is trigonal by lemma~\ref{lem:trigonal-pair}. This achieves the proof.
\end{proof}


\subsection{Case III: \texorpdfstring{$\cov_{2}(\bH)$}{Cov2(H)} is orthotropic}

\begin{lem}\label{lem:v5-V6}
  Let $\bH\in \HH^{4}$ be a fourth order harmonic tensor. Then
  \begin{equation*}
    \vv_{5} = \vv_{6} = 0 \quad \implies \quad \cov_{1}(\bH) = \set{0},
  \end{equation*}
  where $\vv_{5}: =  \pmb{\varepsilon} \2dots (\bd_{2}\bc_{3})$ and $\vv_{6}: =  \pmb{\varepsilon} \2dots (\bd_{2}\bc_{4})$.
\end{lem}

\begin{proof}
  If $\vv_{5} = \vv_{6} = 0$, then the commutators $[\bd_{2},\bc_{3}]$ and $[\bd_{2},\bc_{4}]$ vanish. Without loss of generality, we can assume that $\bd_{2}$ and $\bc_{3}$ are diagonal matrices. We will now show that
  \begin{equation*}
    [\bc_{3},\bc_{4}] = 0.
  \end{equation*}
  \begin{enumerate}
    \item[(a)] If $\bd_{2}$ is orthotropic, then $\bc_{4}$ is also diagonal (since $[\bd_{2},\bc_{4}]=0$) and thus $[\bc_{3},\bc_{4}] = 0$.
    \item[(b)] If $\bd_{2}$ is transversely isotropic, we can assume, without loss of generality, that $\bd_{2} = \mathrm{diag}(\lambda,\lambda,\mu)$ where $\lambda \ne \mu$. Then, since $\bc_{3} = \bH\2dots\bd_{2}$ is also diagonal, we get
      \begin{equation*}
        X_{1} + X_{2} = Y_{2} = Z_{1} = 0.
      \end{equation*}
      Expressing now that $(\td{d}_{2})_{11} = (\td{d}_{2})_{22}$ and $(\td{d}_{2})_{12} = 0$, we have
      \begin{equation*}
        Z_{2}(\Lambda_{1}-\Lambda_{2}) = (\Lambda_{3} + 2\Lambda_{1} + 2\Lambda_{2})(\Lambda_{1}-\Lambda_{2}) = 0.
      \end{equation*}
      But, since $[\bd_{2},\bc_{4}] = 0$ where $\bc_{4} = \bH^{2}\2dots\bd_{2}$, we get
      \begin{equation*}
        (\td{c}_{4})_{13} = (\td{c}_{4})_{23}  = 0,
      \end{equation*}
      and thus
      \begin{equation*}
        (\Lambda_{1}-\Lambda_{2})Y_{1} = (\Lambda_{1}-\Lambda_{2})X_{2}=0.
      \end{equation*}
      \begin{itemize}
        \item If $\Lambda_{1}=\Lambda_{2}$, then, $(\td{c}_{4})_{12} = (\mu-\lambda)(\Lambda_{2}-\Lambda_{1})Z_{2}=0$. Thus $\td{c}_{4}$ is diagonal and $[\bc_{3},\bc_{4}] = 0$.
        \item If $\Lambda_{1}\ne \Lambda_{2}$, then
              \begin{equation*}
                X_{2}=Y_{1}=Z_{2} = 0, \qquad \Lambda_{3} + 2\Lambda_{1} + 2\Lambda_{2} =0
              \end{equation*}
              and, once again, $\td{c}_{4}$ is diagonal and thus $[\bc_{3},\bc_{4}] = 0$.
      \end{itemize}
    \item[(c)] if $\bd_{2}$ is isotropic, then all second order covariants vanish (by theorem~\ref{thm:cov2-isotropic}).
  \end{enumerate}
  In each case, $\bd_{2}$, $\bc_{3}$, $\bc_{4}$ commute with each other and thus all the first-order covariants in~\autoref{tab:cov-basis-H4} vanish, leading to $\cov_{1}(\bH) = \set{0}$.
\end{proof}

\begin{thm}\label{thm:cov2-orthotropic}
  Let $\bH\in \HH^{4}$ be a fourth order harmonic tensor. The following propositions are equivalent.
  \begin{enumerate}
    \item $\cov_{2}(\bH)$ is orthotropic;
    \item $\bH$ is \emph{orthotropic};
    \item $\vv_{5} = \vv_{6} = 0$ and the pair $(\bd_{2},\bc_{3})$ is orthotropic.
  \end{enumerate}
  In that case, $G_{\bH} = G_{(\bd_{2}, \bc_{3}, \bc_{4})}$.
\end{thm}

\begin{rem}\label{rem:orthotropic-triplet}
  Condition $(3)$ implies that the triplet $(\bd_{2},\bc_{3}, \bc_{4})$ is orthotropic by lemma~\ref{lem:v5-V6}. Conversely, if the triplet $(\bd_{2},\bc_{3},\bc_{4})$ is orthotropic, then $(3)$ holds because if $(\bd_{2},\bc_{3})$ was at least transversely isotropic, then so would be $\cov_{2}(\bH)$ by theorem~\ref{thm:cov2-transversely-isotropic} and theorem~\ref{thm:cov2-isotropic}, which would lead to a contradiction. Thus, these two conditions are equivalent. However, checking condition $(3)$ requires less computations than checking that $(\bd_{2},\bc_{3}, \bc_{4})$ is orthotropic using theorem~\ref{thm:n-quadratic-forms}.
\end{rem}

\begin{proof}
  We will show that $(1) \implies (2) \implies (3) \implies (1)$.

  Suppose first that $\cov_{2}(\bH)$ is orthotropic and thus of dimension 3 by proposition~\ref{prop:cov2-symmetry-classes}. Then, by corollary~\ref{cor:dimension-trans-iso-subspaces}, there exists $\bc \in \cov_{2}(\bH)$ which is orthotropic. By lemma~\ref{lem:q-a-a2}, we deduce that $\cov_{2}(\bH) = \langle \bq, \bc, \bc^{2} \rangle$, and thus, without loss of generality, we can assume that $\cov_{2}(\bH)$ is the space of all diagonal tensors. Now, since $\bH \2dots \bq = 0$, $\bH \2dots \bc$ and $\bH \2dots \bc^{2}$ are second-order symmetric covariants, we deduce moreover that the space of diagonal matrices is invariant under $\bH$, which has thus the matrix representation
  \begin{equation}\label{eq:orthotropic-matrix-form}
    [\bH] =    \begin{pmatrix}
      \Lambda_{2} + \Lambda_{3} & -\Lambda_{3}            & -\Lambda_{2}            & 0              & 0             & 0             \\
      -\Lambda_{3}              & \Lambda_{3}+\Lambda_{1} & -\Lambda_{1}            & 0              & 0             & 0             \\
      -\Lambda_{2}              & -\Lambda_{1}            & \Lambda_{1}+\Lambda_{2} & 0              & 0             & 0             \\
      0                         & 0                       & 0                       & -2 \Lambda_{1} & 0             & 0             \\
      0                         & 0                       & 0                       & 0              & -2\Lambda_{2} & 0             \\
      0                         & 0                       & 0                       & 0              & 0             & -2\Lambda_{3}
    \end{pmatrix}
  \end{equation}
  which is the normal form of an harmonic tensor which is \emph{at least} orthotropic. Since it cannot be of lower symmetry by theorem~\ref{thm:cov2-transversely-isotropic} and theorem~\ref{thm:cov2-isotropic}, we conclude that $\bH$ is orthotropic.

  Suppose now that $\bH$ is orthotropic. Then, $\cov_{1}(\bH) = \set{0}$ by proposition~\ref{prop:cov1-symmetry-classes} and thus $\vv_{5} = \vv_{6} = 0$. Moreover, the pair $(\bd_{2},\bc_{3})$ is at least orthotropic and thus orthotropic by theorem~\ref{thm:cov2-transversely-isotropic} and theorem~\ref{thm:cov2-isotropic}. Thus we get $(3)$.

  Finally, suppose that $(3)$ holds. Then by lemma~\ref{lem:v5-V6}, $\cov_{1}(\bH) = \set{0}$ and thus $\cov_{2}(\bH)$ is at least orthotropic by lemma~\ref{lem:cov1-cov2} and thus orthotropic since $(\bd_{2},\bd_{3})$ is orthotropic.
\end{proof}


\subsection{Case IV: \texorpdfstring{$\cov_{2}(\bH)$}{Cov2(H)} is monoclinic}

\begin{lem}\label{lem:v5}
  Let $\bH\in \HH^{4}$ be a fourth order harmonic tensor. Then
  \begin{equation*}
    \vv_{5} = 0 \implies \dim \cov_{1}(\bH) \le 1,
  \end{equation*}
  where $\vv_{5} : = \pmb{\varepsilon} \2dots (\bd_{2}\bc_{3})$.
\end{lem}

\begin{proof}
  If $\vv_{5} = 0$, then $\bd_{2}$ and $\bc_{3}$ commute. We will distinguish 2 cases whether $\bd_{2}$ is orthotropic or transversely isotropic (if $\bd_{2}$ is isotropic, the result already holds by theorem~\ref{thm:cov2-isotropic}).

  (1) Suppose that $\bd_{2}$ is transversely isotropic. Without loss of generality we can assume that
  $\bd_{2} = \mathrm{diag}(\lambda,\lambda,\mu)$, $\lambda \ne \mu$,
  and that $\bc_{3} = \bH \2dots \bd_{2}$ is diagonal. We get then
  \begin{equation*}
    X_{1}+X_{2} = 0, \qquad Y_{2} = 0, \qquad Z_{1} = 0.
  \end{equation*}
  Using these substitutions, we have
  \begin{align*}
    (\bd_{2})_{11} & = 4\Lambda_{3}^{2} + 2 \Lambda_{2}\Lambda_{3} + 4\Lambda_{2}^{2} + 4Z_{2}^{2} + 6Y_{1}^{2} + 6 X_{1}^{2} \\
    (\bd_{2})_{22} & = 4\Lambda_{3}^{2} + 2\Lambda_{1}\Lambda_{3} + 4\Lambda_{1}^{2} + 4Z_{2}^{2} + 6Y_{1}^{2} + 6X_{1}^{2}   \\
    (\bd_{2})_{12} & = (\Lambda_{1}-\Lambda_{2})Z_{2}                                                                         \\
    (\bd_{2})_{13} & = (4\Lambda_{3} - 2\Lambda_{2} + 3\Lambda_{1})Y_{1} + 4X_{1}Z_{2}                                        \\
    (\bd_{2})_{23} & = (4\Lambda_{3} - 2\Lambda_{1} + 3\Lambda_{2})X_{1} - 4Y_{1}Z_{2}                                        \\
  \end{align*}
  and thus
  \begin{align*}
     & (\Lambda_{1} - \Lambda_{2})Z_{2} = 0,                                       \\
     & (4\Lambda_{3} - 2\Lambda_{2} + 3\Lambda_{1})Y_{1} + 4X_{1}Z_{2} = 0,        \\
     & (4\Lambda_{3} - 2\Lambda_{1} + 3\Lambda_{2})X_{1} - 4Y_{1}Z_{2} = 0,        \\
     & (\Lambda_{1} - \Lambda_{2})(2\Lambda_{1} + 2\Lambda_{2} + \Lambda_{3}) = 0.
  \end{align*}
  (a) If $\Lambda_{1} = \Lambda_{2}$, we get
  \begin{align*}
    (4\Lambda_{3} + \Lambda_{1})Y_{1} + 4X_{1}Z_{2}  & = 0, \\
    -(4\Lambda_{3} + \Lambda_{1})X_{1} + 4Y_{1}Z_{2} & = 0.
  \end{align*}
  Then either the determinant of the system $16Z_{2}^{2}+(4\Lambda_{3} + \Lambda_{1})^{2}$ does not vanish, and thus $ X_{1} = Y_{1} = 0$. In this case we get
  \begin{equation*}
    X_{1} = X_{2} = Y_{1} = Y_{2} = 0
  \end{equation*}
  and $\bH$ is invariant under the rotation by angle $\pi$ around $\ee_{3}$. Otherwise, we have $Z_{2} = 0$ and $4\Lambda_{3} + \Lambda_{1} = 0$. In that case, $\bc_{4}$ is also diagonal and commutes thus with both $\bd_{2}$ and $\bc_{3}$ and we are done by lemma~\ref{lem:v5-V6}.

  (b) If $\Lambda_{1} \ne \Lambda_{2}$ , then $Z_{2} = 0$ and
  \begin{align*}
    Y_{1}(4\Lambda_{3} - 2\Lambda_{2} + 3\Lambda_{1}) & = 0, \\
    X_{1}(4\Lambda_{3} - 2\Lambda_{1} + 3\Lambda_{2}) & = 0, \\
    \Lambda_{3} + 2\Lambda_{2} + 2\Lambda_{1}         & = 0. \\
  \end{align*}
  If $X_{1}= 0$ or $Y_{1} = 0$ then, we are done since $\bH$ is at least monoclinic in either cases. Thus we can assume that
  \begin{align*}
    4\Lambda_{3} - 2\Lambda_{2} +3\Lambda_{1}  & = 0, \\
    4\Lambda_{3} + 3\Lambda_{2} - 2\Lambda_{1} & = 0, \\
    \Lambda_{3} + 2\Lambda_{2} + 2\Lambda_{1}  & = 0,
  \end{align*}
  but the unique solution of this linear system is $\Lambda_{1}  = \Lambda_{2} = \Lambda_{3} = 0$, and then $\bc_{4} = 0$. Again, we are done by lemma~\ref{lem:v5-V6}.

  (2) Suppose that $\bd_{2}$ is orthotropic. Our strategy will be to show that $\vv_{6} = \pmb\varepsilon\2dots(\bd_{2}\bc_{4})$ is a common eigenvector of both $\bd_{2}$, $\bc_{3}$ and $\bc_{4}$, in which case $\dim \cov_{1}(\bH) = 1$ (if $\vv_{6} = 0$, then $\cov_{1}(\bH) = \set{0}$ by lemma~\ref{lem:v5-V6}). Note that, if we can prove that $\vv_{6}$ is an eigenvector of $\bd_{2}$, then, we are done because
  \begin{equation*}
    \bd_{2}(\bc_{3}\vv_{6}) = \bc_{3}(\bd_{2}\vv_{6}), \qquad \bd_{2}(\bc_{4}\vv_{6}) = \bc_{4}(\bd_{2}\vv_{6})
  \end{equation*}
  and $\bd_{2}$ (which is orthotropic) has only simple eigenvalues. Now $\bd_{2} \vv_{6}$ can be recast as a product of covariants in~\autoref{tab:cov-basis-H4}. Indeed, we have
  \begin{equation}\label{eq:d2v6}
    \bd_{2} \vv_{6} = J_{2}\vv_{6}-\vv_{8b},
  \end{equation}
  where
  \begin{equation*}
    \vv_{6} = \pmb\varepsilon\2dots(\bd_{2}\bc_{4}), \quad \text{and} \quad \vv_{8b} = \pmb\varepsilon\2dots(\bd_{2}^{2}\bc_{4}).
  \end{equation*}

  Without loss of generality, we can assume that $\bd_{2}$ is diagonal, and hence that $(\tq{q},\bd_{2},\bd_{2}^{2})$ is a basis of the space of diagonal matrices. Therefore
  \begin{equation}\label{eq:c3diag}
    \bc_{3} = \alpha\td{q} + \beta\bd_{2} + \gamma\bd_{2}^{2}.
  \end{equation}
  If $\gamma = 0$, then
  \begin{equation*}
    \bc_{4} = \bH\2dots \bc_{3} = \beta (\bH\2dots \bd_{2}) = \beta\bc_{3}
  \end{equation*}
  and we are done by lemma~\ref{lem:v5-V6}. Therefore, we can suppose that $\gamma \ne 0$. Contracting with $\bc_{4}$ both sides of \eqref{eq:c3diag}, we get
  \begin{equation*}
    \bc_{4}\bc_{3} = \alpha\bc_{4} + \beta\bc_{4}\bd_{2} + \gamma\bc_{4}\bd_{2}^{2},
  \end{equation*}
  and contracting with $\pmb\varepsilon$ leads to
  \begin{equation}\label{eq:v7beq1}
    \vv_{7b} = \pmb\varepsilon\2dots(\bc_{4}\bc_{3}) = -\beta\vv_{6}-\gamma\vv_{8b}.
  \end{equation}
  Now, contracting $\bH$ with both sides of \eqref{eq:c3diag}, we get
  \begin{equation*}
    \bc_{4} = \bH\2dots\bc_{3} = \beta\bH\2dots\bd_{2} + \gamma\bH\2dots\bd_{2}^{2} = \beta\bc_{3} + \gamma\bH\2dots\bd_{2}^{2}.
  \end{equation*}
  But $\bH\2dots\bd_{2}^{2}$ can be recast as a product of covariants in~\autoref{tab:cov-basis-H4}. Indeed
  \begin{equation*}
    8\bH\2dots\bd_{2}^{2}  = (-2J_{2}J_{3} + 8J_{5})\td{q}-2J_{3}\bd_{2} + 7J_{2}\bc_{3} + 10\bc_{5}-12(\bd_{2}\bc_{3})^{s}.
  \end{equation*}
  Therefore (remember that $\bd_{2}$ and $\bc_{3}$ commute), we have
  \begin{equation*}
    \bc_{4} = \left(\beta + \frac{7\gamma}{8}J_{2}\right)\bc_{3} + \gamma\left(J_{5}-\frac{J_{2}J_{3}}{4}\right)\bq - \frac{\gamma}{4}J_{3}\bd_{2} + \frac{5\gamma}{4}\bc_{5} - \frac{3\gamma}{2}\bd_{2}\bc_{3}
  \end{equation*}
  and thus
  \begin{equation*}
    \bc_{4}\bc_{3} = \left(\beta + \frac{7\gamma}{8}J_{2}\right) \bc_{3}^{2} + \gamma \left(J_{5} - \frac{J_{2}J_{3}}{4}\right) \bc_{3} - \frac{\gamma}{4}J_{3}\bd_{2}\bc_{3} + \frac{5\gamma}{4}\bc_{5}\bc_{3} - \frac{3\gamma}{2}\bd_{2}\bc_{3}^{2}.
  \end{equation*}
  Hence
  \begin{equation*}
    \vv_{7b} = \pmb\varepsilon\2dots(\bc_{4}\bc_{3}) = \frac{5\gamma}{3}\pmb\varepsilon\2dots(\bc_{5}\bc_{3}).
  \end{equation*}
  But $\pmb\varepsilon\2dots(\bc_{5}\bc_{3})$ can be recast as a product of covariants in~\autoref{tab:cov-basis-H4}. Indeed
  \begin{equation}\label{eq:c5c3}
    15\, \pmb\varepsilon \2dots (\bc_{5}\td{c_{3}}) = 4J_{3}\vv_{5} + 15J_{2}\vv_{6} + 18\vv_{8a} - 24\vv_{8b},
  \end{equation}
  where
  \begin{equation*}
    \vv_{5} = \pmb\varepsilon \2dots (\bd_{2}\td{c_{3}}) = 0, \quad \text{and} \quad  \vv_{8a} = \pmb\varepsilon \2dots (\bd_{2}\td{{c_{3}}^{2}}) = 0.
  \end{equation*}
  We have thus
  \begin{equation*}
    \vv_{7b} = \frac{5\gamma}{3}J_{2}\vv_{6}-\frac{8\gamma}{3}\vv_{8b}.
  \end{equation*}
  Using~\eqref{eq:v7beq1}, we deduce that
  \begin{equation*}
    \frac{5\gamma}{3}J_{2}\vv_{6} - \frac{8\gamma}{3}\vv_{8b} = \vv_{7b} = -\beta\vv_{6}-\gamma\vv_{8b}
  \end{equation*}
  and hence that
  \begin{equation*}
    \vv_{8b} = \left(J_{2} + \frac{3\beta}{5\gamma}\right)\vv_{6}.
  \end{equation*}
  Therefore
  \begin{equation*}
    \bd_{2}\vv_{6} = J_{2}\vv_{6}-\vv_{8b} = \frac{3\beta}{5\gamma}\vv_{6}
  \end{equation*}
  and $\vv_{6}$ is an eigenvector of $\bd_{2}$, which achieves the proof.
\end{proof}

\begin{cor}\label{cor:cov1-v5}
  Let $\bH\in \HH^{4}$ be a fourth order harmonic tensor. Then,
  \begin{equation}\label{eq:v5-v5}
    \vv_{5} \times \left[(\vv_{5}\cdot \bH \cdot \vv_{5}) \vv_{5}\right]  = 0 \quad  \text{and} \quad \vv_{5} \times \left[(\vv_{5} \cdot \bH^{2} \cdot \vv_{5}) \vv_{5}\right] = 0,
  \end{equation}
  if and only if $\bH$ is at least monoclinic.
\end{cor}

\begin{proof}
  Suppose first that~\eqref{eq:v5-v5} is satisfied. If $\vv_{5} = 0$, we are done by lemma~\ref{lem:v5}. Otherwise, we can suppose, without loss of generality, that $\vv_{5} = k \ee_{1}$ with $k \ne 0$. But then we get
  \begin{equation*}
    H_{1112} = H_{1113} = 0, \qquad (\bH^{2})_{1112} = (\bH^{2})_{1113} = 0,
  \end{equation*}
  and thus
  \begin{align*}
    Y_{1} + Y_{2}                                 & = 0, \\
    Z_{2}                                         & = 0, \\
    2X_{1}Y_{2} -(\Lambda_{2}-\Lambda_{3})Z_{1}   & = 0, \\
    (\Lambda_{2}-\Lambda_{3})Y_{2} + 2X_{1} Z_{1} & = 0.
  \end{align*}
  If $4X_{1}^{2} +(\Lambda_{2}-\Lambda_{3})^{2} \ne 0$, then $Y_{2} = Z_{1} = 0$ and we are done (since then, $\bH$ is a normal form of a  monoclinic tensor). Otherwise, we get $\Lambda_{3} = \Lambda_{2}$ and $X_{1} = 0$. Then $\bd_{2}$ and $\bc_{3}$ commute so that $\vv_{5} = 0$, which leads to a contradiction. Conversely, if $\bH$ is at least monoclinic, then $\dim \cov_{1}(\bH) \le 1$ by proposition~\ref{prop:cov1-symmetry-classes}, and thus we get \eqref{eq:v5-v5}. This achieves the proof.
\end{proof}

\begin{thm}\label{thm:cov2-monoclinic}
  Let $\bH\in \HH^{4}$ be a fourth order harmonic tensor. The following propositions are equivalent.
  \begin{enumerate}
    \item $\bH$ is monoclinic;
    \item $\cov_{2}(\bH)$ is monoclinic;
    \item the triplet $(\bd_{2}, \bc_{3}, \bc_{4})$ is monoclinic.
  \end{enumerate}
  In that case, $G_{\bH} = G_{(\bd_{2}, \bc_{3}, \bc_{4})}$.
\end{thm}

\begin{proof}
  We will prove that $(1) \implies (2) \implies (3) \implies (1)$.

  Suppose first that $(1)$ holds. Then, $\cov_{2}(\bH)$ is at least monoclinic. But $\cov_{2}(\bH)$ cannot be orthotropic, transversely isotropic, nor isotropic by Theorems~\ref{thm:cov2-orthotropic}, \ref{thm:cov2-transversely-isotropic}, and  \ref{thm:cov2-isotropic}. Thus $\cov_{2}(\bH)$ is monoclinic.

  Suppose now that $(2)$ holds. Then, the triplet $(\bd_{2}, \bc_{3}, \bc_{4})$ is at least monoclinic. Since it cannot be at least orthotropic by lemma~\ref{lem:v5-V6} and lemma~\ref{lem:cov1-cov2}, it is thus monoclinic.

  Finally, suppose that $(3)$ holds. Then, $\bH$ is either monoclinic or triclinic. Moreover, there exists a basis where each element of the triplet $(\bd_{2}, \bc_{3}, \bc_{4})$ can be written as
  \begin{equation*}
    \left(
    \begin{array}{ccc}
      * & * & 0 \\
      * & * & 0 \\
      0 & 0 & * \\
    \end{array}
    \right)
  \end{equation*}
  and using the results of~\autoref{tab:cov-basis-H4}, we can conclude that $\dim \cov_{1}(\bH) = 1$. But then, $\bH$ is at least monoclinic by corollary~\ref{cor:cov1-v5} and thus monoclinic.
\end{proof}


\subsection{Case V: \texorpdfstring{$\cov_{2}(\bH)$}{Cov2(H)} is triclinic}

\begin{thm}\label{thm:cov1-triclinic}
  Let $\bH\in \HH^{4}$ be a fourth order harmonic tensor. The following propositions are equivalent.
  \begin{enumerate}
    \item $\tq H$ is triclinic;
    \item $\cov_{2}(\bH)$ is triclinic;
    \item the triplet $(\bd_{2}, \bc_{3}, \bc_{4})$ is triclinic.
  \end{enumerate}
\end{thm}

\begin{proof}
  We will prove that $(1) \implies (2) \implies (3) \implies (1)$. If $(1)$ holds, then $(2)$ holds by Theorems~\ref{thm:cov2-isotropic}, \ref{thm:cov2-transversely-isotropic}, \ref{thm:cov2-orthotropic} and \ref{thm:cov2-monoclinic}. If $(2)$ holds, then $(3)$ holds because if $(\bd_{2}, \bc_{3}, \bc_{4})$ is at least monoclinic, then $\dim \cov_{1}(\bH) \le 1$ and thus $\cov_{2}(\bH)$ is at least monoclinic by lemma~\ref{lem:cov1-cov2}. Finally, if $(3)$ holds, then $\bH$ is necessarily triclinic.
\end{proof}

\section{Characterization of the Symmetry Class of an elasticity tensor}
\label{sec:Ela-symmetry-classes}

The harmonic decomposition of the elasticity tensor was first obtained by Backus~\cite{Bac1970} (see also~\cite{Cow1989,Bae1993}) and is given by
\begin{equation*}
  \Ela \simeq \HH^{0} \oplus \HH^{0} \oplus \HH^{2} \oplus \HH^{2} \oplus \HH^{4}.
\end{equation*}
More precisely (see~\cite{BKO1994} for instance), given an orthonormal frame $(\ee_{1}, \ee_{2}, \ee_{3})$, each elasticity tensor $\bE$ can be written as
\begin{equation}\label{eq:boehler-decomposition}
  \begin{split}
    E_{ijkl} & = \lambda \delta_{ij} \delta_{kl} + \mu(\delta_{ik} \delta_{jl} + \delta_{il} \delta_{jk}) \\
    & \quad + \delta_{ij} a_{kl} + \delta_{kl} a_{ij} \\
    & \quad + \delta_{ik} b_{jl} + \delta_{jl} b_{ik} + \delta_{il} b_{jk} + \delta_{jk} b_{il} \\
    & \quad + H_{ijkl}.
  \end{split}
\end{equation}
In this decomposition, $\lambda,\mu$ (the generalized Lamé coefficients) and the deviators $\ba,\bb$ are related to the \emph{dilatation tensor} $\td d:=\tr_{12} \tq E$ and the \emph{Voigt tensor} $\td v:=\tr_{13}\tq E$ by the following process~\cite{Cow1989}. Starting with \eqref{eq:boehler-decomposition}, we get
\begin{equation*}
  \bd  = (3\lambda + 2 \mu)\bq + 3\ba + 4\bb, \quad
  \bv  = (\lambda + 4 \mu)\bq + 2\ba + 5\bb.
\end{equation*}
Taking the traces of each equation, one obtains
\begin{equation*}
  \tr(\bd)  =  9\lambda + 6 \mu , \quad
  \tr(\bv)  =  3\lambda + 12 \mu,
\end{equation*}
and, finally:
\begin{equation*}
  \begin{aligned}
    \lambda & = \frac{1}{15}( 2  \tr(\bd) - \tr(\bv)),          & \mu & =  \frac{1}{30}(-   \tr(\bd)   + 3    \tr(\bv)),      \\
    \ba     & =  \frac{1}{7}( 5\bd^{\prime} - 4 \bv^{\prime}) , & \bb & = \frac{1}{7} ( -2\bd^{\prime}   + 3   \bv^{\prime}),
  \end{aligned}
\end{equation*}
where $\bd^{\prime}:=\td d-\frac{1}{3}\tr( \td d)\,\bq$ and $\bv^{\prime}:=\td v-\frac{1}{3}\tr(\td v)\,\bq$ are the deviatoric parts of $\bd$ and $\bv$ respectively.

The fourth-order harmonic component $\bH$ is obtained using $\bS := (\bE)^{s}$, the total symmetrization of $\bE$, given by
\begin{equation*}
  S_{ijkl} = \frac{1}{3} (E_{ijkl}+E_{ikjl}+E_{iljk}).
\end{equation*}
The traceless part of $\bS$, $\bH$, is then given by
\begin{equation*}
  \bH = \bS - \frac{2}{7} \bq \odot \left(\bd' + 2 \bv'\right) - \frac{1}{15}\left(\tr \bd + 2 \tr \bv\right) \bq \odot \bq
\end{equation*}
where
\begin{equation*}
  (\ba \odot \bb)_{ijkl} := \frac{1}{6} \big( a_{ij} b_{kl} + b_{ij} a_{kl}+a_{ik} b_{jl} +b_{ik} a_{jl} +a_{il} b_{jk} + b_{il} a_{jk})
\end{equation*}
if $\ba$ and $\bb$ are two symmetric second order tensors.

\begin{rem}
  An elasticity tensor $\bE$ can thus be written as
  \begin{equation*}
    \bE = (\bH, \ba, \bb, \lambda, \mu),
  \end{equation*}
  where $\lambda, \mu$ are scalars, $\ba, \bb \in \HH^{2}$ and $\bH \in \HH^{4}$. This decomposition is however \emph{not unique}. Indeed, substituting for $(\ba, \bb)$ any \emph{invertible linear combination} of them would lead to a similar decomposition and the same is true for the pair of scalars $(\lambda,\mu)$. In particular, in the following theorem, one can use $(\bd^{\prime}, \bv^{\prime})$ instead of $(\ba, \bb)$, for instance.
\end{rem}

We will now state our main theorem, which characterizes, using polynomial covariants, the symmetry class of an elasticity tensor.

\begin{thm}\label{thm:main}
  Let $\bE = (\bH, \ba, \bb, \lambda, \mu) \in \Ela$ be an harmonic decomposition of an elasticity tensor $\bE$, where $\bH \in \HH^{4}$, $\ba, \bb \in \HH^{2}$ and $\lambda, \mu$ are scalars. Then
  \begin{enumerate}
    \item $\bE$ is isotropic if and only if $\ba = \bb = \bd_{2} = 0$.

    \item $\bE$ is cubic if and only if $\ba = \bb = \bd_{2}^{\prime} = 0$ and $\bd_{2} \ne 0$.

    \item $\bE$ is transversely isotropic if and only if $(\bd_{2}, \ba, \bb)$ is transversely isotropic and
          \begin{equation*}
            \bH \times \bd_{2} = \bH \times \ba = \bH \times \bb = 0.
          \end{equation*}

    \item $\bE$ is tetragonal if and only if $(\bd_{2}, \ba, \bb)$ is transversely isotropic,
          \begin{equation*}
            \tr(\bH \times \bd_{2}) = \tr(\bH \times \ba) = \tr(\bH \times \bb) = 0,
          \end{equation*}
          and
          \begin{equation*}
            \bH \times \bd_{2} \ne 0, \quad \text{or} \quad \bH \times \ba \ne 0, \quad \text{or} \quad \bH \times \bb \ne 0.
          \end{equation*}

    \item $\bE$ is trigonal if and only if $(\bd_{2}, \ba, \bb)$ is transversely isotropic,
          \begin{equation*}
            \bd_{2} \times (\bH\2dots\bd_{2}) = \ba \times (\bH\2dots\ba) = \bb \times (\bH\2dots\bb) = 0,
          \end{equation*}
          and
          \begin{equation*}
            \tr( \bH \times \bd_{2}) \ne 0, \quad \text{or} \quad \tr(\bH \times \ba) \ne 0, \quad \text{or} \quad \tr(\bH \times \bb) \ne 0.
          \end{equation*}

    \item $\bE$ is orthotropic if and only if the family of second-order tensors
          \begin{equation*}
            \mathcal{F}_{o} := \set{\bd_{2}, \ba, \bb, \bc_{3}, \bc_{4}, \bH \2dots \ba, \bH \2dots \bb, \bH \2dots \ba^{2}, \bH \2dots \bb^{2}}
          \end{equation*}
          is orthotropic.

    \item $\bE$ is monoclinic if and only if the family of second-order tensors
          \begin{equation*}
            \mathcal{F}_{m} := \left\{\bd_{2}, \ba, \bb, \bc_{3}, \bc_{4}, \bH \2dots \ba, \bH \2dots \bb, \bH \2dots \ba^{2}, \bH \2dots \bb^{2}, \bH \2dots (\ba\bb)^{s}, \bH \2dots (\ba\bd_{2})^{s}, \bH \2dots (\bb\bd_{2})^{s}\right\}
          \end{equation*}
          is monoclinic.

    \item $\bE$ is triclinic if and only if none of the preceding conditions holds.
  \end{enumerate}
\end{thm}

\begin{rem}
  Explicit covariant relations on a finite family $\mathcal{F}$ of second-order tensors which characterize its symmetry class are provided by theorem~\ref{thm:n-quadratic-forms}.
\end{rem}

\begin{rem}
  Note that if the family $\mathcal{F}_{o}$ is transversely isotropic, then, the triplet $(\bd_{2},\ba,\bb)$ is transversely isotropic. Otherwise, it would be isotropic but then, $\bE$ would be either isotropic or cubic by points (1) and (2), and $\mathcal{F}_{o}$ would be isotropic itself, because covariants of $\bE$ cannot have less symmetry than $\bE$ itself. This would lead to a contradiction. Hence, if the family $\mathcal{F}_{o}$ is transversely isotropic, then, $(\bd_{2},\ba,\bb)$ is transversely isotropic and thus $\bE$ is either transversely isotropic (3), tetragonal (4), or trigonal (5).
\end{rem}

\begin{proof}[Proof of theorem~\ref{thm:main}]
  Note first that the symmetry class of $\bE$ is the same as the symmetry class of the triplet $(\bH, \ba, \bb)$ (see Section~\ref{sec:symmetry-classes}).

  (1) If $\bE$ is isotropic, then, $\bH, \ba, \bb$ are all isotropic and thus vanish, since they all belong to irreducible representations. Conversely, if $\ba = \bb = \bd_{2} = 0$, then, $\bH$ vanishes because $\norm{\bH}^{2} = \tr\bd_{2}$. Thus, $(\bH, \ba, \bb)$ is isotropic.

  (2) If $\bE$ is cubic, then, all second-order symmetric covariant are isotropic. Thus, $\ba = \bb = \bd_{2}^{\prime} = 0$ but $\bd_{2} \ne 0$ (otherwise, $(\bH, \ba, \bb)$ would be isotropic, by point (1)). Conversely, if $\ba  = \bb = \bd_{2}^{\prime} = 0$ and $\bd_{2} \ne 0$, then, $\bH$ is cubic according to theorem~\ref{thm:cov2-isotropic} and so is $(\bH, \ba, \bb)$.

  For the sequel of the proof, note that, so far, that we have proved that $\bE$ is either isotropic, or cubic if and only if the family of second-order covariants
  \begin{equation*}
    \mathcal{F}_{i} := \set{\bd_{2}, \ba, \bb}
  \end{equation*}
  is isotropic.

  (3) If $\bE$ is transversely isotropic, then, the triplet $(\ba, \bb, \bd_{2})$ is thus transversely isotropic. Moreover, each pair $(\bH,\bd_{2})$, $(\bH,\ba)$, $(\bH,\bb)$ is at least transversely isotropic and thus
  \begin{equation*}
    \bH \times \bd_{2} = \bH \times \ba = \bH \times \bb = 0,
  \end{equation*}
  by lemma~\ref{lem:Sxd=0} and Remark~\ref{rem:Sxq=0}. Conversely, if the conditions in (3) are satisfied, then, at least one of the covariants $\ba$, $\bb$, $\bd_{2}$ (call it $\bt$) is transversely isotropic and
  \begin{equation*}
    \bH \times \bt  = 0.
  \end{equation*}
  Therefore, the pair $(\bH,\bt)$ is transversely isotropic according to lemma~\ref{lem:Sxd=0} and so is the triplet $(\bH,\ba,\bb)$.

  (4) If $\bE$ is tetragonal, then, each pair of covariants $(\bH,\bd_{2})$, $(\bH,\ba)$, $(\bH,\bb)$ is at least tetragonal and thus
  \begin{equation*}
    \tr(\bH \times \bd_{2}) = \tr(\bH \times \ba) = \tr(\bH \times \bb) = 0,
  \end{equation*}
  by lemma~\ref{lem:tr(Hxd)=0} and Remark~\ref{rem:Sxq=0}. Now, since $(\ba, \bb, \bd_{2})$ is transversely isotropic, at least one of the covariants $\bd_{2}$, $\ba$, $\bb$ is transversely isotropic and thus
  \begin{equation*}
    \bH \times \bd_{2} \ne 0, \quad \text{or} \quad \bH \times \ba \ne 0, \quad \text{or} \quad \bH \times \bb \ne 0,
  \end{equation*}
  by lemma~\ref{lem:Sxd=0} (otherwise, one of the pairs $(\bH,\bd_{2})$, $(\bH,\ba)$, $(\bH,\bb)$ would be at least transversely isotropic and so would be $(\bH, \ba, \bb)$). Conversely, if conditions in (4) are satisfied, we can find a covariant $\bt$ among $\ba$, $\bb$, $\bd_{2}$ such that
  \begin{equation*}
    \tr(\bH \times \bt) = 0, \quad \text{and} \quad \bH \times \bt \ne 0.
  \end{equation*}
  Then, $\bt$ is necessarily transversely isotropic by Remark~\ref{rem:Sxq=0} and thus $(\bH,\bt)$ is at least tetragonal by lemma~\ref{lem:tr(Hxd)=0}. Moreover, $(\bH,\bt)$ cannot be transversely isotropic, nor isotropic by lemma~\ref{lem:Sxd=0}. Since it cannot be either cubic (since $\bt$ is transversely isotropic), it is in fact tetragonal, and so is the triplet $(\bH,\ba,\bb)$.

  (5) If $\bE$ is trigonal, then, $\cov_{2}(\bE)$ is at least transversely isotropic and we get, in particular,
  \begin{equation*}
    \bd_{2} \times (\bH\2dots\bd_{2}) = \ba \times (\bH\2dots\ba) = \bb \times (\bH\2dots\bb) = 0,
  \end{equation*}
  by lemma~\ref{lem:axb=0} and Remark~\ref{rem:Sxq=0}. Moreover, $(\ba, \bb, \bd_{2})$ is transversely isotropic and thus at least one of the covariants $\ba$, $\bb$, $\bd_{2}$ (call it $\bt$) is transversely isotropic. But then,
  \begin{equation*}
    G_{\bE} = G_{(\bH,\ba,\bb)} = G_{\bH} \cap G_{(\bd_{2},\ba,\bb)} = G_{\bH} \cap G_{\bt}.
  \end{equation*}
  Thus, the pair $(\bH, \bt)$ is trigonal and $\tr( \bH \times \bt) \ne 0$ by lemma~\ref{lem:trigonal-pair}. Conversely, if conditions in (5) are satisfied, we can find a covariant $\bt$ among $\ba$, $\bb$, $\bd_{2}$ (and thus at least transversely isotropic) such that
  \begin{equation*}
    \bt \times (\bH\2dots\bt) = 0, \quad \bd_{2} \times \bt = 0, \quad \text{and} \quad \tr( \bH \times \bt) \ne 0.
  \end{equation*}
  Then, $\bt$ is necessarily transversely isotropic by Remark~\ref{rem:Sxq=0} and thus $(\bH,\bt)$ is trigonal by lemma~\ref{lem:trigonal-pair}, and so is the triplet $(\bH,\ba,\bb)$.

  For the sequel of the proof, note that, so far, that we have proved that $\bE$ is either transversely isotropic, tetragonal or trigonal if and only if the family of second-order covariants
  \begin{equation*}
    \mathcal{F}_{ti} := \set{\bd_{2}, \ba, \bb, \bc_{3}, \bH\2dots\ba, \bH\2dots\bb}
  \end{equation*}
  is transversely isotropic.

  (6) If $\bE$ is orthotropic, then, the family of second-order covariants $\mathcal{F}_{o}$ is at least orthotropic. Since, moreover,
  \begin{equation*}
    \mathcal{F}_{i} \subset \mathcal{F}_{ti} \subset \mathcal{F}_{o},
  \end{equation*}
  $\mathcal{F}_{o}$ cannot be isotropic by points (1) and (2), neither transversely isotropic by points (3), (4) and (5). It is thus orthotropic. Conversely, if $\mathcal{F}_{o}$ is orthotropic, then $(\bH, \ba, \bb)$ is either orthotropic, monoclinic or triclinic because $(\bH, \ba, \bb)$ cannot have higher symmetry than its covariants. If either $\ba$ or $\bb$ is orthotropic, then, $(\bH, \ba, \bb)$ is orthotropic by lemma~\ref{lem:orthotropic-criteria}. The same conclusion holds if either $\bd_{2}$ or $\bc_{3}$ is orthotropic by theorem~\ref{thm:cov2-orthotropic}. Otherwise, $\ba$, $\bb$, $\bd_{2}$ and $\bc_{3}$ are each at least transversely isotropic. In that case, if either $(\ba,\bb)$, $(\ba,\bd_{2})$, $(\ba,\bc_{3})$, $(\bb,\bd_{2})$, $(\bb,\bc_{3})$ or $(\bd_{2},\bc_{3})$ is orthotropic, then, $(\bH, \ba, \bb)$ is orthotropic by corollary~\ref{cor:transversely-isotropic-pair-criteria} and the fact that $\bc_{3} = \bH \2dots \bd_{2}$ and $\bc_{4} = \bH \2dots \bc_{3}$
  . Thus, we can assume that the quadruplet $(\ba,\bb,\bd_{2},\bc_{3})$ is transversely isotropic (it cannot be isotropic, otherwise, so would be $\mathcal{F}_{o}$, because $\bH \2dots \bq = 0$). Note that, in this case, the alternative $\bd_{2}$ transversely isotropic is excluded, otherwise, $(\bd_{2},\bc_{3})$ would be transversely isotropic and so would be $\mathcal{F}_{o}$ by theorem~\ref{thm:cov2-transversely-isotropic}. Therefore, $\bd_{2}^{\prime} =0$ and $\bH$ is either isotropic or cubic by theorem~\ref{thm:cov2-isotropic}. The case where $\bH$ is isotropic (and thus vanishes) is excluded because then, $\mathcal{F}_{o}$ would be at least transversely isotropic. We can thus finally assume that $\bH$ is cubic. Then, either $\ba$ or $\bb$ is transversely isotropic. Let suppose it is $\ba$. Then $\bb$ is collinear to $\ba$ (since $\ba$ and $\bb$ are deviators) and the pair $(\ba, \bH \2dots \ba)$ has the same symmetry group as $\mathcal{F}_{o}$ and is thus orthotropic. Th
  erefore,
  \begin{equation*}
    \tr(\ba \times (\bH \2dots \ba)) = 0, \qquad \ba \times (\bH \2dots \ba) \ne 0,
  \end{equation*}
  and $(\bH, \ba)$ is orthotropic by lemma~\ref{lem:cube-orientation}, and so is $(\bH,\ba,\bb)$.

  (7) If $\bE$ is monoclinic then the family of covariants $\mathcal{F}_{m}$ is at least monoclinic and thus monoclinic, by points (1)--(6) and because
  \begin{equation*}
    \mathcal{F}_{i} \subset \mathcal{F}_{ti} \subset \mathcal{F}_{o} \subset \mathcal{F}_{m}.
  \end{equation*}
  Conversely, suppose that
  \begin{equation*}
    G_{\mathcal{F}_{m}} = \set{id,r},
  \end{equation*}
  where $r$ is a second-order rotation. Then, $(\bH,\ba,\bb)$ is at most monoclinic, because it cannot have higher symmetry than its covariants. Besides,
  \begin{equation*}
    r \in G_{\mathcal{F}_{m}} \subset G_{(\ba,\bb)},
  \end{equation*}
  so we have only to check that $r\in G_{\bH}$, to prove that
  \begin{equation*}
    r \in  G_{\bE} = G_{\bH} \cap G_{(\ba,\bb)}.
  \end{equation*}
  Now, since
  \begin{equation*}
    r \in G_{\mathcal{F}_{m}} \subset G_{(\bd_{2}, \bc_{3}, \bc_{4})},
  \end{equation*}
  $\bH$ is at least monoclinic by theorem~\ref{thm:cov1-triclinic}. If $\bH$ is either monoclinic or orthotropic, then, we are done by theorem~\ref{thm:cov2-monoclinic} and theorem~\ref{thm:cov2-orthotropic}, because, in these cases we have
  \begin{equation*}
    G_{\bH} = G_{(\bd_{2}, \bc_{3}, \bc_{4})}.
  \end{equation*}
  If $\bH$ is either transversely isotropic, tetragonal or trigonal, then $(\bd_{2}, \bc_{3})$ is transversely isotropic by theorem~\ref{thm:cov2-transversely-isotropic}. Thus $\bd_{2}$ is transversely isotropic and $\bd_{2} \times \bc_{3} = 0$ with $\bc_{3} = \bH \2dots \bd_{2}$. But then, the triplet $(\bd_{2}, \ba, \bb)$ is at most orthotropic, otherwise the family $\mathcal{F}_{m}$ would have the same symmetry group as $(\bd_{2}, \bH \2dots \bd_{2})$ and would be transversely isotropic. Therefore, either $\ba$ or $\bb$ (let call it $\bc$) is orthotropic, and we are done by lemma~\ref{lem:orthotropic-criteria}, because
  \begin{equation*}
    r \in G_{\mathcal{F}_{m}} \subset G_{(\bc, \bH \2dots \bc, \bH \2dots \bc^{2})} =  G_{(\bH,\bc)} \subset G_{\bH},
  \end{equation*}
  or $\ba$ and $\bb$ are both transversely isotropic but one of the three pair $(\ba,\bb)$, $(\ba, \bd_{2})$ or $(\bb, \bd_{2})$ (let call it $(\bt_{1},\bt_{2})$) is either orthotropic or monoclinic, and we are done by corollary~\ref{cor:transversely-isotropic-pair-criteria} (since $\bc_{3} = \bH \2dots \bd_{2}$), because then
  \begin{equation*}
    r \in G_{\mathcal{F}_{m}} \subset G_{(\bH, \bt_{1},\bt_{2})} \subset G_{\bH}.
  \end{equation*}
  Suppose now that $\bH$ is cubic. If either $\ba$ or $\bb$ is orthotropic, then, we are done by lemma~\ref{lem:orthotropic-criteria} and the same conclusion holds, by corollary~\ref{cor:transversely-isotropic-pair-criteria}, if $\ba$ and $\bb$ are transversely isotropic but the pair $(\ba,\bb)$ is orthotropic or monoclinic. We can thus assume that the pair of deviators $(\ba,\bb)$ is transversely isotropic (it cannot be isotropic otherwise, so would be $\mathcal{F}_{m}$). In that case, either $\ba$ or $\bb$ does not vanish and is thus transversely isotropic. Suppose, for instance, that $\ba \ne 0$. Then, $\bb$ is collinear to $\ba$ and
  \begin{equation*}
    G_{\mathcal{F}_{m}} = G_{(\ba, \bH \2dots \ba)}.
  \end{equation*}
  Thus, $(\ba, \bH \2dots \ba)$ is monoclinic and
  \begin{equation*}
    r \in G_{(\ba, \bH \2dots \ba)} = G_{(\ba, \bH)},
  \end{equation*}
  by lemma~\ref{lem:cube-orientation} and Remark~\ref{rem:cub-orientation}. Finally, if $\bH$ is isotropic, $\bH = 0$ and we are done. This achieves the proof.
\end{proof}

\appendix

\section{Covariants of binary forms}
\label{sec:binary-forms-covariants}

A binary form $\ff$ of degree $n$ is a homogeneous complex polynomial in two variables $u,v$ of degree $n$:
\begin{equation*}
  \ff(\bxi) = a_{0}u^{n} + a_{1}u^{n-1}v + \dotsb + a_{n-1}uv^{n-1} + a_{n}v^{n},
\end{equation*}
where $\bxi = (u,v)\in \CC^{2}$ and $a_{k}\in \CC$. The set of all binary forms of degree $n$ is a complex vector space of dimension $n + 1$ which will be denoted by $\Sn{n}$. The special linear group
\begin{equation*}
  \SL(2,\CC) : =  \set{\gamma: =
    \begin{pmatrix}
      a & b \\
      c & d
    \end{pmatrix}
    ,\quad ad-bc = 1}
\end{equation*}
acts naturally on $\CC^{2}$ and induces a left action on $\Sn{n}$, given by
\begin{equation*}
  (\gamma \star \ff)(\bxi): = \ff(\gamma^{-1} \bxi),
\end{equation*}
where $\gamma\in \SL(2,\CC)$. The spaces $\Sn{n}$ are irreducible representations of $\SL(2,\CC)$ (see~\cite{Ste1994} for instance) and every complex algebraic linear representation $V$ of $\SL(2,\CC)$ can be decomposed into a direct sum
\begin{equation*}
  V \simeq  \Sn{n_{1}}\oplus \dotsc \oplus \Sn{n_{p}}.
\end{equation*}

\begin{defn}
  The \emph{transvectant} of index $r$ of two binary forms $\ff\in \Sn{n}$ and $\bg\in \Sn{p}$ is defined as
  \begin{equation}\label{eq:transvectant}
    \trans{\ff}{\bg}{r} =  \frac{(n-r)!}{n!}\frac{(p-r)!}{p!}  \sum_{i = 0}^{r}(-1)^{i} \binom{r}{i} \frac{\partial^{r} \ff}{\partial^{r-i}u \partial^{i} v}
    \frac{\partial^{r} \bg}{\partial^{i}u \partial^{r-i}v},
  \end{equation}
  which is a binary form of degree $n + p-2r$ (which vanishes if $r > \min(n,p)$).
\end{defn}

\begin{ex}
  For two $n$-th powers binary forms
  \begin{equation}\label{eq:n-powers-transvectant}
    (\ba\bxi)^{n} : =  (a_{1}u + a_{2}v)^{n},\quad (\bb\bxi)^{p} : =  (b_{1}u + b_{2}v)^{p},
  \end{equation}
  we get the particularly simple form
  \begin{equation*}
    \trans{(\ba\bxi)^{n}}{(\bb\bxi)^{p}}{r} = (\ba\bb)^{r}(\ba\bxi)^{n-r}(\bb\bxi)^{p-r},
  \end{equation*}
  where by definition $(\ba\bb): = a_{1}b_{2}-a_{2}b_{1}$.
\end{ex}

\begin{defn}
  The covariant algebra of $V$ is defined as
  \begin{equation*}
    \cov(V) : =  \CC[V\oplus \CC^{2}]^{\SL(2,\CC)}.
  \end{equation*}
  The \textbf{degree} of a covariant $\hh \in \cov(V)$ is the total degree $d$ of $\hh$ in $\ff \in V$, whereas the total degree $k$ of $\hh$ in $\xi \in \CC^{2}$ is called the \textbf{order} of $\hh$.
\end{defn}

The key point is that the transvectant of two binary forms is $\SL(2,\CC)$-equivariant and that $\cov(V)$ is generated by the infinite set of \emph{iterated transvectants}~\cite{GY2010,Olv1999,Oli2017}:
\begin{equation*}
  \ff_{1}, \dotsc ,\ff_{p} \quad \trans{\ff_{i}}{\ff_{j}}{r}, \quad \trans{\ff_{i}}{\trans{\ff_{j}}{\ff_{k}}{r}}{s}, \quad \dotsc
\end{equation*}

\begin{rem}\label{rem:even-order-covariants}
  A consequence of this observation is that for every integer $n\geq 1$, the covariant algebra $\cov(\Sn{2n})$ is generated by \emph{even order} covariants.
\end{rem}

The remarkable achievement of Gordan is that he was able to provide a constructive (and extremely efficient) way to obtain a \emph{finite} generating set of transvectants for the covariant algebra of finite dimensional representation of $\SL(2,\CC)$. This algorithm is now known as \emph{Gordan's algorithm} (see~\cite{Oli2017}). There are in fact two versions of this algorithm; one of them produces a basis for $\cov(\Sn{n})$, provided we know bases for $\cov(\Sn{k})$, for each $k < n$. The other one produces a basis for $\cov(V_{1}\oplus V_{2})$, if we know bases for $\cov(V_{1})$ and $\cov(V_{2})$. More precisely, if $\set{\ff_{1}, \dotsc , \ff_{p}}$ and $\set{\bg_{1}, \dotsc , \bg_{q}}$ generate respectively $\cov(V_{1})$ and $\cov(V_{2})$, then the covariant algebra $\cov(V_{1} \oplus V_{2})$ is generated by the finite family of transvectants
\begin{equation*}
  \trans{\ff_{1}^{\alpha_{1}}\dotsb \ff_{p}^{\alpha_{p}}}{\bg_{1}^{\beta_{1}}\dotsb \bg_{q}^{\beta_{q}}}{r},
\end{equation*}
where the integers $(\alpha_{i},\beta_{i},u,v,r)$ are the \emph{irreducible solutions} of the \emph{Diophantine equation}
\begin{equation*}
  \sum_{i = 1}^{p} a_{i} \alpha_{i} = u + r, \qquad \sum_{j = 1}^{p} b_{j} \beta_{j} = v + r
\end{equation*}
and $a_{i},b_{j}$ are the orders of $\ff_{i},\bg_{j}$.

Using this algorithm, we will formulate a theorem which connects generating sets for $\cov(\Sn{2n})$ and $\inv(\Sn{2n}\oplus \Sn{2})$. First, observe that there is a natural covariant mapping
\begin{equation*}
  \psi : \CC^{2} \to \Sn{2}, \qquad \eta \mapsto \bw_{\eta},
\end{equation*}
where
\begin{equation*}
  \bw_{\eta}(\pmb{\bxi}) : = (\eta_{1}v - \eta_{2}u)^{2}, \qquad \bxi = (u,v).
\end{equation*}
By pullback, this mapping induces an algebra homomorphism
\begin{equation*}
  \psi^{*} : \CC[\Sn{2n}\oplus \Sn{2}]^{\SL(2,\CC)} \to \CC[\Sn{2n}\oplus \CC^{2}]^{\SL(2,\CC)} = \cov(\Sn{2n})
\end{equation*}
given by
\begin{equation*}
  \psi^{*}(\rp)(\ff, \bxi) = \rp(\ff, \bw_{\eta}),\qquad \rp \in \CC[\Sn{2n}\oplus \Sn{2}]^{\SL(2,\CC)}.
\end{equation*}
Consider now the covariant linear mapping
\begin{equation}\label{eq:psi-section-definition}
  \varsigma: \cov(\Sn{2n}) \to \CC[\Sn{2n}\oplus \Sn{2}]^{\SL(2,\CC)}, \qquad \bh \mapsto \rp(\ff,\bw) := \sum_{k=0}^{r} \trans{\bh_{2k}}{\bw^{k}}{2k}
\end{equation}
where $\bh(\ff,\bxi)=\sum_{k=0}^{r}\bh_{2k}(\ff,\bxi)$ is the decomposition of $\bh$ into homogeneous covariants of order $2k$ (see remark~\ref{rem:even-order-covariants}). We have the following result.

\begin{lem}
  The algebra homomorphism $\psi^{*}$ is surjective and $\varsigma$ is a linear equivariant section of $\psi^{*}$. In other words
  \begin{equation*}
    \psi^{*} \circ \varsigma = \mathrm{Id}.
  \end{equation*}
\end{lem}

\begin{proof}
  Note first that \emph{$\varsigma$ is linear but is not an algebra homomorphism}. We will show that $\varsigma$ is a section of $\psi^{*}$ (as a linear mapping) and the surjectivity will follow. If $\bh$ is homogeneous of order $2r$, we have
  \begin{equation*}
    \varsigma(\bh) = \trans{\bh}{\bw^{r}}{2r},
  \end{equation*}
  and hence
  \begin{equation*}
    [(\psi^{*} \circ \varsigma)(\bh)](\ff,\eta) = \trans{\bh}{\bw_{\eta}^{r}}{2r}.
  \end{equation*}
  Thus, if $\bh(\ff,\bxi) = (\ba\pmb{\bxi})^{2r} = (a_{1}u + a_{2}v)^{2r}$ is a $2r$-th power binary form, we get
  \begin{equation*}
    (\psi^{*} \circ \varsigma)(\bh)(\ff,\eta) = \trans{(a_{1}u + a_{2}v)^{2r}}{(\eta_{1}v - \eta_{2}u)^{2r}}{2r} = (a_{1}\eta_{1} + a_{2}\eta_{2})^{2r} = \bh(\ff,\eta),
  \end{equation*}
  by virtue of~\eqref{eq:n-powers-transvectant}. Since every binary form of degree $2r$ is a linear combination of $2r$-th power binary forms, this achieves the proof.
\end{proof}

\begin{thm}\label{thm:basis-for-S2n-oplus-S2}
  Let $\set{\hh_{1},\dotsc ,\hh_{N}}$ be a minimal basis for $\cov(\Sn{2n})$. Then a minimal basis for the joint invariant algebra
  \begin{equation*}
    \CC[\Sn{2n}\oplus \Sn{2}]^{\SL(2,\CC)}
  \end{equation*}
  is given by
  \begin{equation*}
    \set{\varsigma(\hh_{1}),\dotsc ,\varsigma(\hh_{N}), \Delta}
  \end{equation*}
  where $\Delta(\bw) : =  b_{1}^{2} - b_{0}b_{2}$, if $\bw(\bxi) := b_{0} u^{2} + b_{1}uv + b_{2}v^{2} \in \Sn{2}$.
\end{thm}

\begin{rem}
  The result is still true if we replace, in the theorem, $\Sn{2n}$ by a direct sum of binary forms of even degree $\Sn{2n_{1}} \oplus \dotsb \oplus \Sn{2n_{k}}$.
\end{rem}

\begin{proof}
  Applying Gordan's algorithm to obtain a basis for $\inv(\Sn{2n}\oplus \Sn{2})$, and since $\cov(\Sn{2})$ is generated by the binary form $\bw$ itself and the invariant $\Delta$, we deduce that a generating set for $\inv(\Sn{2n}\oplus \Sn{2})$ is given by $\Delta$ and transvectants
  \begin{equation*}
    \trans{\bh_{1}^{\alpha_{1}}\cdots \bh_{N}^{\alpha_{N}}}{\bw^{r}}{2r},
  \end{equation*}
  where $(\alpha_{i},2r)$ is an irreducible solution of
  \begin{equation}\label{eq:Diophantine-equation}
    \sum_{i = 1}^{N} \alpha_{i} a_{i} = 2r
  \end{equation}
  and $\bh_{1}, \dotsc ,\bh_{N}$ are generators for $\cov(\Sn{2n})$, all of them being of even order.

  Now observe that, if a product $\bh_{1}^{\alpha_{1}}\cdots \bh_{N}^{\alpha_{N}}$ contains more than two factors, then $(\alpha_{i},2r)$ is reducible. Indeed it can be written as a sum of two non-trivial solutions $(\alpha^{1}_{i},2r_{1})$ and $(\alpha^{2}_{i},2r_{2})$ of~\eqref{eq:Diophantine-equation}, where $2r_{1} + 2r_{2} = 2r$. Thus, a finite set of generators for $\inv(\Sn{2n}\oplus \Sn{2})$ is given by $\Delta$ and
  \begin{equation*}
    \trans{\bh_{i}}{\bw^{r_{i}}}{2r_{i}}, \qquad i = 1, \dotsc , N
  \end{equation*}
  where $2r_{i}$ (see remark~\ref{rem:even-order-covariants}) is the order of $\bh_{i}$. To achieve the proof, it remains to show that if $\set{\hh_{1},\dotsc ,\hh_{N}}$ is minimal the same is true for $\set{\varsigma(\hh_{1}),\dotsc ,\varsigma(\hh_{N}), \Delta}$. To do so, observe that if for some $i \in \set{1, \dotsc , N}$, there exists a polynomial $P$ such that
  \begin{equation*}
    \varsigma(\hh_{i}) = P(\Delta,\varsigma(\hh_{j})), \qquad j \ne i,
  \end{equation*}
  then using the fact that $\psi^{*}$ is an algebra homomorphism, we get
  \begin{equation*}
    \hh_{i} = \psi^{*}(\varsigma(\hh_{i})) = \psi^{*}(P(\Delta,\varsigma(\hh_{j}))) = P(\psi^{*}(\Delta), \psi^{*}(\varsigma(\hh_{j}))) = P(0,\hh_{j}),
  \end{equation*}
  because $\psi^{*}(\Delta) = \Delta (\bw_{\eta}) = \Delta ((\eta_{1}v - \eta_{2}u)^{2}) = 0$, which leads to a contradiction.
\end{proof}

\section{Covariants of harmonic tensors}
\label{sec:harmonic-tensors-covariants}

There is a closed relation between covariant/invariant algebras of harmonic polynomials of three variables and those of binary forms which is recalled in this section (see also~\cite[Appendix B]{OKA2017}).

The complexification of the $\SO(3)$-representation on the real space of harmonic polynomials $\Hn{n}(\RR^{3})$ extends to a representation of the complex algebraic group
\begin{equation*}
  \SO(3,\CC) : =  \set{P \in \mathrm{M}_{3}(\CC); \; P^{t}P = \id ,\, \det P = 1}
\end{equation*}
on the space of complex harmonic polynomials $\Hn{n}(\CC^{3})$, which remains irreducible. There is, moreover, a group homomorphism~\cite[Appendix B]{OKA2017}
\begin{equation*}
  \pi : \SL(2,\CC) \to \SO(3,\CC), \qquad \gamma \mapsto \Ad_{\gamma},
\end{equation*}
where
\begin{equation*}
  \Ad_{\gamma} : M \mapsto \gamma M \gamma^{-1}, \qquad \gamma \in \SL(2,\CC), \quad M \in \slc(2,\CC)
\end{equation*}
is the adjoint action of $\SL(2,\CC)$ on its Lie algebra $\slc(2,\CC)$. When restricted to the real Lie group
\begin{equation*}
  \SU(2,\CC) : =  \set{\gamma \in \SL(2,\CC);\; \bar{\gamma}^{t} \gamma = \id},
\end{equation*}
it induces the well-known two-fold covering
\begin{equation*}
  \pi : \SU(2,\CC) \to \SO(3,\RR), \qquad \gamma \mapsto \Ad_{\gamma}.
\end{equation*}
Using these constructions, $\Hn{n}(\CC^{3})$ becomes an $\SL(2,\CC)$-representation if we set
\begin{equation}\label{eq:actiongamma}
  \gamma \star \rh : =  \pi(\gamma) \star \rh, \qquad \rh \in\Hn{n}(\CC^{3}),\quad \gamma \in \SL(2,\CC),
\end{equation}
and $\Hn{n}(\RR^{3})$ becomes an $\SU(2,\CC)$-representation if we set
\begin{equation*}
  \gamma \star \rh : =  \pi(\gamma) \star \rh, \qquad \rh \in\Hn{n}(\RR^{3}),\quad \gamma \in \SU(2,\CC),
\end{equation*}
both of them remaining irreducible.

Now, there is an equivariant isomorphism between the space $\Hn{n}(\CC^{3})$ of complex harmonic polynomials of degree $n$ and binary forms of degree $2n$. This isomorphism derives from an equivariant mapping introduced first in Cartan's theory of spinors going back to 1913 (see~\cite[Chapter 3]{Car1981}) and rediscovered later by Backus~\cite{Bac1970}. More precisely, let us introduce the \emph{Cartan map}
\begin{equation}\label{eq:Cartan-map}
  \phi : \CC^{2} \to \slc(2,\CC), \qquad \bxi \mapsto \bxi \, \bxi^{\omega}=
  \begin{pmatrix}
    -uv    & u^{2} \\
    -v^{2} & uv
  \end{pmatrix},
\end{equation}
where
\begin{equation*}
  \bxi =
  \begin{pmatrix}
    u \\
    v
  \end{pmatrix}
  , \qquad \bxi^{\omega} =
  \begin{pmatrix}
    -v & u
  \end{pmatrix},
\end{equation*}
and $\bxi^{\omega}$ means the covariant version of the vector $\bxi$, defined using the determinant $\omega$ on $\CC^{2}$ (a nondegenerate bilinear form). The main property of this mapping is that it is $\SL(2,\CC)$-equivariant, meaning that
\begin{equation*}
  \phi(\gamma \bxi) = \Ad_{\gamma} \phi(\bxi), \qquad \forall \gamma \in \slc(2,\CC).
\end{equation*}
Choosing the following basis
\begin{equation*}
  \begin{pmatrix}
    0  & 1 \\
    -1 & 0
  \end{pmatrix},
  \qquad
  \begin{pmatrix}
    0 & i \\
    i & 0
  \end{pmatrix},
  \qquad
  \begin{pmatrix}
    i & 0  \\
    0 & -i
  \end{pmatrix},
\end{equation*}
of the Lie algebra $\slc(2,\CC)$ (corresponding to multiplication by $i$ of Pauli matrices), allows us to identify $\slc(2,\CC)$ with $\CC^{3}$, using the parametrization
\begin{equation*}
  \begin{pmatrix}
    iz    & x+iy \\
    -x+iy & -iz
  \end{pmatrix}.
\end{equation*}
In this basis, the Cartan map~\eqref{eq:Cartan-map} writes
\begin{equation}\label{eq:explicit-Cartan-map}
  \phi : \CC^{2} \to \CC^{3}, \qquad (u,v) \mapsto \left( x= \frac{u^{2} + v^{2}}{2}, y= \frac{u^{2} - v^{2}}{2i}, z= iuv \right).
\end{equation}
By pullback, the Cartan map $\phi$ induces an equivariant isomorphism
\begin{equation*}
  \phi^{*} : \Hn{n}(\CC^{3}) \to \Sn{2n}, \qquad \rh \mapsto \rh \circ \phi ,
\end{equation*}
which is equivariant in the following sense (using~\eqref{eq:actiongamma})
\begin{equation*}
  \phi^{*}( \Ad_{\gamma} \star \rh) = \gamma \star \phi^{*}(\rh), \qquad \rh \in \Hn{n}(\CC^{3}), \, \gamma \in \SL(2, \CC).
\end{equation*}

\begin{thm}\label{thm:Hn-S2n-isomorphism}
  The linear mapping $\phi^{*}: \Hn{n}(\CC^{3}) \to \Sn{2n}$ defined by
  \begin{equation*}
    (\phi^{*}(\rh))(u,v) : =  \rh \left( \frac{u^{2} + v^{2}}{2}, \frac{u^{2} - v^{2}}{2i}, iuv \right).
  \end{equation*}
  is an $\SL(2,\CC)$-equivariant \emph{isomorphism}. The invariant algebras $\CC[\Hn{n}(\CC^{3})]^{\SO(3,\CC)}$ and $\CC[\Sn{2n}]^{\SL(2, \CC)}$ are thus isomorphic.
\end{thm}

\begin{rem}
  The equivariant isomorphism $\phi^{*}: \Hn{n}(\CC^{3}) \to \Sn{2n}$ is unique, up to a scaling factor, thanks to Schur's lemma. A different basis and thus a different representation has been considered by Backus in its study of the elasticity tensor, as \cite[Eq. 50]{Bac1970}
  \begin{equation*}
    (u,v) \mapsto \left( u^{2} - v^{2}, -i(u^{2} + v^{2}), 2uv \right).
  \end{equation*}
  Note also that a different representation was given in~\cite[Theorem 5.1]{OKA2017}, as
  \begin{equation*}
    (u,v) \mapsto \left( \frac{u^{2} - v^{2}}{2}, \frac{u^{2} + v^{2}}{2i}, uv \right).
  \end{equation*}
  However, expression \eqref{eq:explicit-Cartan-map} seems to be finally more convenient, especially when one works with transvectants.
\end{rem}

Let $\RSn{2n} : =  \phi^{*}(\Hn{n}(\RR^{3}))$ be the space of binary forms which correspond to real harmonic polynomials. This space is characterized as follows
\begin{equation*}
  \RSn{2n} = \set{\ff \in \Sn{2n}; \; S\ff = \ff},
\end{equation*}
where $S$ is the linear involution of $\Sn{2n}$ defined by
\begin{equation*}
  (S\ff)(u,v) = \bar{\ff}(-v,u),
\end{equation*}
and where $\bar{\ff}(u,v) := \overline{\ff(\bar{u},\bar{v})}$. This means that if
\begin{equation*}
  \ff = \sum_{k=0}^{2n} a_{k}u^{k}v^{2n-k},
\end{equation*}
then,
\begin{equation*}
  \ff \in \RSn{2n} \iff a_{2n-k} = (-1)^{k}\overline{a_{k}}, \qquad k = 0, \dotsc, 2n.
\end{equation*}

Note that $\RSn{2n}$ is invariant under the action of $\SU(2,\CC)$ and that the decomposition of the space $\Sn{2n}$ into irreducible components of $\SU(2,\CC)$ writes
\begin{equation*}
  \Sn{2n} = \RSn{2n} \oplus i\RSn{2n},
\end{equation*}
where $i\RSn{2n}$ is characterized by the functional equation $S\ff = -\ff$. Moreover, since we have the following commuting relations
\begin{equation*}
  \partial_{u} \circ S = S \circ \partial_{v}, \qquad \partial_{v} \circ S = -S \circ \partial_{u},
\end{equation*}
we deduce that
\begin{equation*}
  \trans{S\ff}{S\bg}{r} = S\trans{\ff}{\bg}{r},
\end{equation*}
by~\eqref{eq:transvectant}. Therefore, we have the following result.

\begin{lem}
  Let $\ff \in \RSn{2n}$ and $\bg \in \RSn{2p}$. Then $\trans{\ff}{\bg}{2r} \in \RSn{2n+2p-2r}$.
\end{lem}

In particular, an iterated transvectant of order 0 (\textit{i.e.} an invariant) is necessary real when evaluated on binary forms in $\RSn{2n}$, because $\RSn{0} = \RR$. Therefore, if $I_{1}, \dotsc, I_{N}$ are invariants of binary forms in $\Sn{2n_{1}}\oplus \dotsb \oplus \Sn{2n_{p}}$ obtained by such a transvectant process, they become real polynomials when evaluated on $\RSn{2n_{1}}\oplus \dotsb \oplus \RSn{2n_{p}}$.

Consider now the covariant algebra
\begin{equation*}
  \cov(\VV) = \RR[\VV \oplus \RR^{3}]^{\SO(3)}
\end{equation*}
where
\begin{equation*}
  \VV : =  \Hn{n_{1}}(\RR^{3})\oplus \dotsb \oplus \Hn{n_{p}}(\RR^{3}).
\end{equation*}
Then, using the group morphism $\pi : \SU(2,\CC) \to \SO(3,\RR)$ and the isomorphism introduced in theorem~\ref{thm:Hn-S2n-isomorphism}, we deduce an explicit real algebra isomorphism
\begin{equation*}
  \cov(\VV) \simeq \RR[\RSn{2n_{1}}\oplus \dotsb \oplus \RSn{2n_{p}} \oplus \RSn{2}]^{\SU(2,\CC)},
\end{equation*}
where we have made the trivial identification $\RR^{3} = \Hn{1}(\RR^{3})$ and we have the following result.

\begin{thm}\label{thm:decomplexification}
  Let $\set{\bg_{1},\dotsc,\bg_{N}}$ be a \emph{minimal} generating set of the complex covariant algebra
  \begin{equation*}
    \CC[\Sn{2n_{1}}\oplus \dotsb \oplus \Sn{2n_{p}} \oplus \Sn{2}]^{\SL(2,\CC)}
  \end{equation*}
  obtained by iterated transvectants. Then, by restriction, the set $\set{\bg_{1},\dotsc,\bg_{N}}$ defines a \emph{minimal} generating set of the real invariant algebra
  \begin{equation*}
    \RR[\RSn{2n_{1}}\oplus \dotsb \oplus \RSn{2n_{p}} \oplus \RSn{2}]^{\SU(2,\CC)}.
  \end{equation*}
\end{thm}

In practice, to obtain an explicit basis of the covariant algebra
\begin{equation*}
  \cov(\Hn{n_{1}}(\RR^{3})\oplus \dotsb \oplus \Hn{n_{p}}(\RR^{3}))
\end{equation*}
starting from the knowledge of a basis
\begin{equation*}
  \set{\bh_{1},\dotsc,\bh_{N}}
\end{equation*}
of the covariant algebra
\begin{equation*}
  \cov(\Sn{2n_{1}}\oplus \dotsb \oplus \Sn{2n_{p}}),
\end{equation*}
obtained by iterated transvectants, one can use lemma~\ref{lem:trad-transvectants} to translate these iterated transvectants and obtained tensorial expressions (or their polynomial counterparts, using the results of Section~\ref{sec:covariant-tensor-operations}) for the generators of
\begin{equation*}
  \cov(\Hn{n_{1}}(\RR^{3})\oplus \dotsb \oplus \Hn{n_{p}}(\RR^{3})),
\end{equation*}
don't omitting to add $\qq : =  x^{2} + y^{2} + z^{2}$ to this list.

\begin{lem}\label{lem:trad-transvectants}
  Let $\bH_{1} \in \HH^{n}(\RR^{3})$ and $\bH_{2} \in \HH^{p}(\RR^{3})$ be two harmonic tensors. Then we have
  \begin{equation}\label{eq:even-order-transvectant}
    \trans{\phi^{*}\bH_{1}}{\phi^{*}\bH_{2}}{2r} = 2^{-r}\phi^\ast((\bH_{1} \overset{(r)}{\cdot}\bH_{2})^{s}_0)
  \end{equation}
  and
  \begin{equation}\label{eq:odd-order-transvectant}
    \trans{\phi^{*}\bH_{1}}{\phi^{*}\bH_{2}}{2r+1} = \kappa(n,p,r) \phi^\ast((\tr^r(\bH_{1} \times \bH_{2}))_0)
  \end{equation}
  where
  \begin{equation*}
    \kappa(n,p,r) = \frac{1}{2^{2r+1}} \frac{(n+p-1)! (n-r-1)!(p-r-1)!}{(n+p-1-2r)! (n-1)! (p-1)! }.
  \end{equation*}
\end{lem}

\section{Invariants of the elasticity tensor}
\label{sec:elasticity-invariants}

An integrity basis of $297$ invariants for the elasticity tensor $\bE = (\bH, \ba, \bb, \lambda,\mu)$ was produced for the first time in~\cite{OKA2017}, using Gordan's algorithm~\cite{Oli2017}. In this appendix, we provide an alternative integrity basis, using the covariants of $\bH$ given in~\autoref{tab:cov-basis-H4}. All present computations have been done first using Macaulay2~\cite{GS}, and then, independently with Mathematica software. A \emph{minimal} integrity basis of $294$ invariants is finally obtained here. It turns out that the integrity basis produced in~\cite{OKA2017} is not minimal: indeed, the computations achieved here show that the three joint invariants of degree 11 in~\cite[Table 5]{OKA2017} are superfluous.

This basis has been checked to be correct using the Hilbert series of $\inv(\Ela)$ and using the method which was outlined in section~\ref{sec:H4-covariant-algebra}. This basis consists in:

\begin{enumerate}
  \item 15 simple invariants:
        \begin{itemize}
          \item $\lambda$, $\mu$;
          \item the simple invariants of $\ba$ and $\bb$: $\tr\ba^{2}$, $\tr\ba^{3}$, $\tr\bb^{2}$, $\tr\bb^{3}$:
          \item and the nine simple invariants of $\bH$, computed first in~\cite{BKO1994}:
                \begin{gather*}
                  \tr\td{d}_{2},\quad \tr\td{d}_{3},\quad \tr\td{d}_{2}^{2},\quad \tr\left(\td{d}_{2}\td{d}_{3}\right),\quad \tr\td{d}_{2}^{3}, \\
                  \quad \tr\left(\td{d}_{2}^{2}\td{d}_{3}\right), \quad \tr\left(\td{d}_{2}\td{d}_{3}^{2}\right), \quad \tr\td{d}_{3}^{3}, \quad \tr\left(\td{d}_{2}^{2}\td{d}_{3}^{2}\right);
                \end{gather*}
        \end{itemize}
  \item 4 joint invariants of $(\ba,\bb)$:
        \begin{equation*}
          \tr\left(\ba\bb\right),\quad \tr\left(\ba^{2}\bb\right),\quad \tr\left(\ba\bb^{2}\right),\quad \tr\left(\ba^{2}\bb^{2}\right)
        \end{equation*}
  \item 52 joint invariants of $(\bH,\ba)$, and similarly 52 joint invariants of $(\bH,\bb)$ given in~\autoref{tab:inv-H4-H2}, where
        \begin{equation*}
          \cv{3}{3} = \tr\big(\bH\times \td{d}_{2}\big) ,\quad \cv{4b}{5}=\big(\bH^{2}\big)^s\times\td{d}_{2},\quad \vv_{5} = \pmb\varepsilon\2dots\big(\td{d}_{2}\td{c}_{3}\big).
        \end{equation*}
  \item 171 joint invariant of $(\bH,\ba,\bb)$ given in~\autoref{tab:inv-H4-H2-H2-a}, ~\autoref{tab:inv-H4-H2-H2-b} and~\autoref{tab:inv-H4-H2-H2-c}, where
        \begin{gather*}
          \cv{3}{7} = \bH\times\big(\bH^{2}\big)^s, \quad \cv{3}{9} = \big( \big(\bH\cdot\bH\big)^s\times \bH \big), \\
          \cv{3}{5} = \bH\times \td{d}_{2}, \quad \cv{4}{7}=\big( \bH \times \big(\bH^{3}\big)^s\big), \quad
          \vv_{8a} = \pmb{\varepsilon} \2dots (\bd_{2}{\bc_{3}}^{2}).
        \end{gather*}
\end{enumerate}

\begin{table}[p]
  \centering
  \begin{minipage}{0.49\linewidth}
    \centering
    \resizebox{\textwidth}{!}{%
      \renewcommand{\arraystretch}{1.2}
      \begin{tabular}{|l|c|c|c|}
        \hline\renewcommand{\arraystretch}{1.2}
        \#        & Deg. & M. deg. & Inv.                                                                                                     \\
        \hline
        $j_{a1}$  & 3    & (2,1,0) & $\tr\big(\ba\td{d_{2}}\big)$                                                                             \\
        $j_{a2}$  & 3    & (1,2,0) & $\ba\2dots\bH\2dots\ba$                                                                                  \\
        $j_{a3}$  & 4    & (2,2,0) & $\tr\big(\ba^{2}\td{d_{2}}\big)$                                                                         \\
        $j_{a4}$  & 4    & (3,1,0) & $\tr\big(\ba\td{d_{3}}\big)$                                                                             \\
        $j_{a5}$  & 4    & (1,3,0) & $\ba\2dots\bH\2dots\ba^{2}$                                                                              \\
        $j_{a6}$  & 4    & (2,2,0) & $\ba\2dots\big(\bH^{2}\big)^s\2dots\ba$                                                                  \\
        $j_{a7}$  & 5    & (1,4,0) & $\ba^{2}\2dots\bH\2dots\ba^{2}$                                                                          \\
        $j_{a8}$  & 5    & (2,3,0) & $\ba\2dots\big(\bH^{2}\big)^s\2dots\ba^{2}$                                                              \\
        $j_{a9}$  & 5    & (2,3,0) & $\ba\2dots\big( \ba\2dots\big(\bH\cdot\bH\big)^s\2dots\ba\big)$                                          \\
        $j_{a10}$ & 5    & (3,2,0) & $\ba\2dots(\bH^{3})^s\2dots\ba$                                                                          \\
        $j_{a11}$ & 5    & (3,2,0) & $\tr\big(\ba^{2}\td{d}_{3}\big)$                                                                         \\
        $j_{a12}$ & 5    & (4,1,0) & $\tr\big(\ba\td{d}_{2}^2\big)$                                                                           \\
        $j_{a13}$ & 5    & (4,1,0) & $\ba\2dots\big(\bH^{2}\big)^s\2dots\td{d}_{2}$                                                           \\
        $j_{a14}$ & 6    & (2,4,0) & $\ba^{2}\2dots\big(\bH^{2}\big)^s\2dots\ba^{2}$                                                          \\
        $j_{a15}$ & 6    & (2,4,0) & $\ba^{2}\2dots\big( \ba\2dots\big(\bH\cdot\bH\big)^s\2dots\ba\big)$                                      \\
        $j_{a16}$ & 6    & (3,3,0) & $\ba\2dots\big( \ba\2dots\big(\bH^{2}\cdot\bH\big)^s\2dots\ba\big)$                                      \\
        $j_{a17}$ & 6    & (3,3,0) & $\ba\2dots\big(\bH^{3}\big)^s\2dots\ba^{2}$                                                              \\
        $j_{a18}$ & 6    & (3,3,0) & $\tr\big(\bH\times\td{d}_{2}\big)\3dots \big( \ba^{2}\times \ba\big)$                                    \\
        $j_{a19}$ & 6    & (4,2,0) & $\td{d}_{2}^{2}\2dots\ba^{2}$                                                                            \\
        $j_{a20}$ & 6    & (4,2,0) & $\ba\2dots\big(\bH^{4}\big)^s\2dots\ba$                                                                  \\
        $j_{a21}$ & 6    & (4,2,0) & $\ba^{2}\2dots\td{c}_{4}$                                                                                \\
        $j_{a22}$ & 6    & (5,1,0) & $\ba\2dots(\td{d}_{2}\td{d}_{3})$                                                                        \\
        $j_{a23}$ & 6    & (5,1,0) & $\ba\2dots\big(\bH^{3}\big)^s\2dots\td{d}_{2}$                                                           \\
        $j_{a24}$ & 7    & (2,5,0) & $\ba\2dots\big( \ba^{2}\2dots\big(\bH\cdot\bH\big)^s\2dots\ba^{2}\big)$                                  \\
        $j_{a25}$ & 7    & (3,4,0) & $\ba\2dots\big(\big(\bH\times\td{d}_{2}\big)\3dots \big( \ba^{2}\times \ba\big)\big)$                    \\
        $j_{a26}$ & 7    & (3,4,0) & $\ba\2dots\big( \ba\2dots\big(\bH^{2}\cdot\bH\renewcommand{\arraystretch}{1.2}\big)^s\2dots\ba^{2}\big)$ \\
        $j_{a27}$ & 7    & (4,3,0) & $\ba\2dots\big(\bH^{4}\big)^s\2dots\ba^{2}$                                                              \\
        \hline
      \end{tabular}}
  \end{minipage}
  \begin{minipage}{0.5\linewidth}
    \centering
    \resizebox{\textwidth}{!}{%
      \renewcommand{\arraystretch}{1.2}
      \begin{tabular}{|l|c|c|c|}
        \hline
        \#        & Deg. & M. deg.                                    & Inv.                                                                       \\
        \hline
        $j_{a28}$ & 7    & (4,3,0)                                    & $\tr\big(\bH\times\td{c}_{3}\big)\3dots \big( \ba^{2}\times \ba\big)$      \\
        $j_{a29}$ & 7    & (4,3,0)                                    & $\ba\2dots\big( \ba\2dots\big(\bH^{2}\cdot\bH^{2}\big)^s\2dots\ba\big)$    \\
        $j_{a30}$ & 7    & (5,2,0)                                    & $\ba\2dots\big(\bH\cdot\td{d}_{2}^{2}\big)^s\2dots\ba$                     \\
        $j_{a31}$ & 7    & (5,2,0)  \renewcommand{\arraystretch}{1.2} & $\td{c}_{5}\2dots\ba^{2}$                                                  \\
        $j_{a32}$ & 7    & (5,2,0)                                    & $(\td{d}_{2}\td{c}_{3})\2dots\ba^{2}$                                      \\
        $j_{a33}$ & 7    & (6,1,0)                                    & $(\td{d}_{2}\td{c}_{4})\2dots\ba$                                          \\
        $j_{a34}$ & 7    & (6,1,0)                                    & $\td{c}_{3}^{2}\2dots\ba$                                                  \\
        $j_{a35}$ & 8    & (7,1,0)                                    & $\big(\td{d}_{2}^{2}\td{c}_{3}\big)\2dots\ba$                              \\
        $j_{a36}$ & 8    & (7,1,0)                                    & $\big(\td{c}_{4}\td{c}_{3}\big)\2dots\ba$                                  \\
        $j_{a37}$ & 8    & (6,2,0)                                    & $\big(\td{d}_{2}\td{c}_{4}\big)\2dots\ba^{2}$                              \\
        $j_{a38}$ & 8    & (6,2,0)                                    & $\td{c}_{3}^{2}\2dots\ba^{2}$                                              \\
        $j_{a39}$ & 8    & (6,2,0)                                    & $\ba\2dots\big(\bH^{2}\cdot\td{d}_{2}^{2}\big)^s\2dots\ba$                 \\
        $j_{a40}$ & 8    & (5,3,0)                                    & $\tr\big(\bH\times\td{d}_{2}^{2}\big)\3dots \big( \ba^{2}\times \ba\big)$  \\
        $j_{a41}$ & 8    & (5,3,0)                                    & $\ba\2dots\big(\bH\cdot\td{d}_{2}^{2}\big)^s\2dots\ba^{2}$                 \\
        $j_{a42}$ & 8    & (4,4,0)                                    & $\ba\2dots\big(\cv{4b}{5}\3dots \big( \ba^{2}\times \ba\big)\big)$         \\
        $j_{a43}$ & 8    & (4,4,0)                                    & $\ba\2dots\big(\ba\2dots\big(\bH^{2}\cdot\bH^{2}\big)^s\2dots\ba^{2}\big)$ \\
        $j_{a44}$ & 8    & (3,5,0)                                    & $\ba^{2}\2dots\big(\ba^{2}\2dots\big(\bH^{2}\cdot\bH\big)^s\2dots\ba\big)$ \\
        $j_{a45}$ & 9    & (6,3,0)                                    & $\big(\cv{3}{3}\2dots\ba\big)\cdot\ba\cdot\big(\cv{3}{3}\2dots\ba\big)$    \\
        $j_{a46}$ & 9    & (7,2,0)                                    & $\big(\td{c}_{4}\td{c}_{3}\big)\2dots\ba^{2}$                              \\
        $j_{a47}$ & 9    & (7,2,0)                                    & $\big(\td{d}_{2}^{2}\td{c}_{3}\big)\2dots\ba^{2}$                          \\
        $j_{a48}$ & 9    & (8,1,0)                                    & $\big(\td{d}_{2}\td{c}_{3}^{2}\big)\2dots\ba$                              \\
        $j_{a49}$ & 9    & (8,1,0)                                    & $\td{c}_{4}^{2}\2dots\ba$                                                  \\
        $j_{a50}$ & 10   & (8,2,0)                                    & $\td{c}_{4}^{2}\2dots\ba^{2}$                                              \\
        $j_{a51}$ & 10   & (9,1,0)                                    & $\big(\td{d}_{2}^{2}\td{c}_{5}\big)\2dots\ba$                              \\
        $j_{a52}$ & 11   & (10,1,0)                                   & $\vv_{5}\cdot\ba\cdot\vv_{5}$                                              \\ \hline
      \end{tabular}}
  \end{minipage}
  \caption{Joint Invariants of $(\bH,\ba)$}
  \label{tab:inv-H4-H2}
\end{table}

\begin{table}[p]
  \begin{minipage}[B]{0.49\linewidth}
    \centering
    \resizebox{\textwidth}{!}{%
      \renewcommand{\arraystretch}{1.2}
      \begin{tabular}{|l|c|c|c|}
        \hline
        \#        & Deg. & M. deg. & Inv.                                                              \\ \hline \rule{0pt}{3ex}
        $J_{i1}$  & 3    & (1,1,1) & $\ba\2dots\bH\2dots\bb$                                           \\
        $J_{i2}$  & 4    & (1,1,2) & $\ba\2dots\bH\2dots\bb^{2}$                                       \\
        $J_{i3}$  & 4    & (1,1,2) & $\bb\2dots\bH\2dots(\ba\bb)$                                      \\
        $J_{i4}$  & 4    & (1,2,1) & $\bb\2dots\bH\2dots\ba^{2}$                                       \\
        $J_{i5}$  & 4    & (1,2,1) & $\ba\2dots\bH\2dots(\ba\bb)$                                      \\
        $J_{i6}$  & 4    & (2,1,1) & $\ba\2dots(\td{d}_{2}\bb)$                                        \\
        $J_{i7}$  & 4    & (2,1,1) & $\ba\2dots\big(\bH^{2}\big)^s\2dots\bb$                           \\
        $J_{i8}$  & 5    & (1,1,3) & $\bb^{2}\2dots\bH\2dots\ab$                                       \\
        $J_{i9}$  & 5    & (1,1,3) & $\bb\2dots\bH\2dots\big(\ba\bb^{2}\big)$                          \\
        $J_{i10}$ & 5    & (1,2,2) & $\ba^{2}\2dots\bH\2dots\bb^{2}$                                   \\
        $J_{i11}$ & 5    & (1,2,2) & $\bb\2dots\bH\2dots\big(\ba^{2}\bb\big)$                          \\
        $J_{i12}$ & 5    & (1,2,2) & $\ab\2dots\bH\2dots\ab$                                           \\
        $J_{i13}$ & 5    & (1,3,1) & $\ba^{2}\2dots\bH\2dots\ab$                                       \\
        $J_{i14}$ & 5    & (1,3,1) & $\ba\2dots\bH\2dots\big(\ba^{2}\bb\big)$                          \\
        $J_{i15}$ & 5    & (2,1,2) & $\ba\2dots\big(\bH^{2}\big)^s\2dots\bb^{2}$                       \\
        $J_{i16}$ & 5    & (2,1,2) & $\ba\2dots\big( \bb\2dots\big(\bH\cdot\bH\big)^s\2dots\bb\big)$   \\
        $J_{i17}$ & 5    & (2,2,1) & $\bb\2dots\big( \ba\2dots\big(\bH\cdot\bH\big)^s\2dots\ba\big)$   \\
        $J_{i18}$ & 5    & (2,1,2) & $\ba\2dots\big(\bb^{2}\td{d}_{2}\big)$                            \\
        $J_{i19}$ & 5    & (2,2,1) & $\ba^{2}\2dots\big(\bb\td{d}_{2}\big)$                            \\
        $J_{i20}$ & 5    & (2,2,1) & $\bb\2dots\big(\bH^{2}\big)^s\2dots\ba^{2}$                       \\
        $J_{i21}$ & 5    & (2,2,1) & $\ba\2dots\big(\bH^{2}\big)^s\2dots\ab$                           \\
        $J_{i22}$ & 5    & (3,1,1) & $\ba\2dots\big(\bH^{3}\big)^s\2dots\bb$                           \\
        $J_{i23}$ & 5    & (3,1,1) & $\ab\2dots\td{d}_{3}$                                             \\
        $J_{i24}$ & 5    & (2,1,2) & $\bb\2dots\big(\bH^{2}\big)^s\2dots\ab$                           \\
        $J_{i25}$ & 5    & (3,1,1) & $\tr\big(\bH\times\td{d}_{2}\big)\3dots \big( \ba\times \bb\big)$ \\
        $J_{i26}$ & 6    & (1,1,4) & $\bb^{2}\2dots\bH\2dots(\ba\bb^{2})$                              \\
        $J_{i27}$ & 6    & (1,2,3) & $\bb\2dots\bH\2dots(\ba^{2}\bb^{2})$                              \\
        $J_{i28}$ & 6    & (1,2,3) & $(\ba\bb)\2dots\bH\2dots(\ba\bb^{2})$                             \\
        $J_{i29}$ & 6    & (1,3,2) & $(\ba\bb)\2dots\bH\2dots(\ba^{2}\bb)$                             \\
        $J_{i30}$ & 6    & (1,3,2) & $\ba\2dots\bH\2dots(\ba^{2}\bb^{2})$                              \\
        $J_{i31}$ & 6    & (1,4,1) & $\ba^{2}\2dots\bH\2dots(\ba^{2}\bb$                               \\
        \hline
      \end{tabular}
    }
  \end{minipage}
  \centering
  \begin{minipage}[t]{0.50\linewidth}
    \resizebox{\textwidth}{!}{%
      \renewcommand{\arraystretch}{1.2}
      \begin{tabular}{|l|c|c|c|}
        \hline
        \#        & Deg. & M. deg. & Inv.                                                                              \\ \hline \rule{0pt}{3ex}
        $J_{i32}$ & 6    & (2,1,3) & $\bb\2dots\big(\bH^{2}\big)^s\2dots(\ba\bb^{2})$                                  \\
        $J_{i33}$ & 6    & (2,1,3) & $\bb^{2}\2dots\big(\bH^{2}\big)^s\2dots(\ba\bb)$                                  \\
        $J_{i34}$ & 6    & (2,2,2) & $\bb\2dots\big(\bH^{2}\big)^s\2dots(\ba^{2}\bb)$                                  \\
        $J_{i35}$ & 6    & (2,2,2) & $\ab\2dots\big(\bH^{2}\big)^{s}\2dots(\ba\bb)$                                    \\
        $J_{i36}$ & 6    & (2,2,2) & $\ba^{2}\2dots\big(\bH^{2}\big)^s\2dots\bb^{2}$                                   \\
        $J_{i37}$ & 6    & (2,3,1) & $\ba^{2}\2dots\big(\bH^{2}\big)^s\2dots(\ba\bb)$                                  \\
        $J_{i38}$ & 6    & (2,3,1) & $\ba\2dots\big(\bH^{2}\big)^s\2dots(\ba^{2}\bb)$                                  \\
        $J_{i39}$ & 6    & (2,1,3) & $\bb^{2}\2dots\big( \ba\2dots\big(\bH\cdot\bH\big)^s\2dots\bb\big)$               \\
        $J_{i40}$ & 6    & (2,1,3) & $\bb\2dots\big( \bb\2dots\big(\bH\cdot\bH\big)^s\2dots(\ba\bb)\big)$              \\
        $J_{i41}$ & 6    & (2,2,2) & $\bb\2dots\big( \bb\2dots\big(\bH\cdot\bH\big)^s\2dots\ba^{2}\big)$               \\
        $J_{i42}$ & 6    & (2,3,1) & $\ba\2dots\big( \ba\2dots\big(\bH\cdot\bH\big)^s\2dots(\ba\bb)\big)$              \\
        $J_{i43}$ & 6    & (2,3,1) & $\ba\2dots\big( \bb\2dots\big(\bH\cdot\bH\big)^s\2dots\ba^{2}\big)$               \\
        $J_{i44}$ & 6    & (2,2,2) & $\ba\2dots\big( \bb\2dots\big(\bH\cdot\bH\big)^s\2dots(\ba\bb)\big)$              \\
        $J_{i45}$ & 6    & (2,2,2) & $\ba\2dots\big( \ba\2dots\big(\bH\cdot\bH\big)^s\2dots\bb^{2}\big)$               \\
        $J_{i46}$ & 6    & (3,1,2) & $\tr\big(\bH\times\td{d}_{2}\big)\3dots \big( \ba\times \bb^{2}\big)$             \\
        $J_{i47}$ & 6    & (3,2,1) & $\tr\big(\bH\times\td{d}_{2}\big)\3dots \big( \ba\times \abS\big)$                \\
        $J_{i48}$ & 6    & (3,1,2) & $\tr\big(\bH\times\td{d}_{2}\big)\3dots \big( \bb\times \abS\big)$                \\
        $J_{i49}$ & 6    & (3,2,1) & $\tr\big(\bH\times\td{d}_{2}\big)\3dots \big( \ba^{2}\times \bb\big)$             \\
        $J_{i50}$ & 6    & (3,1,2) & $\ba\2dots\big(\bH^{3}\big)^s\2dots\bb^{2}$                                       \\
        $J_{i51}$ & 6    & (3,1,2) & $\bb\2dots\big(\bH^{3}\big)^s\2dots\ab$                                           \\
        $J_{i52}$ & 6    & (3,2,1) & $\bb\2dots\big(\bH^{3}\big)^s\2dots\ba^{2}$                                       \\
        $J_{i53}$ & 6    & (3,2,1) & $\ba\2dots\big(\bH^{3}\big)^s\2dots\ab$                                           \\
        $J_{i54}$ & 6    & (3,1,2) & $\bb\2dots\big(\big(\bH\times\td{d}_{2}\big)\3dots \big( \ba\times \bb\big)\big)$ \\
        $J_{i55}$ & 6    & (3,2,1) & $\ba\2dots\big(\big(\bH\times\td{d}_{2}\big)\3dots \big( \ba\times \bb\big)\big)$ \\
        $J_{i56}$ & 6    & (3,1,2) & $\ba\2dots\big( \bb\2dots\big(\bH^{2}\cdot\bH\big)^s\2dots\bb\big)$               \\
        $J_{i57}$ & 6    & (3,2,1) & $\bb\2dots\big( \ba\2dots\big(\bH^{2}\cdot\bH\big)^s\2dots\ba\big)$               \\
        $J_{i58}$ & 6    & (4,1,1) & $\td{d}_{2}^{2}\2dots\ab$                                                         \\
        $J_{i59}$ & 6    & (4,1,1) & $\td{c}_{4}\2dots\ab$                                                             \\
        $J_{i60}$ & 6    & (4,1,1) & $\tr\big(\bH\times\td{c}_{3}\big)\3dots \big( \ba\times \bb\big)$                 \\
        $J_{i61}$ & 6    & (4,1,1) & $\ba\2dots\big(\bH^{4}\big)^s\2dots\bb$                                           \\
        \hline
      \end{tabular}
    }
  \end{minipage}
  \caption{Joint Invariants of $(\bH,\ba,\bb)$ (degree $\leq 6$)}
  \label{tab:inv-H4-H2-H2-a}
\end{table}

\begin{table}[p]
  \begin{minipage}[B]{0.49\linewidth}
    \centering
    \resizebox{\textwidth}{!}{%
      \renewcommand{\arraystretch}{1.2}
      \begin{tabular}{|l|c|c|c|}
        \hline
        \#        & Deg. & M. deg. & Inv.                                                                             \\ \hline
        $J_{i62}$ & 7    & (2,1,4) & $\bb\2dots\big( \bb\2dots\big(\bH\cdot\bH\big)^s\2dots\big(\ba\bb^{2}\big)\big)$ \\
        $J_{i63}$ & 7    & (2,1,4) & $\ba\2dots\big( \bb^{2}\2dots\big(\bH\cdot\bH\big)^s\2dots\bb^{2}\big)$          \\
        $J_{i64}$ & 7    & (2,1,4) & $\bb\2dots\big( \bb^{2}\2dots\big(\bH\cdot\bH\big)^s\2dots\ab\big)$              \\
        $J_{i65}$ & 7    & (2,2,3) & $\bb\2dots\big( \bb\2dots\big(\bH\cdot\bH\big)^s\2dots\big(\ba^{2}\bb\big)\big)$ \\
        $J_{i66}$ & 7    & (2,2,3) & $\ba\2dots\big( \bb^{2}\2dots\big(\bH\cdot\bH\big)^s\2dots\ab\big)$              \\
        $J_{i67}$ & 7    & (2,3,2) & $\bb\2dots\big( \ba^{2}\2dots\big(\bH\cdot\bH\big)^s\2dots\ab\big)$              \\
        $J_{i68}$ & 7    & (2,3,2) & $\ba\2dots\big( \bb\2dots\big(\bH\cdot\bH\big)^s\2dots\big(\ba^{2}\bb\big)\big)$ \\
        $J_{i69}$ & 7    & (2,3,2) & $\ba\2dots\big( \ba^{2}\2dots\big(\bH\cdot\bH\big)^s\2dots\bb^{2}\big)$          \\
        $J_{i70}$ & 7    & (2,4,1) & $\ba\2dots\big( \ba^{2}\2dots\big(\bH\cdot\bH\big)^s\2dots\ab\big)$              \\
        $J_{i71}$ & 7    & (2,4,1) & $\ba\2dots\big( \ba\2dots\big(\bH\cdot\bH\big)^s\2dots\big(\ba^{2}\bb\big)\big)$ \\
        $J_{i72}$ & 7    & (2,4,1) & $\bb\2dots\big( \ba^{2}\2dots\big(\bH\cdot\bH\big)^s\2dots\ba^{2}\big)$          \\
        $J_{i73}$ & 7    & (2,3,2) & $\ba\2dots\big( \ab\2dots\big(\bH\cdot\bH\big)^s\2dots\ab\big)$                  \\
        $J_{i74}$ & 7    & (2,2,3) & $\bb\2dots\big( \ab\2dots\big(\bH\cdot\bH\big)^s\2dots\ab\big)$                  \\
        $J_{i75}$ & 7    & (2,2,3) & $\bb\2dots\big( \ba^{2}\2dots\big(\bH\cdot\bH\big)^s\2dots\bb^{2}\big)$          \\
        $J_{i76}$ & 7    & (3,1,3) & $\bb\2dots\big(\cv{3}{5}\3dots \big( \bb\times \abS\big)\big)$                   \\
        $J_{i77}$ & 7    & (3,2,2) & $\ba\2dots\big(\cv{3}{5}\3dots \big( \bb\times \abS\big)\big)$                   \\
        $J_{i78}$ & 7    & (3,2,2) & $\ba\2dots\big(\cv{3}{5}\3dots \big( \ba\times \bb^{2}\big)\big)$                \\
        $J_{i79}$ & 7    & (3,2,2) & $\bb\2dots\big(\cv{3}{5}\3dots \big( \ba^{2}\times \bb\big)\big)$                \\
        $J_{i80}$ & 7    & (3,3,1) & $\ba\2dots\big(\cv{3}{5}\3dots \big( \ba^{2}\times \bb\big)\big)$                \\
        $J_{i81}$ & 7    & (3,3,1) & $\ba\2dots\big(\cv{3}{5}\3dots \big( \ba\times \abS\big)\big)$                   \\
        $J_{i82}$ & 7    & (3,3,1) & $\bb\2dots\big(\cv{3}{5}\3dots \big( \ba^{2}\times \ba\big)\big)$                \\
        $J_{i83}$ & 7    & (3,1,3) & $\bb\2dots\big(\cv{3}{5}\3dots \big( \ba\times \bb^{2}\big)\big)$                \\
        $J_{i84}$ & 7    & (3,1,3) & $\ba\2dots\big(\cv{3}{5}\3dots \big( \bb^{2}\times \bb\big)\big)$                \\
        $J_{i85}$ & 7    & (3,2,2) & $\ab\2dots\big(\bH^{3}\big)^s\2dots\ab$                                          \\
        $J_{i86}$ & 7    & (3,1,3) & $\bb\2dots\big( \bb\2dots\big(\bH^{2}\cdot\bH\big)^s\2dots\ab\big)$              \\
        $J_{i87}$ & 7    & (3,1,3) & $\ba\2dots\big( \bb\2dots\big(\bH^{2}\cdot\bH\big)^s\2dots\bb^{2}\big)$          \\
        $J_{i88}$ & 7    & (3,2,2) & $\ba\2dots\big( \ba\2dots\big(\bH^{2}\cdot\bH\big)^s\2dots\bb^{2}\big)$          \\
        $J_{i89}$ & 7    & (3,2,2) & $\ba\2dots\big( \bb\2dots\big(\bH^{2}\cdot\bH\big)^s\2dots\ab\big)$              \\
        \hline
      \end{tabular}
    }
  \end{minipage}
  \centering
  \begin{minipage}[t]{0.50\linewidth}
    \resizebox{\textwidth}{!}{%
      \renewcommand{\arraystretch}{1.2}
      \begin{tabular}{|l|c|c|c|}
        \hline
        \#         & Deg. & M. deg. & Inv.                                                                              \\ \hline
        $J_{i90}$  & 7    & (3,2,2) & $\bb\2dots\big( \bb\2dots\big(\bH^{2}\cdot\bH\big)^s\2dots\ba^{2}\big)$           \\
        $J_{i91}$  & 7    & (3,3,1) & $\ba\2dots\big( \bb\2dots\big(\bH^{2}\cdot\bH\big)^s\2dots\ba^{2}\big)$           \\
        $J_{i92}$  & 7    & (3,3,1) & $\ba\2dots\big( \ba\2dots\big(\bH^{2}\cdot\bH\big)^s\2dots\ab\big)$               \\
        $J_{i93}$  & 7    & (3,1,3) & $\bb\2dots\big( \cv{3}{7}\3dots\big(\ba\times \bb\big)\big)\2dots\bb$             \\
        $J_{i94}$  & 7    & (3,2,2) & $\ba\2dots\big( \cv{3}{7}\3dots\big(\ba\times \bb\big)\big)\2dots\bb$             \\
        $J_{i95}$  & 7    & (3,3,1) & $\ba\2dots\big( \cv{3}{7}\3dots\big(\ba\times \bb\big)\big)\2dots\ba$             \\
        $J_{i96}$  & 7    & (4,1,2) & $\tr\big(\bH\times\td{c}_{3}\big)\3dots \big( \ba\times \bb^{2}\big)$             \\
        $J_{i97}$  & 7    & (4,1,2) & $\tr\big(\bH\times\td{c}_{3}\big)\3dots \big( \bb\times \abS\big)$                \\
        $J_{i98}$  & 7    & (4,2,1) & $\tr\big(\bH\times\td{c}_{3}\big)\3dots \big( \ba\times \abS\big)$                \\
        $J_{i99}$  & 7    & (4,2,1) & $\tr\big(\bH\times\td{c}_{3}\big)\3dots \big( \ba^{2}\times \bb\big)$             \\
        $J_{i100}$ & 7    & (4,1,2) & $\bb\2dots\big(\bH^{4}\big)^s\2dots\ab$                                           \\
        $J_{i101}$ & 7    & (4,2,1) & $\ba\2dots\big(\bH^{4}\big)^s\2dots\ab$                                           \\
        $J_{i102}$ & 7    & (4,2,1) & $\bb\2dots\big(\bH^{4}\big)^s\2dots\ba^{2}$                                       \\
        $J_{i103}$ & 7    & (4,1,2) & $\ba\2dots\big(\bH^{4}\big)^s\2dots\bb^{2}$                                       \\
        $J_{i104}$ & 7    & (4,1,2) & $\bb\2dots\big(\big(\bH\times \td{c}_{3}\big)\3dots \big(\ba\times \bb\big)\big)$ \\
        $J_{i105}$ & 7    & (4,2,1) & $\ba\2dots\big(\big(\bH\times \td{c}_{3}\big)\3dots \big(\ba\times \bb\big)\big)$ \\
        $J_{i106}$ & 7    & (4,1,2) & $\bb\2dots\big(\cv{4b}{5}\3dots \big(\ba\times \bb\big)\big)$                     \\
        $J_{i107}$ & 7    & (4,2,1) & $\ba\2dots\big(\cv{4b}{5}\3dots \big(\ba\times \bb\big)\big)$                     \\
        $J_{i108}$ & 7    & (4,1,2) & $\ba\2dots\big( \bb\2dots\big(\bH^{2}\cdot\bH^{2}\big)^s\2dots\bb\big)$           \\
        $J_{i109}$ & 7    & (4,2,1) & $\ba\2dots\big( \ba\2dots\big(\bH^{2}\cdot\bH^{2}\big)^s\2dots\bb\big)$           \\
        $J_{i110}$ & 7    & (5,1,1) & $\td{c}_{5}\2dots\ab$                                                             \\
        $J_{i111}$ & 7    & (5,1,1) & $\big(\td{d}_{2}\td{c}_{3}\big)^s\2dots\ab$                                       \\
        $J_{i112}$ & 7    & (5,1,1) & $\big(\td{d}_{2}\times \td{c}_{3}\big)\3dots \big( \ba\times \bb\big)$            \\
        $J_{i113}$ & 7    & (5,1,1) & $\tr\big(\bH\times\td{d}_{2}^{2}\big)\3dots \big( \ba\times \bb\big)$             \\
        $J_{i114}$ & 7    & (5,1,1) & $\ba\2dots\big(\bH\cdot\td{d}_{2}^{2}\big)^s\2dots\bb$                            \\  \hline
      \end{tabular}
    }
  \end{minipage}
  \caption{Joint Invariants of $(\bH,\ba,\bb)$ (degree $7$)}
  \label{tab:inv-H4-H2-H2-b}
\end{table}

\begin{table}[p]
  \begin{minipage}[b]{0.49\linewidth}
    \centering
    \resizebox{\textwidth}{!}{%
      \renewcommand{\arraystretch}{1.2}
      \begin{tabular}{|l|c|c|c|}
        \hline
        \#         & Deg. & M. deg. & Inv.                                                                                     \\ \hline
        $J_{i115}$ & 8    & (3,1,4) & $\bb\2dots\big(\bb^{2}\2dots\big(\bH^{2}\cdot\bH\big)^s\2dots\ab\big)$                   \\
        $J_{i116}$ & 8    & (3,2,3) & $\bb\2dots\big(\ab\2dots\big(\bH^{2}\cdot\bH\big)^s\2dots\ab\big)$                       \\
        $J_{i117}$ & 8    & (3,4,1) & $\bb\2dots\big(\ba^{2}\2dots\big(\bH^{2}\cdot\bH\big)^s\2dots\ba^{2}\big)$               \\
        $J_{i118}$ & 8    & (3,3,2) & $\bb\2dots\big(\ba^{2}\2dots\big(\bH^{2}\cdot\bH\big)^s\2dots\ab\big)$                   \\
        $J_{i119}$ & 8    & (3,1,4) & $\bb\2dots\big( \cv{3}{7}\3dots \big( \bb\times \abS\big)\big)\2dots\bb$                 \\
        $J_{i120}$ & 8    & (3,2,3) & $\ba\2dots\big( \cv{3}{7}\3dots \big( \bb\times \abS\big)\big)\2dots\bb$                 \\
        $J_{i121}$ & 8    & (3,3,2) & $\ba\2dots\big( \cv{3}{7}\3dots \big( \ba\times \abS\big)\big)\2dots\bb$                 \\
        $J_{i122}$ & 8    & (3,1,4) & $\bb\2dots\big(\big(\cv{3}{9}\3dots \big( \ba\times\bb\big)\big)\2dots\bb\big)\2dots\bb$ \\
        $J_{i123}$ & 8    & (3,4,1) & $\ba\2dots\big(\big(\cv{3}{9}\3dots \big( \ba\times\bb\big)\big)\2dots\ba\big)\2dots\ba$ \\
        $J_{i124}$ & 8    & (3,3,2) & $\ba\2dots\big(\big(\cv{3}{9}\3dots \big( \ba\times\bb\big)\big)\2dots\ba\big)\2dots\bb$ \\
        $J_{i125}$ & 8    & (3,2,3) & $\ba\2dots\big(\big(\cv{3}{9}\3dots \big( \ba\times\bb\big)\big)\2dots\bb\big)\2dots\bb$ \\
        $J_{i126}$ & 8    & (3,4,1) & $\ba\2dots\big( \cv{3}{7}\3dots \big( \ba\times \abS\big)\big)\2dots\ba$                 \\
        $J_{i127}$ & 8    & (4,3,1) & $\ba\2dots\big(\cv{4b}{5}\3dots \big( \ba\times \abS\big)\big)$                          \\
        $J_{i128}$ & 8    & (4,3,1) & $\ba\2dots\big(\cv{4b}{5}\3dots \big( \ba^{2}\times \bb\big)\big)$                       \\
        $J_{i129}$ & 8    & (4,2,2) & $\bb\2dots\big(\cv{4b}{5}\3dots \big( \ba^{2}\times \bb\big)\big)$                       \\
        $J_{i130}$ & 8    & (4,1,3) & $\ba\2dots\big(\bb\2dots\big(\bH^{2}\cdot\bH^{2}\big)^s\2dots\bb^{2}\big)$               \\
        $J_{i131}$ & 8    & (4,3,1) & $\ba\2dots\big(\bb\2dots\big(\bH^{2}\cdot\bH^{2}\big)^s\2dots\ba^{2}\big)$               \\
        $J_{i132}$ & 8    & (4,3,1) & $\ba\2dots\big(\ba\2dots\big(\bH^{2}\cdot\bH^{2}\big)^s\2dots\ab\big)$                   \\
        $J_{i133}$ & 8    & (4,1,3) & $\bb\2dots\big(\bb\2dots\big(\bH^{2}\cdot\bH^{2}\big)^s\2dots\ab\big)$                   \\
        $J_{i134}$ & 8    & (4,2,2) & $\ba\2dots\big(\bb\2dots\big(\bH^{2}\cdot\bH^{2}\big)^s\2dots\ab\big)$                   \\
        $J_{i135}$ & 8    & (4,2,2) & $\bb\2dots\big(\bb\2dots\big(\bH^{2}\cdot\bH^{2}\big)^s\2dots\ba^{2}\big)$               \\
        $J_{i136}$ & 8    & (4,2,2) & $\ba\2dots\big(\ba\2dots\big(\bH^{2}\cdot\bH^{2}\big)^s\2dots\bb^{2}\big)$               \\
        $J_{i137}$ & 8    & (4,1,3) & $\bb\2dots\big(\cv{4}{7}\3dots \big( \ba\times \bb\big)\big)\2dots\bb$                   \\
        $J_{i138}$ & 8    & (4,3,1) & $\ba\2dots\big(\cv{4}{7}\3dots \big( \ba\times \bb\big)\big)\2dots\ba$                   \\
        $J_{i139}$ & 8    & (4,2,2) & $\bb\2dots\big(\cv{4}{7}\3dots \big( \ba\times \bb\big)\big)\2dots\ba$                   \\
        $J_{i140}$ & 8    & (4,3,1) & $\bb\2dots\big(\cv{4b}{5}\3dots \big( \bb\times \abS\big)\big)$                          \\
        $J_{i141}$ & 8    & (4,2,2) & $\ba\2dots\big(\cv{4b}{5}\3dots \big( \bb\times \abS\big)\big)$                          \\
        $J_{i142}$ & 8    & (4,1,3) & $\bb\2dots\big(\cv{4b}{5}\3dots \big( \ba\times \bb^{2}\big)\big)$                       \\
        $J_{i143}$ & 8    & (5,2,1) & $\ba\2dots\big(\bH\cdot\td{d}_{2}^{2}\big)^s\2dots\ab$                                   \\
        $J_{i144}$ & 8    & (5,2,1) & $\bb\2dots\big(\bH\cdot\td{d}_{2}^{2}\big)^s\2dots\ba^{2}$                               \\
        $J_{i145}$ & 8    & (5,1,2) & $\bb\2dots\big(\bH\cdot\td{d}_{2}^{2}\big)^s\2dots\ab$                                   \\
        $J_{i146}$ & 8    & (5,1,2) & $\ba\2dots\big(\bH\cdot\td{d}_{2}^{2}\big)^s\2dots\bb^{2}$                               \\
        \hline
      \end{tabular}
    }
  \end{minipage}
  \centering
  \begin{minipage}{0.50\linewidth}
    \resizebox{\textwidth}{!}{%
      \renewcommand{\arraystretch}{1.2}
      \begin{tabular}{|l|c|c|c|}
        \hline
        \#         & Deg. & M. deg. & Inv.                                                                                                  \\ \hline
        $J_{i147}$ & 8    & (5,1,2) & $\bb\2dots \big( \big( \bH\times \td{d}_{2}^{2}\big)\3dots \big( \ba\times \bb\big)\big)$             \\
        $J_{i148}$ & 8    & (5,2,1) & $\ba\2dots \big( \big( \bH\times \td{c}_{4}\big)\3dots \big( \ba\times \bb\big)\big)$                 \\
        $J_{i149}$ & 8    & (5,1,2) & $\bb\2dots \big( \big( \bH\times \td{c}_{4}\big)\3dots \big( \ba\times \bb\big)\big)$                 \\
        $J_{i150}$ & 8    & (5,2,1) & $\ba\2dots \big( \big( \big(\bH^{2}\big)^s\times \td{c}_{3}\big)\3dots \big( \ba\times \bb\big)\big)$ \\
        $J_{i151}$ & 8    & (5,1,2) & $\bb\2dots \big( \big( \big(\bH^{2}\big)^s\times \td{c}_{3}\big)\3dots \big( \ba\times \bb\big)\big)$ \\
        $J_{i152}$ & 8    & (5,2,1) & $\ba\2dots \big( \big( \bH\times \td{d}_{2}^{2}\big)\3dots \big( \ba\times \bb\big)\big)$             \\
        $J_{i153}$ & 8    & (5,2,1) & $\tr\big(\bH\times \td{d}_{2}^{2}\big)\3dots \big( \ba^{2}\times \bb\big)$                            \\
        $J_{i154}$ & 8    & (5,1,2) & $\tr\big(\bH\times \td{d}_{2}^{2}\big)\3dots \big( \ba\times \bb^{2}\big)$                            \\
        $J_{i155}$ & 8    & (6,1,1) & $\big(\td{d}_{2}\td{c}_{4}\big)^s\2dots\ab$                                                           \\
        $J_{i156}$ & 8    & (6,1,1) & $\td{c}_{3}^{2}\2dots\ab$                                                                             \\
        $J_{i157}$ & 8    & (6,1,1) & $\tr\big(\bH\times \td{c}_{5}\big)\3dots \big( \ba\times \bb\big)$                                    \\
        $J_{i156}$ & 8    & (6,1,1) & $\td{c}_{3}^{2}\2dots\ab$                                                                             \\
        $J_{i157}$ & 8    & (6,1,1) & $\tr\big(\bH\times \td{c}_{5}\big)\3dots \big( \ba\times \bb\big)$                                    \\
        $J_{i158}$ & 8    & (6,1,1) & $\big(\td{d}_{2}\times \td{c}_{4}\big)\3dots \big( \ba\times \bb\big)$                                \\
        $J_{i159}$ & 8    & (6,1,1) & $\ba\2dots\big(\bH^{2}\cdot\td{d}_{2}^{2}\big)^s\2dots\bb$                                            \\ \hline
        $J_{i160}$ & 9    & (6,2,1) & $\vv_{5}\cdot\ba\cdot\big(\bH\3dots \big(\ba\times \bb\big)\big)$                                     \\
        $J_{i161}$ & 9    & (6,1,2) & $\vv_{5}\cdot\bb\cdot\big(\bH\3dots \big(\ba\times \bb\big)\big)$                                     \\
        $J_{i162}$ & 9    & (6,1,2) & $\big(\cv{3}{3}\2dots\bb\big)\cdot\ba\cdot\big(\big(\cv{3}{3}\big)\2dots\bb\big)$                     \\
        $J_{i163}$ & 9    & (6,2,1) & $\big(\cv{3}{3}\2dots\ba\big)\cdot\bb\cdot\big(\big(\cv{3}{3}\big)\2dots\ba\big)$                     \\
        $J_{i164}$ & 9    & (6,1,2) & $\bb\2dots\big( \big(\bH\times \td{c}_{5}\big)\3dots\big(\ba\times \bb\big)\big)$                     \\
        $J_{i165}$ & 9    & (6,2,1) & $\ba\2dots\big( \big(\bH\times \td{c}_{5}\big)\3dots\big(\ba\times \bb\big)\big)$                     \\
        $J_{i166}$ & 9    & (7,1,1) & $\td{d}_{2}\2dots\big(\big(\ba\times \bb\big)\cdot\vv_{5}\big)$                                       \\
        $J_{i167}$ & 9    & (7,1,1) & $\big(\td{c}_{3}\times \td{c}_{4}\big)\3dots \big( \ba\times \bb\big)$                                \\
        $J_{i168}$ & 9    & (7,1,1) & $\big(\td{c}_{3}\td{c}_{4}\big)^s\2dots\ab$                                                           \\
        $J_{i169}$ & 9    & (7,1,1) & $\big(\td{d}_{2}^{2}\td{c}_{3}\big)^s\2dots\ab$                                                       \\ \hline
        $J_{i170}$ & 10   & (8,1,1) & $\vv_{8a}\cdot\big( \pmb\varepsilon\2dots\big(\ba\bb\big)\big)$                                       \\
        $J_{i171}$ & 10   & (8,1,1) & $\td{c}_{4}^{2}\2dots\ab$                                                                             \\ \hline
      \end{tabular}
    }
  \end{minipage}
  \caption{Joint Invariants of $(\bH,\ba,\bb)$ (degree $8$ to $10$)}
  \label{tab:inv-H4-H2-H2-c}
\end{table}


\end{document}